\documentclass[11pt,a4paper,aps]{article}
\usepackage{etex}
\usepackage{inputenc}
\usepackage[english]{babel}
\usepackage{lmodern}					% harmonie des polices
\usepackage{textcomp}				% table de symboles
\usepackage{verbatim}			% pour faire des commentaires par block avec begin{comment}...end{comment}
\usepackage[toc,page]{appendix} 		% pour faire des annexes
\usepackage{cite}
\usepackage{amsmath,					% packages mathmatiques
			amsthm,
			amsrefs,
			amssymb, 
			mathtools,xypic}%,a4wide}
\usepackage{physics}
\usepackage{stmaryrd}
\usepackage{slashed}
\usepackage{enumitem} 

 \usepackage{tikz-cd} 
\usetikzlibrary{tikzmark}
\usetikzlibrary{arrows}
 \usepackage{verbatim}
 \usepackage[mathscr]{euscript}
 \usepackage{tikz-cd} 
\usepackage{tabularx}				% pour inclure des tableaux (careful)

\usepackage{authblk}

\usepackage[hmargin=2.2cm,vmargin=2.5cm]{geometry}		% adapter les dimensions de la page
\numberwithin{equation}{section}
\usepackage{musicography}

\usepackage[colorlinks]{hyperref}			% liens dans le .pdf

\newcommand{\bigslant}[2]{{\left.\raisebox{.2em}{$#1$}\middle/\raisebox{-.2em}{$#2$}\right.}}

\usetikzlibrary{calc,bending}

\newtheorem{definition}{Definition}[section]
\newtheorem{theoreme}[definition]{Theorem}
\newtheorem{proposition}[definition]{Proposition}
\newtheorem{lemme}[definition]{Lemma}

\newtheorem{corollaire}[definition]{Corollary}
\theoremstyle{remark}

\newtheorem*{remarque}{Remark}

\newtheorem{example}{Example}

\title{The modular class of a singular foliation}

\author{Sylvain Lavau} %\thanks{Correspondence: lavau@math.univ-lyon1.fr.}} %Present address: \emph{Aristotle University of Thessaloniki, Department of Mathematics}, 541 24 Thessaloniki, Greece.}}

\affil{{\small \emph{Euler Institute (EIMI) \& Steklov Institute}, %\\
 St-Petersburg, Russian Federation.}\\
 \small \emph{Aristotle University of Thessaloniki}, Thessaloniki, Greece.}

\date{}

\begin{document}
\maketitle

\begin{abstract}
The modular class of a regular foliation is a cohomological obstruction to the existence of a volume form transverse to the leaves which is invariant under the flow of the vector fields of the foliation. It has been used in dynamical systems and operator theory, from which the name \emph{modular} comes from. 
By drawing on the relationship between Lie algebroids and regular foliations, this paper extends the notion of modular class to the realm of singular foliations. The singularities are dealt with by replacing the singular foliations by any of their universal Lie $\infty$-algebroids, and by picking up the modular class of the latter. The geometric meaning of the modular class of a singular foliation $\mathcal{F}$ is not as transparent as for regular foliations: it is an obstruction to the existence of a universal Lie $\infty$-algebroid  of $\mathcal{F}$ whose  Berezinian line bundle is a trivial $\mathcal{F}$-module. This paper  illustrates the relevance of using universal Lie $\infty$-algebroids to extend mathematical notions from regular foliations to singular ones, thus paving the road to defining other characteristic classes in this way.

\end{abstract}

%\noindent\textbf{Keywords:} singular foliation, Lie $\infty$-algebroid, modular class. 

%\noindent\textbf{MSC classification:} 17B70, 53C12. 

%\noindent\textbf{\emph{J. Geom. Diff.} classification:} real and complex differential geometry, supermanifolds and supergroups.

\tableofcontents

\section{Introduction}\label{introduction}

Singular foliations form a class of objects that appear in geometry, operator algebra, dynamical systems, etc. that generalize the notion of regular foliation by allowing  the leaves to have non-constant dimension. This property is manifest at the infinitesimal level: a regular foliation is primarily defined as a vector subbundle of the tangent bundle, %-- and hence is entirely defined from its sheaf of sections,  
whereas a singular foliation can be considered as a particular subsheaf of the sheaf of vector fields, so that the rank of the induced distribution  may not be constant.
%The leaves of the foliation are the (weakly embedded) submanifolds whose tangent bundle coincides with the generalized distribution induced by the foliation. Hence, a foliation defines an equivalence relation on any manifold
Due to the difficulty to properly  handle the singularities, many notions and theories that are otherwise well established for regular foliations have not been generalized to the singular case yet.
In this paper we intend to discuss the notion of modular class associated to a singular foliation, by drawing on previous work \cite{laurent-gengouxUniversalLieInfinityAlgebroid2020} that allows to `regularize' the foliation.

A regular foliation induces an equivalence relation on the manifold: that of belonging to the same leaf. The quotient of the manifold by the equivalence relation -- the \emph{leaf space} --  is not necessarily a manifold, hence the difficulty of making sense of a measure -- or volume form -- on it. An alternative way of characterizing such a measure is to define a measure on the transversal of the leaves (if they exist), that would be invariant with respect to the flow of the vector fields defining the (regular) foliation. Under appropriate assumptions on the leaves, such a \emph{transverse measure} would then be considered as a measure on the leaf space. Transverse measures of regular foliations have proved to be very important in the study of dynamical systems on differentiable manifolds \cite{planteFoliationsMeasurePreserving1975, ruelleCurrentsFlowsDiffeomorphisms1975, hurderGodbillonMeasureAmenable1986}.
 The existence of a globally defined transverse measure is conditioned to the vanishing of a certain cohomology class of the foliated de Rham cohomology: the \emph{modular class} (or \emph{Reeb class}).
 
The name \emph{modular} originally refers to the Tomita-Takesaki theory in non-commutative geometry, which classifies von Neumann algebras from their so-called \emph{modular automorphisms}. This theory dating back from the end of the 1960 has been widely fueled  by examples taken from regular foliations, since any regular foliation equipped with a transverse measure canonically induces a von Neumann algebra \cite{connesNeumannAlgebraFoliation1978, yamagamiModularCohomologyClass1986}. 
The notion of modular automorphisms has been transported to the field of Poisson geometry \cite{weinsteinModularAutomorphismGroup1997}, since the notion of modular flow in Poisson geometry forms a classical limit of the notion of modular automorphism in the theory of von Neumann algebras. Due to the proximity between Poisson manifolds and Lie algebroids, the notion of modularity was extended to Lie algebroids and Lie groupoids \cite{evensTransverseMeasuresModular1999, kosmann-schwarzbachPoissonManifoldsLie2008}, generalizing the already existing and well-known notion of modular function of Lie groups which evaluates the existence of a bi-invariant Haar measure.
The modular class of Lie algebroids also restricts to the well-known notion of modular class of regular foliations when the Lie algebroid in question is a foliation Lie algebroid.  This allows to think about the modular class of a Lie algebroid as an obstruction to the existence of a certain invariant measure on the differentiable stack associated to this Lie algebroid \cite{weinsteinVolumeDifferentiableStack2009, crainicMeasuresDifferentiableStacks2020}. 

Given this background, it seems legitimate to generalize the notion of modular class to the singular setting, e.g. to singular foliations. The definition of the modular class of a regular foliation cannot straightforwardly apply to the singular case because the rank of the induced distribution of tangent vectors is not constant. Hence, the construction using Bott connections and transversals to the leaves cannot be imported as such. Since the modular class of a regular foliation coincides with the modular class of its associated foliation Lie algebroid, one way to overcome this difficulty is to find an analogue of the foliation Lie algebroid in the singular case, namely: a \emph{universal Lie $\infty$-algebroid} associated to the singular foliation \cite{laurent-gengouxUniversalLieInfinityAlgebroid2020}. A Lie $\infty$-algebroid is to a Lie algebroid what a (non-positively graded) $L_\infty$-algebra is to a Lie algebra.   
A universal Lie $\infty$-algebroid of a singular foliation $\mathcal{F}$ is a projective resolution of $\mathcal{F}$ by locally free $\mathcal{C}^\infty(M)$-modules which, by a transfer theorem, can be equipped with a Lie $\infty$-algebroid structure lifting the Lie bracket of vector fields. The name \emph{universal} is justified by the observation that two such universal Lie-$\infty$ algebroids resolving the same singular foliation are homotopically equivalent~\cite{laurent-gengouxUniversalLieInfinityAlgebroid2020}.

In a recent work, Caseiro and Laurent-Gengoux have extended the notion of modular class from Lie algebroids to Lie $\infty$-algebroids: it is an obstruction to the existence of an invariant section of the associated Berezinian line bundle \cite{caseiroModularClassLie2022}. 
They show that the modular class neither depends on the choice of the resolution, nor on the choice of the section of its associated Berezinian, etc. This allows to obtain a well-defined and unique notion of modular class of a singular foliation by setting it to be the modular class of any of its universal Lie $\infty$-algebroids. More precisely, it is the unique cohomology class in the foliated cohomology of $\mathcal{F}$ whose representant in the cohomology of any universal Lie $\infty$-algebroid $E$ of $\mathcal{F}$ is precisely the modular class of $E$ (see Definition \ref{def:modclasssing}). In this context, the modular class of a singular foliation $\mathcal{F}$ is an obstruction the the existence of a trivial $\mathcal{F}$-module structure on the Berezinian of any of its universal Lie $\infty$-algebroids (see Proposition \ref{plop}). In full rigor, we only define modular classes of \emph{solvable} singular foliations, i.e. those foliations which admit at least one universal Lie $\infty$-algebroid of finite length. This is a necessary condition for the Berezinian to be well-defined.

This notion of modular class for singular foliations coincides with that of regular foliations when $\mathcal{F}$ is regular and the Lie $\infty$-algebroid $E$ is the foliation Lie algebroid or, more generally, is a resolution of the latter. However, contrary to the regular case, singular foliations preserve their specificity: what is central in singular foliations is not the set of  leaves but the set of vector fields generating the singular foliation. There are indeed examples of singular foliations whose regular leaves admit a transverse volume form but which do not necessarily admit invariant sections of the Berezinian of \emph{any} of its universal Lie $\infty$-algebroids.  This curiosity can be explained by the behavior  at the singularities of the vector fields generating the singular foliation.
 This paper thus illustrates the relevance of using universal Lie $\infty$-algebroids to extend mathematical notions from regular foliations to singular ones \cite{laurent-gengouxUniversalLieInfinityAlgebroid2020}, thus paving the road to defining other characteristic classes in this way.
 
Part of the material of this paper is already known  but is scattered in the literature. It is here presented in a unified fashion, with natural and adapted conventions susceptible to be reused for further treatment of characteristic classes of singular foliations. 
Section \ref{modularclassLieinf} is a reminder on generalities on Lie $\infty$-algebroids and their representations up to homotopy. This material is already known but has been included here for self-consistency of the paper, and because it has not been presented under such conventions. Subsection \ref{singlg} recalls some basics about singular foliations and elaborates about the relationship between their foliated cohomology and the cohomology of their (universal) Lie $\infty$-algebroids.  %in order to introduce notations that will be used throughout the paper. 
Section \ref{sec3} is dedicated to presentation of the notion of modular class of Lie $\infty$-algebroids and singular foliations. %between modular classes of regular foliations and that of singular foliations. 
Subsection \ref{sec:char} does not contain new important results but adapts the material of \cite{caseiroModularClassLie2022} to the skew-symmetric convention for Lie $\infty$-algebroids. %While subsection \ref{sec:boot} is a reminder about the notion of invariant transverse volume form. 
Drawing on the material introduced in subsections \ref{singlg} and \ref{sec:char}, subsection~\ref{subsecc3} defines the modular class of singular foliations (Definition~\ref{def:modclasssing}), its meaning (Proposition~\ref{plop}), and its main relationship with that of regular foliations (Proposition \ref{prop:ok}).  Section~\ref{exxxon} is entirely dedicated to examples of various kinds. In particular, several cases are investigated in order to illustrate the extent to which the notion of modular class of singular foliations is meaningful.

%Recent advances in the understanding of singular foliations have shown that the singular character of these objects may be `circumvented' by using graded vector bundles that form a geometric resolution of the foliation.

 %finding a geometric resolution

\section{Lie $\infty$-algebroids and singular foliations}\label{modularclassLieinf}

Throughout the paper,  we will stay in the category of real smooth manifolds and will denote by $\mathcal{C}^\infty$ the sheaf of real valued smooth functions on a given manifold. 
In the present section we give the necessary mathematical background and we introduce notations that will be used later~on.

\subsection{Generalities on Lie $\infty$-algebroids}\label{secccc}

There exists
two equivalent conventions for $L_\infty$-algebra brackets:  the graded skew-symmetric version  is the original and most natural convention \cite{ladaIntroductionSHLie1993, ladaStronglyHomotopyLie1995}, while
the graded symmetric version can be more useful for computations; the latter convention being sometimes called $L_\infty[1]$-algebras~\cite{schatzBFVcomplexHigherHomotopy2008, mehtaAlgebraActions2012}.
An $L_\infty$-algebra structure on $V$ is equivalent to an $L_\infty[1]$-algebra structure on the desuspension of $V$, denoted $s^{-1}V$ or $V[1]$\footnote{In the present paper, we adopt the convention that the suspension operator $s$ (also denoted $[-1]$) increases the degree of the associated space by $1$, while the desuspension operator $s^{-1}$ (also denoted $[1]$) decreases the degree by $1$. In other words, $(sV)_{i}=V_{i-1}$ while $(s^{-1}V)_{i}=V_{i+1}$.} \cite{voronovHigherDerivedBrackets2005}. In this section -- and more generally in this paper -- we will restrict ourselves to the use of the graded skew-symmetric convention to stick with the analogy with Lie algebras, as it seems more natural and easier to understand for newcomers. For an equivalent formulation under the symmetric convention, see \cite{laurent-gengouxUniversalLieInfinityAlgebroid2020}.

 %$\{\cdot,\cdot\}_2$ that satisfies, $\{\{x,y\}_2,z\}_2 + \{\{z,x\}_2,y\}_2 +\{\{y,z\}_2,x\}_2=0$; in particular, this bracket can be identified with an antisymmetric bracket $[\cdot , \cdot]$ 
%on the sections of $A$ satisfying the standard Jacobi identity.

\begin{definition} \cite{ladaIntroductionSHLie1993, ladaStronglyHomotopyLie1995}
An \emph{$L_\infty$-algebra} is a graded vector space $ {V} = \bigoplus_{ j\in\mathbb{Z}} {V}_{j}  $ 
together with a family of graded skew-symmetric $k$-multilinear maps $ \big( l_k\big)_{k \geq 1} $ of  degree $2-k$, called \emph{$k$-brackets}, which 
satisfy a set of compatibility conditions called \emph{higher Jacobi identities}.
%\item $ \dd^2=0$, where $\dd:= \{\cdot \}_1$ is the $1$-ary bracket,
For all $k\geq1$ and for every $k$-tuple of homogeneous elements $a_{1},\ldots,a_{k}\in {V}$, they are given by:
\begin{equation}\label{superjacobi}
\sum_{j=1}^{k}(-1)^{j(k-j)}\sum_{\sigma\in {Un}(j,k-j)}\epsilon(\sigma)l_{k-j+1}\big(l_j(a_{\sigma(1)},\ldots,a_{\sigma(j)}),a_{\sigma(j+1)},\ldots,a_{\sigma(k)}\big)=0
\end{equation}
where  $Un(j,k-j)$ denotes the set of \emph{$(j,k-j)$-unshuffles}, and where $\epsilon(\sigma)$ is the sign induced by the permutation of the elements $a_1,\ldots, a_k$ under $\sigma$ in the exterior algebra of $V$:
  %where $(-1)^\sigma$ is the signature of the permutation $\sigma$, and where $\epsilon(\sigma)$ is the \emph{Koszul sign} induced by the permutation of the elements $x_1,\ldots, x_k$ under $\sigma$:
\begin{equation}
 a_{\sigma(1)}\wedge\ldots \wedge a_{\sigma(k)} =\epsilon(\sigma)\, a_1\wedge\ldots\wedge a_k
\end{equation}
\end{definition}
%A $L_\infty$-algebra is said to be a \emph{Lie $n$-algebra} when ${E}_{-i}=0$ for all $i \geq n$.

\begin{remarque}
The definition implies that the 1-bracket is a differential on $V$ that it is compatible with the 2-bracket, i.e. it is a derivation of the 2-bracket:
\begin{equation}\label{leibnizzzz}
%\big\{\{x,y\}_2\big\}_1=-\big\{\{x\}_1,y\big\}_2-(-1)^{|x|}\big\{x,\{y\}_1\big\}_2
l_1\big(l_2(a,b)\big)=l_2\big(l_1(a),b\big)+(-1)^{|a|}l_2\big(a,l_1(b)\big)
\end{equation}
Moreover, the 2-bracket is a Lie bracket \emph{up to homotopy}:
\begin{equation}\label{jacobiiiii}
l_2\big(l_2(a,b),c\big)+\circlearrowleft\ =-[l_1,l_3](a,b,c)
\end{equation}
where on the right-hand side, $[\,.\,,.\,]$ symbolizes the graded commutator of operators on $\mathrm{End}(\wedge^\bullet V)$.
\end{remarque}

We can now define what is a \emph{Lie $\infty$-algebroid}:

\begin{definition} \label{def:Linftyoids}
 %Let  $M$ be a smooth manifold and let $E=(E_{-i})_{i\geq0}$ be a graded vector bundle over $M$.
 A Lie $\infty$-algebroid over a smooth manifold $M$ is a triple $\big(E,(l_k)_{k\geq1},\rho\big)$, where $E=(E_{-i})_{i\geq0}$ is a non-positively graded vector bundle over $M$, $(l_k)_{k\geq1}$ is a family of $\mathbb{R}$-linear brackets on  the sheaf $\Gamma(E)$ of sections of $E$ defining an $L_\infty$-algebra structure on $\Gamma(E)$, 
%\begin{enumerate}
and $\rho\colon E_{0} \to TM$ is a vector bundle morphism, called the \emph{anchor map},
%\item and a family, for all $n \geq 1$, of graded symmetric $n$-multilinear maps $ \big(\{ \ldots \}_n\big)_{n \geq 1} $ of  degree $+1$, called $n$-ary brackets,
% on the sheaf of graded vector spaces $\Gamma(E)$,
%\end{enumerate}
%satisfying a set of constraints.
such that the following items hold true:%\footnote{In fact, the fifth item follows from the third and fourth one.}% the \emph{Leibniz conditions} hold:
\begin{enumerate}
\item For every $k\geq1$ the $k$-brackets $l_k$ are always $\mathcal{C}^\infty$-linear in each of their $k$ arguments, except when $k=2$ and at least one of its two entries has degree zero \begin{equation}\label{robinson}
 l_2(a,fb)=fl_2(a,b)+ \rho(a)[f] \, b
 \end{equation}
 for all $a \in \Gamma(E_{0})$, $b \in \Gamma(E)$ and $f\in\mathcal{C}^\infty$;%  \item The anchor map is compatible with the bracket on $\Gamma\big(S^2(E_{-1})\big)$. For all $x,y\in\Gamma(E_{-1})$:
% \begin{equation}
% \rho\big(\{x,y\})\big)=\big\{\rho(x),\rho(y)\big\}
% \end{equation}
 % whereas $ \{x,fy\}_2=f\{x,y\}_2$ for all $x \in \Gamma(E_{-i})$ with $i \geq 2$,
%\end{enumerate}
%The second ones are the compatibility conditions of the anchor:
%\begin{enumerate}
\item The anchor map satisfies, for every $a\in \Gamma(E_{-2})$:
\begin{equation}\label{eq:direct}
\rho \big(l_1(a)\big)=0
\end{equation}
\end{enumerate}
A Lie $\infty$-algebroid is said to be a \emph{Lie $n$-algebroid} when ${E}_{-i}=0$ for all $i \geq n$. In that case we say that $E$ is of \emph{length $n$}. A Lie $1$-algebroid is called a \emph{Lie algebroid}.
\end{definition}

\begin{remarque}\begin{enumerate}

\item Equation \eqref{robinson} together with Equation \eqref{jacobiiiii}
 imply that the anchor map and the 2-bracket on $\Gamma(\wedge^2E_{0})$ satisfy the same relationship as in the Lie algebroid case:
 \begin{equation}\label{jpp}
 \rho\big(l_2(a,b)\big)=\big[\rho(a),\rho(b)\big]
 \end{equation}
 for every  $a,b\in\Gamma(E_{0})$, and where $[\,.\,,.\,]$ is the Lie bracket of vector fields on $M$.
 
 \item The $\mathcal{C}^\infty$-linearity of $l_1$ together with the Jacobi identity $l_1^2=0$ implies that $l_1$ defines a chain complex of vector bundles:
\begin{center}
\begin{tikzcd}[column sep=0.7cm,row sep=0.4cm] 
\ldots\ar[r,"l_1"]&E_{-2}\ar[r,"l_1"]&E_{-1}\ar[r,"l_1"]&E_{0}\ar[r,"\rho"]&TM\ar[r]&0
\end{tikzcd}
\end{center}
We call the triple $(E,l_1,\rho)$ the \emph{linear part} of the Lie $\infty$-algebroid.

%\item The grading convention in Definition \ref{def:Linftyoids} differs from the one in Definition 1.16 in \cite{laurent-gengouxUniversalLieInfinityAlgebroid2020}. This is due to the fact that we chose the skew-symmetric (unshifted) convention for $L_\infty$-algebras in the present paper. Obviously, there is a one-to-one correspondence between the notion of Lie $\infty$-algebroids that is given in Definition \ref{def:Linftyoids} and the one that is given in Definition 1.16 of \cite{laurent-gengouxUniversalLieInfinityAlgebroid2020}: it is the same correspondence that stands between $L_\infty$-algebras and $L_\infty[1]$-algebras.
\end{enumerate}
\end{remarque}

Given a Lie $\infty$-algebroid $E=\bigoplus_{i\geq0}E_{-i}$, we can define a cohomology generalizing the Lie algebroid cohomology and the Chevalley-Eilenberg cohomology of Lie algebras. We call \emph{forms on $E$} elements  of the graded vector space $\Omega^\bullet(E)=\bigoplus_{k\geq1}\Omega^k(E)$, where $\Omega^k(E)=\Gamma\big(\bigwedge^k E^*\big)$. This space carries a bidegree: the form degree $k$ and the degree $p$ induced by the  grading of $E^*$, which is reverse than that of $E$.  More precisely, since the Lie $\infty$-algebroid  $E$ is non-positively graded, the grading $p$ takes values in non-negative integers. Then, we say that a form $\eta$ is a \emph{$k$-form of degree $p$} %the total degree of the $k$-form $\eta$ is $p$ i.e.
 if $\eta\in\Omega^k(E)\big|_p=\Gamma\big(\bigwedge^k E\big)\big|_p$. The sum $h=k+p$ is called the \emph{height} of $\eta$, and we write $\Omega^k_h(E)$ instead of $\Omega^k(E)_p$ if one wants to emphasize the use of the height instead of the degree on $E^*$ (which can always be found again by setting $p=h-k$). For every $h\geq0$, one sets
 \begin{equation*}
\Omega_h(E)=\bigoplus_{k\geq0}\Omega^k_h(E)\qquad \text{and}\qquad \Omega_\blacktriangle(E)=\bigoplus_{h\geq0}\Omega_h(E)\end{equation*}

\begin{remarque}
The height is defined so that, when $E$ is a Lie algebroid (a Lie $\infty$-algebroid concentrated in degree 0, then), the height of a $k$-form is precisely $k$. This is the correct notion to use when generalizing results on Lie algebroid cohomology to Lie $\infty$-algebroid cohomology. 
\end{remarque}

We define a sequence of operators $\big({d}^{(s)}_E\colon\Omega^\bullet(E)\to\Omega^{\bullet+s}(E)\big)_{s\geq0}$ by their action on $k$-forms. The index $s$ is called \emph{the arity} and measures by how much the form-degree is increased under the action of $d_E^{(s)}$ on a  $k$-form $\eta\in\Omega^k(E)$, for $k\geq1$:
\begin{align}
d^{(0)}_E\eta\big(a_1,\ldots,a_k\big)&=\sum_{\sigma\in\, Un(1,k-1)}\epsilon(\sigma)\,\eta\big(l_1(a_{\sigma(1)}),a_{\sigma(2)},\ldots,a_{\sigma(k)}\big)\label{eqdiff1}\\
d^{(1)}_E\eta\big(a_1,\ldots,a_{k+1}\big)&=\sum_{\sigma\in\, Un(1,k)}\epsilon(\sigma)\,\rho(a_{\sigma(1)})\big[\eta\big(a_{\sigma(2)},\ldots,a_{\sigma(k+1)}\big)\big]\label{eqdiff2}\\
&\hspace{1cm}-\sum_{\sigma\in\, Un(2,k-1)}\epsilon(\sigma)\,\eta\big(l_2(a_{\sigma(1)},a_{\sigma(2)}),a_{\sigma(3)},\ldots,a_{\sigma(4)}\big)\nonumber
\end{align}
and, for every $s\geq2$, as:
\begin{equation*}
d^{(s)}_E\eta\big(a_1,\ldots,a_{k+s}\big)=(-1)^{s}\sum_{\sigma\in\, Un(s+1,k-1)}\epsilon(\sigma)\,\eta\big(l_{s+1}(a_{\sigma(1)},\ldots,a_{\sigma(s+1)}),a_{\sigma(s+2)},\ldots,a_{\sigma(k+s)}\big)
\end{equation*}
On zero-forms, i.e. smooth functions on $M$, the only operator acting is $d^{(1)}_E\colon\mathcal{C}^\infty(M)\longrightarrow\Gamma\big((E_{0})^*\big)$, via the anchor map:
\begin{equation}\label{ancre}
d^{(1)}_Ef(a)=\rho(a)[f]
\end{equation}
for every $f\in\mathcal{C}^\infty(M)$, and $a\in\Gamma(E_0)$.
%
%When $s=1$ we define ${d}^{(1)}_E$ from the anchor map and the 2-bracket as: \begin{align}
%{d}_E^{(s)}\eta(a_0,\ldots,a_k)&= \sum_{0\leq i\leq k} (-1)^i \rho(a_i)\big[\eta(a_0,\ldots,\widehat{a_i},\ldots,a_k)\big]\label{algebroiddiff}\\
%&\hspace{1cm}+\sum_{0\leq i<j\leq k}(-1)^{i+j}\eta\big([a_i,a_j]_A,a_0,\ldots,\widehat{a_i},\ldots,\widehat{a_j},\ldots,a_k\big)\nonumber
%\end{align}
%for every $a_0,\ldots, a_k\in\Gamma(E)$. When $k=0$, we set:
% \begin{equation}\label{algebroiddiffbis}
% {d}_E^{(0)}(f)(a)=\rho(a)[f]
% \end{equation}.  When $s\neq1$ we have two cases: if $k=0$, we set ${d}^{(s)}_E=0$, but if $k\geq1$ we set:
%\begin{equation}
%{d}^{(s)}_E\eta(a_0,\ldots,a_{k+s-1})= \sum_{\sigma\in {\rm Un}(s,k-s)}\epsilon(\sigma)\,\eta\big(l_{s+1}(a_{\sigma(0)},\ldots,a_{\sigma(s)}),a_{\sigma(s+1)},\ldots,a_{\sigma(k+s-1)}\big)\label{algebroiddiff2}
%\end{equation}
%for every $\eta\in\Omega^k(E)$ and $a_0,\ldots,a_k\in\Gamma(E)$. 

From these formulas, the higher Jacobi identities \eqref{superjacobi}  together with Equation \eqref{jpp} -- which is itself a consequence of the former -- imply that ${d}_E=\sum_{s\geq0} {d}^{(s)}_E$ is a differential, i.e. that $d_E^2=0$. %However, this differential does not  the form degree, it is not very meaningful since $d_E(\Omega^\bullet(E))\subset \bigoplus_{s\geq0}\Omega^{\bullet+s}(E)$.
The degree of this operator cannot be based on the form degree alone as \begin{equation*}d_E(\Omega^\bullet(E))\subset \bigoplus_{s\geq0}\Omega^{\bullet+s}(E)\end{equation*} but a straightforward calculation shows that $d^E$ increases the height by $1$, turning $\Omega_\blacktriangle(E)$ into a co-chain complex.
\begin{definition}\label{defccco}
The cohomology of the co-chain complex $\big(\Omega_\blacktriangle(E),{d}_E\big)$ is called the \emph{Lie $\infty$-algebroid cohomology of $E$}, and is denoted $H^\blacktriangle(E)$, or $H(E)$ for short. 
\end{definition}

%\begin{remarque}
%When $E$ is a Lie algebroid, the height and the form degree coincide so the Lie $\infty$-algebroid of $E$ coincides with its Lie algebroid cohomology.
%\end{remarque}

%Conversely, given a graded vector bundle $E=\bigoplus_{i\geq0} E_{-i}$ over $M$, an operator $d_E\colon\Omega_{\blacktriangle}(E)\to\Omega_{\blacktriangle+1}(E)$ encodes all $s$-brackets of some Lie $\infty$-algebroid structure through its component $d_E^{(s)}$, as well as a compatible anchor map in $d_E^{(1)}\big|_{\Omega^0(E)}$. The homological condition $d_E^2=0$ is equivalent to the higher Jacobi identities that these brackets should satisfy. This discussion can be made more precise by using the language of differential graded manifolds, see for example \cite{vaintrob, voronov ?}. 
The Lie $\infty$-algebroid cohomology is analogous to the Chevalley-Eilenberg cohomology for Lie algebras, and coincides with the Lie algebroid cohomology when $E$ is a Lie algebroid. %In both latter cases the Chevalley-Eilenberg differential contains all the informations on the Lie algebra/oid structure .
The one-to-one correspondence between Lie algebroid structures on a vector bundle $A$ and a differential graded manifold structure on $A[1]$ \cite{vaintrobLieAlgebroidsHomological1997}
 extends to the Lie $\infty$-algebroid context: %the Chevalley-Eilenberg differential $d_E$ corresponds to the homological vector field $Q$ defined over $E[1]$ which encodes the Lie $\infty$-algebroid structure on $E$  \cite{voronovManifoldsHigherAnalogs2010}.

\begin{theoreme} \cite{voronovManifoldsHigherAnalogs2010}
Let $E=\bigoplus_{i\geq0} E_{-i}$ be a non-positively graded vector bundle over $M$. Then there is a one-to-one correspondence between Lie $\infty$-algebroid structures on $E$ and linear operators $d_E\colon\Omega_{\blacktriangle}(E)\to\Omega_{\blacktriangle+1}(E)$ satisfying the homological condition $d_E^2=0$.
\end{theoreme}

This correspondence provides a very efficient way of defining morphisms of Lie $\infty$-algebroids~\cite{bonavolontaCategoryLieNalgebroids2013}:

\begin{definition}
Let $E$ (resp. $E'$) be a Lie $\infty$-algebroid with base manifold $M$ (resp. $M'$). Then a  \emph{Lie $\infty$-morphism between $E$ and $E'$} is a height-preserving graded algebra morphism $\Phi\colon\Omega_\blacktriangle(E')\to\Omega_\blacktriangle(E)$ that intertwines $d_E$ and $d_{E'}$:
\begin{equation}\label{ultra}
d_E\circ\Phi=\Phi\circ d_{E'}
\end{equation}
\end{definition}

%To say that $\Phi$ is height-preserving means that it sends elements of height $h$ in $\Omega(E')$ to elements of height $h$ in $\Omega(E)$.
 The Lie $\infty$-morphism $\Phi$ can be decomposed into components of various arities: $\Phi=\sum_{s\geq0} \Phi^{(s)}$  where $\Phi^{(s)}$ is a map from $\Omega^\bullet(E')$ to $\Omega^{\bullet+s}(E)$. Each of these components, for $s\geq1$, defines a dual map $f_s\colon\Gamma(\wedge^sE)\to \Gamma(E')$ of degree $1-s$ satisfying some intricate consistency conditions generalizing the Lie algebra homomorphism condition \cite{kajiuraHomotopyAlgebrasInspired2006}.
Also the restriction of $\Phi$ to $\mathcal{C}^\infty(M')$ induces a map of smooth manifolds $\varphi\colon M\to M'$ so that for every $f\in\mathcal{C}^\infty(M')$ and $\eta\in\Omega(E')$, we have:
\begin{equation}
\Phi(f\eta)=\varphi^*(f)\Phi(\eta)
\end{equation}
When $M'=M$ and $\varphi=\mathrm{id}_M$, we say that $\Phi$ is \emph{over $M$}.

Morphisms of Lie $\infty$-algebroids admit homotopies, that admit in turn homotopies of homotopies and so on. %In the present section we only summarize the notion of homotopy and 
We refer to Section 3 of \cite{laurent-gengouxUniversalLieInfinityAlgebroid2020} for a detailed discussion on this notion. We will only say that this notion of homotopy between Lie $\infty$-morphisms coincides, when the Lie $\infty$-algebroids are of finite length, with the usual cylinder homotopies (see Proposition 3.61 in~\cite{laurent-gengouxUniversalLieInfinityAlgebroid2020}).
Homotopy is an  equivalence relation -- denoted $\sim$ -- between Lie $\infty$-morphisms,  which allows us to define a notion of equivalence between Lie $\infty$-algebroids: %\cite{laurent-gengouxUniversalLieInfinityAlgebroid2020}:
\begin{definition}
	\label{def:homtequiv}
Let $E$ and $E'$ be two Lie $\infty$-algebroids over $M$ and $\Phi\colon \Omega_\blacktriangle(E')\to\Omega_\blacktriangle(E)$ a Lie $\infty$-algebroid morphism between them.
%from $(E,Q)$ to $(E',Q')$.% and from $(E',Q')$ to $(E,Q)$ respectively. 
We say that $\Phi$ is a \emph{homotopy equivalence} if there exists a Lie $\infty$-morphism $\Psi\colon \Omega_\blacktriangle(E)\to\Omega_\blacktriangle(E')$
such that
\begin{equation*}%\label{higherhomotopy}
\Phi\circ\Psi\sim\mathrm{id}_{\Omega(E)}\hspace{1cm}\text{and}\hspace{1cm}\Psi\circ\Phi\sim\mathrm{id}_{\Omega(E')}.
\end{equation*}
In such a case,  the Lie $\infty$-algebroids $E$ and $E'$ are said to be \emph{homotopy equivalent}.
\end{definition}

\begin{remarque} In particular, any homotopy equivalence between $E$ and $E'$ is a quasi-isomorphism, i.e. it is an isomorphism at the cohomology level. Since moreover any two homotopy equivalence are homotopic (see Theorem 2.8 in  \cite{laurent-gengouxUniversalLieInfinityAlgebroid2020}) we deduce that this isomorphism at the cohomology level is canonical.
\end{remarque}

\subsection{Representations up to homotopy of Lie $\infty$-algebroids}\label{appendix}

 Representations up to homotopy of a  Lie algebroid $A$ \cite{abadRepresentationsHomotopyLie2011} are characterized by the fact that the vector bundle $K$ serving as the representation of $A$ is promoted to a graded vector bundle, and the  $A$-connection $\nabla$ on $K$ becomes a family of `connections $k$-forms'. Then it is required that the curvatures associated to these `connections $k$-forms' vanish only \emph{up to homotopy}.  Such notion of representation up to homotopy, although intricate, emerge in a very natural way in the definition of the \emph{adjoint representation} of a Lie algebroid. Representations up to homotopy of Lie algebroids can be generalized to Lie $\infty$-algebroids  \cite{caseiroModularClassLie2022, jotzleanModulesRepresentationsHomotopy2023}, although it is a bit more involved due to the fact that a Lie $\infty$-algebroid is itself graded. For self-consistency of the paper, the present section is devoted to present this notion based on the graded skew-symmetric convention for Lie $\infty$-algebroids, building on the original notations of~\cite{abadRepresentationsHomotopyLie2011}. See also \cite{mehtaLieAlgebroidModules2014} for a discussion about the various notions of representations of Lie algebroids.

Let $E=\bigoplus_{i\geq0}E_{-i}$ be a Lie $\infty$-algebroid and let $K=\bigoplus_{j\in\mathbb{Z}}K_{j}$ be a graded vector bundle, both defined over a smooth manifold $M$.
%algebroids, we need to allow forms on a Lie $\infty$-algebroid $E$ over a smooth manifold $M$ to take values in a graded vector bundle $K=\bigoplus_{i\in\mathbb{Z}} K_i$ over $M$, and correspondingly introduce a new grading. 
A $k$-form on $E$ taking values in $K$ of \emph{total degree} $d$ is a finite sum of $k$-forms $\eta=\sum_{l=1}^m\eta_{l}$ of respective height $h_l\geq0$ (as defined in Section \ref{secccc}), each of the $\eta_{l}$ taking values in some $K_{j_l}$, where $j_l$ is such that $d=h_l+j_l$. %Then $\eta_l$ is an element of $\Omega^k_{h_l}(E)\otimes K_{j_l}$.
   For example, a $3$-form $\eta\in (E_{-a})^*\wedge (E_{-b})^*\wedge (E_{-c})^*\otimes K_{j}$ has a total degree of  $3+a+b+c+j$, where the integer 3 is the form degree and the height is $3+a+b+c$. We emphasize the requirement that the sum should be finite. %, because it may happen that 
  
  One sets $\Omega^k_{h}(E,K_{j})$ to be the space of $k$-forms of height $h$ taking values in $K_j$.   For every integer $d$, one then sets $\Omega^k(E,K)_d$ to be the space of $K$-valued $k$-forms of total degree $d$:
 \begin{equation}
\Omega^k(E,K)_d= \bigoplus_{h\geq k}\Omega^k_{h}(E, K_{d-h}) %\bigoplus_{j\geq0}\Gamma\Big(\bigwedge^kE^*\Big|_{j}\otimes K_{d-k-j}\Big) = 
\end{equation}
 This notation should not be confused with the notation corresponding to singling out an element of degree $d$ of a $k$-form:
 \begin{equation}
 \Omega^k(E,K)_d\neq \Omega^k(E,K)\big|_d
 \end{equation}
 In the former case, one takes into account the form degree $k$ to compute $d$, while in the latter case, one only sums the contributions from $E^*$ and that from $K$, but does not take into account the form degree $k$.
 As the black triangle in $\Omega_{\blacktriangle}(E)$ denotes the height of a form, 
in the rest of the text, the black lozenge will denote the total degree:
\begin{equation*}\Omega^\bullet(E,K)_\blacklozenge=\bigoplus_{k\geq0}\bigoplus_{d\in\mathbb{Z}}\Omega^k(E,K)_d
\end{equation*}
Sometimes the latter space $\Omega^\bullet(E,K)_\blacklozenge$ is denoted $\Omega(E,K)_\blacklozenge$ when emphasis is made on the total degree. 
% While the label of the height is localized right before the parenthesis (e.g. as in~$\Omega^k_h(E)$), the total degree is localized at the very bottom right corner.
  %In particular, the total degree of a form on $E$ merely corresponds to its height.

A Lie algebroid being a Lie $\infty$-algebroid concentrated in degree 0, the only degrees that enter the total degree are the form (polynomial) degree and the grading of $K$, which is actually  what appears for representations up to homotopy of Lie algebroids \cite{abadRepresentationsHomotopyLie2011}. We now have enough material to generalize the latter notion to Lie $\infty$-algebroids~\cite{caseiroModularClassLie2022, jotzleanModulesRepresentationsHomotopy2023}:

\begin{definition}  \label{raquelLG}
A \emph{representation up to homotopy} of a Lie $\infty$-algebroid $E$ is a graded vector bundle $K$ over $M$ together with a linear map $D\colon\Omega(E,K)_{\blacklozenge}\to\Omega(E,K)_{\blacklozenge+1}$ of total degree $+1$ satisfying:
\begin{equation}\label{definitionhomotop}
D(\eta\wedge u)={d}_E(\eta)\wedge u+(-1)^{m}\eta\wedge D(u)
\end{equation}
for every $\eta\in\Omega_m(E)$ and $u\in\Omega(E,K)$, and such that $D^2=0$.
%A \emph{morphism between two representations $(K,D)$ and $(K',D')$ of $E$} is a $\Omega(E)$-linear map $\Phi\colon\Omega(E,K)\to\Omega(E,K')$ of total degree 0 that commutes with the respective operators $D$ and $D'$. We say that $\Phi$ is an \emph{isomorphism of representations up to homotopy} if it induces an isomorphism at the cohomology level.
\end{definition}

The differential $D$ is more intricate than in the Lie algebroid case.
Here, the operator $D$ splits into a sum $D=\sum_{k\geq0}D^{(k)}$ where each component is a $\mathbb{R}$-linear map $
D^{(k)}\colon\Omega^\bullet(E,K)_{\blacklozenge}\to \Omega^{\bullet+k}(E,K)_{\blacklozenge+1}$ 
 of arity $k$ and of total degree $+1$ . %That is to say: the graded vector space $\Omega^\bullet(E,K)_\blacklozenge$ admits a form (polynomial) degree (symbolized by $\bullet$) and a total degree (symbolized by the $\blacklozenge$). 
 Each of these operators $D^{(k)}$ can be further decomposed with respect to the height: $D^{(k)}=\sum_{h\geq k}D^{(k)}_h$, where each component is a $\mathbb{R}$-linear map: \begin{equation}D^{(k)}_h\colon\Omega^\bullet_\blacktriangle(E,K)_{\blacklozenge}\to \Omega^{\bullet+k}_{\blacktriangle+h}(E,K)_{\blacklozenge+1}\end{equation} %Recall that the height of a $K_i$-valued $k$-form $\omega$ on $E$ is the total degree of the part of $\omega$ that resides in $\Omega^k(E)$, so that the total degree of $\omega$ is $\blacklozenge=h+i$.
 %When $E$ is a Lie algebroid, the fiber of $E^*$ has degree $0$, hence $p=0$ so $h=k$, and the operator $D^{(k)}$ restricts to $D_{k,0}\colon\Omega^\bullet(E,K_{\blacklozenge})\to\Omega^{\bullet+k}(E,K_{\blacklozenge+1-k})$. 
 %The action of $D^{(k)}_h$ on $\Omega(E,K)$ is defined by its action on $\Gamma(K)$ and by Equation \eqref{definitionhomotop}.
 Since the vector bundle $K$ is graded, the vector bundle of fiberwise endomorphisms of $K$ is graded as well: $\mathrm{End}(K)=\bigoplus_{i\in\mathbb{Z}}\mathrm{End}(K)_i $.
 For every $k\neq1$, the restriction of the action of $D^{(k)}_h$ to $\Gamma(K)$ canonically induces an $\mathrm{End}(K)_{1-h}$-valued $k$-form $\omega^{(k)}_h$ on $E$:
  \begin{equation}
  D^{(k)}_h\colon\Gamma(K_i)\longrightarrow \Omega^k_h(E,K_{i+1-h})\qquad\Longleftrightarrow \qquad \omega^{(k)}_h\in \Omega^k_h\big(E,\mathrm{End}(K)_{1-h}\big)
  \end{equation}
where the space on the right hand side should be understood as:
\begin{equation*}
\Omega^k_h\big(E,\mathrm{End}(K)_{1-h}\big)=\sum_{i_1+\ldots+i_k=h-k}\Gamma\Big((E^*)_{i_1}\wedge\ldots\wedge (E^*)_{i_k}\Big)\otimes\mathrm{End}(K)_{1-h}
\end{equation*}
%We will see below that the object $\omega^{(k)}=\sum_{h\geq k}\omega^{(k)}_h$ can be seen as a `connection $k$-form' on~$E$.
%Conversely, the family $\big(\omega^{(k)}\big)_{k\geq0}$ of $\mathrm{End}(K)$-valued $k$-forms of total degree $+1$ induces  the desired family of operators $D^{(k)}$ by the following correspondences:
%\begin{equation}
%D^{(1)}(u)=d^{\nabla}(u)\qquad \text{and}\qquad D^{(k)}(u)=\omega^{(k)}\wedge u\quad \text{for $k\neq1$}\label{eqokkk}
%\end{equation}

 The 0-connection $\omega^{(0)}\in\Gamma\big(\mathrm{End}_{1}(K)\big)$, being a zero-form, has height $0$. Being of total degree~$+1$ implies that it is a degree $+1$ vector bundle morphism on~$K$ that we denote~$\partial:K_{\bullet}\to K_{\bullet+1}$.
%Next, the connection form $\omega^{(1)}$ corresponds to what we call an \emph{$E$-connection}, that is to say, a family of differential operators  $\big(\nabla^{i}:\Gamma(E_{-i})\times\Gamma(K_\bullet)\to \Gamma(K_{\bullet-i})\big)_{i\geq0}$ satisfying the usual axioms of a connection:
%\begin{align}
%\nabla_{fe}^i(s)&=f\nabla_e^i(s)\label{eqrep1}\\
%\nabla_e^i(fs)&=f\nabla_e^i(s)+\rho(e)[f]\,s \label{eqrep2}
%\end{align}
%for every $f\in\mathcal{C}^\infty, e\in\Gamma(E_{-i})$ and $s\in\Gamma(K)$, and where for consistency we assume that $\rho|_{E_{-i}}=0$ if $i\geq1$.  % Equations~\eqref{eqrep1} and~\eqref{eqrep2}.
%Then, the component $D^{(1)}$ is the associated covariant derivative $d^{\nabla}$. More generally, 
For $k\geq2$, the correspondence between  $D^{(k)}$ and $\omega^{(k)}$ satisfies the following rule:
 \begin{equation}
%D^{(1)}(u)=d^{\nabla}(u)\qquad \text{and}\qquad 
D^{(k)}(u)=\omega^{(k)}\wedge u\quad \text{for $k\neq1$}\label{eqokkk}
\end{equation}
for every $u\in\Omega(E,K)$. Here, the wedge product means that we apply a concatenation on the $E$-form part of $u$, and that $\omega^{(k)}$ acts as an endomorphism on the part of $u$ sitting in $K$.  We call \emph{connection $k$-form} the $\mathrm{End}(K)$-valued $k$-form $\omega^{(k)}$. For $k=1$, the operator $D^{(1)}$ is the exterior covariant derivative associated to   what we call an \emph{$E$-connection}, that is to say, a family of differential operators  $\big(\nabla^{i}:\Gamma(E_{-i})\times\Gamma(K_\bullet)\to \Gamma(K_{\bullet-i})\big)_{i\geq0}$ satisfying the usual axioms of a connection:
\begin{align}
\nabla_{fe}^i(s)&=f\nabla_e^i(s)\label{eqrep1}\\
\nabla_e^i(fs)&=f\nabla_e^i(s)+\rho(e)[f]\,s \label{eqrep2}
\end{align}
for every $f\in\mathcal{C}^\infty, e\in\Gamma(E_{-i})$ and $s\in\Gamma(K)$, and where for consistency we assume that $\rho|_{E_{-i}}=0$ if $i\geq1$.   The exterior covariant derivative $D^{(1)}$ is also denoted $d^\nabla$.

\begin{remarque} When the grading of the graded vector bundle $K$ is bounded (which will be the case in the present paper), %ranges from $k_{min}$ to $k_{max}$,
 then there exists an open covering of $M$ by open sets $U_\alpha$ trivializing~$K$.
Over such a covering,  to every  differential operator~$\nabla^i$ corresponds a family of locally defined 1-forms $\omega^{(1),i}_\alpha\in\Omega^1\big(E_{-i}\big|_{U_\alpha},\mathrm{End}_{-i}(K)\big|_{U_\alpha}\big)$. %More precisely, for every $i\geq0$, there is a family of locally defined connection 1-forms $\omega^{(1),i}_\alpha\in\Omega^1\big(E_{-i}\big|_{U_\alpha},\mathrm{End}_{-i}(K)\big|_{U_\alpha}\big)$ associated to $\nabla^i$. 
%For $i\geq1$ the object $\omega^{(1),i}_\alpha$ is a differential 1-forms which transforms following the usual formula $\omega^{(1),i}_\beta=g_{\alpha\beta}^{-1}\omega^{(1),i}_\alpha g_{\alpha\beta}$
On the overlap $U_\alpha\cap U_\beta$, the family of local 1-forms~$\omega^{(1),i}_\alpha$ satisfies the following conditions, depending on the value of $i$:
\begin{equation}\label{eq:transitionmap}
\omega^{(1),i}_\beta=g_{\alpha\beta}^{-1}\omega^{(1),i}_\alpha g_{\alpha\beta}\ \text{ if $i\geq1$ and }\ 
\omega^{(1),0}_\beta=g_{\alpha\beta}^{-1}d_E^{(1)}g_{\alpha\beta}+g_{\alpha\beta}^{-1}\omega^{(1),0}_\alpha g_{\alpha\beta}\ \text{ otherwise}
\end{equation}
where $g_{\alpha\beta}$ represents the transition map of $K$ over $U_\alpha\cap U_\beta$. 
A family  of locally defined 1-forms $(\omega^{(1),i}_\alpha)_{i,\alpha}$ satisfying Conditions \eqref{eq:transitionmap} is abusively called a \emph{connection 1-form} and denoted $\omega^{(1)}$ in the rest of the text. 
\end{remarque}

When decomposing the equation $D^2=0$ by form degree, one obtains for all $k\geq0$:
\begin{equation}
\sum_{l=0}^k\, D^{(l)}D^{(k-l)}=0\label{eq:tropezi1}
\end{equation}
For $k=0$, this equation means that $\partial:K_{\bullet}\to K_{\bullet+1}$ is a differential, turning $K$ into a chain complex.
Let us extend this differential by setting $\slashed{\partial}\colon\Omega^\bullet\big(E,\mathrm{End}(K)\big)_\blacklozenge\to \Omega^\bullet\big(E,\mathrm{End}(K)\big)_{\blacklozenge+1}$ to be the $\mathcal{C}^\infty$-linear map of total degree $+1$ defined by:
\begin{align}\label{slashedpartial}
\slashed{\partial}(\alpha)(a_1,\ldots,a_k;u)&=\partial\big(\alpha(a_1,\ldots,a_k;u)\big)-(-1)^{|\alpha|+|a_1|+\ldots+|a_k|}\alpha\big(a_1,\ldots,a_k;\partial(u)\big)\\
&\hspace{2cm}-\sum_{j=1}^k(-1)^{|a_1|+\ldots+|a_{j-1}|}\alpha\big(a_1,\ldots,l_1(a_j),\ldots,a_k;u\big)\nonumber
\end{align}
for every $k\geq0$, $\alpha\in\Omega^k\big(E,\mathrm{End}(K)\big)$, $a_1,\ldots,a_k\in\Gamma(E)$ and $u\in\Gamma(K)$. The symbol $|\alpha|$ denotes the total degree of $\alpha$ while the symbols $|a_j|$ denote the degree of the homogeneous element $a_j\in\Gamma(E)$. More precisely, if $\alpha$ is a $\mathrm{End}(K)$-valued $k$-form of height $h$ and of total degree $d$, i.e. if $\alpha\in\Omega^k_h\big(E,\mathrm{End}(K)_{d-h}\big)$, then $\slashed{\partial}(\alpha)\in\Omega^k_{h}\big(E,\mathrm{End}(K)_{d-h+1}\big)\oplus\Omega^k_{h+1}\big(E,\mathrm{End}(K)_{d-h}\big)$. A straightforward calculation shows that $\slashed{\partial}^2=0$, turning $\Omega^k(E,\mathrm{End}(K))_\blacklozenge$ into a chain complex for each $k\geq0$.

Next, Equation \eqref{eq:tropezi1} for $k=1$ means that the $E$-connection
%$\nabla_{fa}(s)=f\nabla_a(s)$ and $\nabla_a(fs)=f\nabla_a(s)+\rho(a)[f]\,s $.} 
$\nabla$ on $K$ is such that, for every $a\in\Gamma(E)$:
\begin{equation}\label{slashed}
\partial\circ\nabla_a-(-1)^{|a|}\nabla_a\circ\partial-\nabla_{l_1(a)}=0
\end{equation}
This can be equivalently written as $d_E^{(1)}\omega^{(0)}+\slashed{\partial}\big(\omega^{(1)}\big)=0$. For $k=2$, Equation \eqref{eq:tropezi1} reads:
\begin{equation}
R_{\nabla}+\slashed{\partial}\big(\omega^{(2)}\big)=0
\end{equation}
More generally, to every connection $k-1$-form $\omega^{(k-1)}$ is associated a \emph{curvature $k$-form} $R^{(k)}$  defined by:
\begin{align}
R^{(1)}&=d_E^{(1)}\omega^{(0)}\\
\forall\ k\geq2\qquad R^{(k)}&=\underset{s+t=k}{\sum_{1\leq s, t\leq k-1}}{d}_E^{(s)}\omega^{(t)}+\omega^{(s)}\circ\omega^{(t)} \label{definitioncurvature2} %{d}_\nabla\big(\omega^{(k)}\big)+\underset{s+t=k+1}{\sum_{2\leq s, t\leq k-1}}\omega^{(s)}\circ\omega^{(t)}=
\end{align}
%This definition is consistent with Equation \eqref{definitioncurvature} because if $E$ is a Lie algebroid, ${d}^{(s)}_E=0$ if $s\geq2$ and ${d}^{(1)}_E$ is the Lie algebroid differential. 
%Equation \eqref{definitioncurvature2} is not defined for $k=1$, so that 
Each curvature has total degree 2, while the index $k$ indicates the form degree of $R^{(k)}$. Then, the system of equations \eqref{eq:tropezi1} can be written in a way synthetizing that of Proposition 3.2 in~\cite{abadRepresentationsHomotopyLie2011}: %A representation up to homotopy of a Lie $\infty$-algebroid $E$ is uniquely defined by the graded vector bundle $K$, and by the differential operator $D$ satisfying $D^2=0$. By decomposing the latter equation with respect to form degree, 

\begin{proposition}There is a one-to-one correspondence between representations up to homotopy $(K,D)$ of $E$, and graded vector bundles $K$ endowed with a fiberwise linear differential $\partial:K_\bullet\to K_{\bullet+1}$, an $E$-connection $\nabla$ on $K$, and a family of $\mathrm{End}(K)$-valued $k$-forms on $E$ $\big(\omega^{(k)}\big)_{k\geq2}$ of total degree $+1$ satisfying the following set of equations:
\begin{equation}\label{eqimpot}
R^{(k)}+\slashed{\partial}\big(\omega^{(k)}\big)=0\hspace{1cm}\text{for every $k\geq1$}
\end{equation}
%where $R^{(k)}$ is the curvature associated to $\omega^{(k-1)}$ as in Equation \eqref{definitioncurvature2}.%, as defined in Equation \eqref{definitioncurvature2}.
\end{proposition}
%The set of Equations \eqref{eqimpot} can be decomposed degreewise to give:
%\begin{align}
%\partial^2&=0\label{ok1}\\
%d_E\omega^{(0)}+\slashed{\partial}\big(\omega^{(1)}\big)&=0\label{ok2}\\
%R^{(k)}+\slashed{\partial}\big(\omega^{(k)}\big)&=0\hspace{1cm}\text{for every $k\geq1$}\label{ok3}
%\end{align}
%where the curvature $k$-forms have been defined in Equation \eqref{definitioncurvature2}. 
%The differential $D$ can be split into a family of operators $D^{(k)}\in \Omega^k\big(E,\mathrm{End}(K)\big)_1$. 
%Thus we see that  Equation \eqref{ok2} is an alternative form of Equation \eqref{slashed}.
%Moreover, 
\begin{remarque}
As for Lie algebroids \cite{abadRepresentationsHomotopyLie2011}, representations up to homotopy of Lie $\infty$-algebroids are characterized by the fact that, for every $k\geq1$, the curvature $R^{(k)}$ corresponding to the connection $k-1$-form is flat \emph{up to homotopy}, and the homotopy is the connection $k$-form $\omega^{(k)}$. This implies that $R^{(k)}$ is a 2-coboundary in the chain complex $\big(\Omega^k\big(E,\mathrm{End}(K)\big)_\blacklozenge,\slashed{\partial}\big)$.
\end{remarque}
  %The one-to-one correspondence between an operator $D$ and a set of connections $k$-forms satisfying Equations \eqref{ok1}--\eqref{ok3} has been proven by Raquel Caseiro and Camille Laurent-Gengoux in \cite{caseiroModularClassLie2022}. 

\begin{example}\label{exadjoint}
As an example let us now present what is the adjoint representation of a Lie $\infty$-algebroid $(E,l_k,\rho)$, under our notations. It is defined on the graded vector bundle $K=E\oplus sTM$,  that is to say: $K_{-i}=E_{-i}$ for every $i\geq0$, and $K_{1}=sTM=TM[-1]$ where $s$ is the \emph{suspension operator} increasing the degree of $TM$ by $+1$.
% This convention is consistent with the grading of homotopy groups appearing in Freudenthal suspension theorem, see Corollary 4.24 in \cite{hatcherAlgebraicTopology2002}.}. %We need to define a family of maps $D^{(k)}\colon\Omega^\bullet\big(E,K)_\blacklozenge\to\Omega^{\bullet+k}\big(E,K)_{\blacklozenge+1}$ satisfying the assumptions of Definition \ref{raquelLG}.
We can then define a differential $\partial:K_\bullet\to K_{\bullet+1}$ by the following requirements:
\begin{equation*}
\partial\big|_{K_{-i}}=
\begin{cases}l_1\big|_{E_{-i}}\hspace{1.1cm}\text{for $i\geq1$}\\
s\circ \rho \hspace{1.32cm} \text{for $i=0$}
\end{cases}
\end{equation*}
% $\partial\big|_{K_{-i}}=l_1\big|_{E_{-i}}$ for every $i\geq1$, and $\partial\big|_{K_{0}}=s\circ \rho\colon E_0\to sTM$.
 This differential turns $K=E\oplus sTM$ into a chain complex.

The possibility of defining the adjoint representation of $E$ relies on a choice of a $TM$-connection $\nabla$ on $E$.
% that is compatible with $l_1$:
%\begin{equation}\label{eq:compatibnabla}
%\nabla\circ l_1=l_1\circ\nabla
%\end{equation}
%This identity is justified by the following argument: $E$ should be a representation up to homotopy of $TM$ \textbf{why?????}. Since this latter vector bundle is concentrated in degree 0, then it is not equipped with a differential, so that Equation \eqref{slashed} reduces to Equation \eqref{eq:compatibnabla}.
%We extend this connection to $K_1=sTM$ by the following
%For the sake of clarity we set $K=E\oplus sTM$, that is to say: $K_{-i}=E_{-i}$ for every $i\geq0$, and $K_{1}=sTM=TM[-1]$ where $s$ is the \emph{suspension operator} increasing the degree of $TM$ by $+1$\footnote{In the present paper, we adopt the convention that the suspension operator $s$ (also denoted $[-1]$) increases the degree of the associated space by $1$, while the suspension operator $s$ (also denoted $[1]$) decreases the degree by $1$. This convention is consistent with the grading of homotopy groups appearing in Freudenthal suspension theorem, see Corollary 4.24 in \cite{hatcherAlgebraicTopology2002}.}. 
Using $\nabla$, let us define an $E$-connection $\overline{\nabla}$ on $K$:
\begin{align}
\overline{\nabla}_{a}({b})&= l_2({a},{b})+\nabla_{\rho({b})}({a})\label{basicc1}\\
\overline{\nabla}_{a}(sX)&= s[\rho({a}),X]+[\partial,\nabla_{X}](a)\label{basicc2}
\end{align}
for every ${a},{b}\in\Gamma(E)$ and $X\in \mathfrak{X}(M)$, and where $\rho|_{E_{-i}}=0$ if $i\geq1$.  In particular, if $|b|\leq-1$ then the right-hand side of Equation  \eqref{basicc1} reduces to the usual bracket $l_2(a,b)$ while if $|a|\leq-1$ the right-hand side of Equation  \eqref{basicc2} reduces to $-\nabla_X\big(l_1(a)\big)$. Following  \cite{abadRepresentationsHomotopyLie2011}, we call this connection $\overline{\nabla}$ \emph{the basic connection associated to $\nabla$}.   We had to emphasize the suspension of $X$ in Equation~\eqref{basicc2}  in order to make explicit the fact that the basic connection is a degree 0 operator.

%As much as the basic connection $\overline{\nabla}$ admits an associated curvature:
%\begin{equation}\label{basiccurvatureinfty0}
%\overline{R}(a,b)=\big[\overline{\nabla}_a,\overline{\nabla}_b\big]-\overline{\nabla}_{l_2(a,b)}
%\end{equation}
As in the Lie algebroid case, we can define a basic curvature associated to the $TM$-connection $\nabla$. The basic 2-curvature $\overline{R}_{[2]}$ measures the `compatibility' between this connection and the 2-bracket $l_2$, corrected by some terms that makes it $\mathcal{C}^\infty$-linear \cite{abadRepresentationsHomotopyLie2011}: %. Its name comes from the observation that its composition with the anchor map $\rho$ is minus the curvature of the basic connection (see Proposition 2.11 in \cite{abadRepresentationsHomotopyLie2011}).  We then reproduce this definition in the current conventions for every $a,b\in\Gamma(E)$ and for every  $X\in \mathfrak{X}(M)$:
\begin{align}\label{basiccurvatureinfty1}
\overline{R}_{[2]}({a},{b})(sX)&=\nabla_{X}\big(l_2({a},{b})\big)-l_2\big(\nabla_{X}({a}),{b}\big)-l_2\big({a},\nabla_{X}({b})\big)\\
&\hspace{3.5cm}-\nabla_{s^{-1}\overline{\nabla}_{b}(sX)}({a})+\nabla_{s^{-1}\overline{\nabla}_{a}(sX)}({b})\nonumber
\end{align}
for every $a,b\in\Gamma(E)$ and for every  $X\in \mathfrak{X}(M)$.
Since the right-hand side has degree $|a|+|b|$, while on the left-hand the degrees of the arguments sum up to $|a|+|b|+1$ (because $|sX|=1$), the basic 2-curvature $\overline{R}_{[2]}$ is a 2-form taking values in $\mathrm{Hom}_{-1}(sTM,E)$, i.e. it has total degree~$+1$: $\overline{R}_{[2]} \in\Omega^2(E,\mathrm{Hom}(sTM,E))_{+1}$.
Now recall that a Lie $\infty$-algebroid comes equipped with a set of graded skew-symmetric brackets $(l_k)_{k\geq3}$. Hence, we can define a whole set of  higher basic curvatures $\overline{R}_{[k]}$ measuring the compatibility of $\nabla$ with the $k$-brackets. More precisely, for every $k\geq3$, and for every $a_1,\ldots,a_k\in\Gamma(E)$ and $X\in \mathfrak{X}(M)$, we define the $k$-th basic curvature $\overline{R}_{[k]}$~as:
\begin{equation}
\overline{R}_{[k]}(a_1,\ldots,a_k)(sX)=\nabla_{X}\big(l_k(a_1,\ldots,a_k)\big)-\sum_{l=1}^kl_k\big(a_1,\ldots,\nabla_{X}(a_l),\ldots,a_k\big)
\end{equation}
Notice that here, contrary to the basic 2-curvature $\overline{R}_{[2]}$, we don't need any correcting term because the higher brackets are naturally $\mathcal{C}^\infty$-linear. Comparing the degree of the arguments on left-hand side and on the right-hand side, we deduce that $\overline{R}_{[k]}$ if a $k$-form on $E$  taking values in $\mathrm{Hom}_{1-k}(sTM,E)$, i.e. it has total degree $+1$ as $\overline{R}_{[2]}$: $\overline{R}_{[k]} \in\Omega^k(E,\mathrm{Hom}(sTM,E))_{+1}$.
%\textbf{n'y a-t-il pas une 1 courbure associee tq $\omega^{(1)}=l_2+\overline{R}_{[1]}$ ?? il semble que $\overline{R}_{[1]}(a)(X)=[\nabla_X,\rho](a)$ car c'est un operateur a valeur dans end K de degre 0. Mais alors $\omega^{(1)}=l_2+\overline{R}_{[1]}$ ?? Ca n'a pas l'air de correspondre a la connexion form associee a  la definition des basic curvatures. Quelles sont les equation satisfaites par la basic curvature ? (notamment vis a vis de l'ancre)}

The adjoint representation of $E$ (with respect to the connection $\nabla$) will be defined on the graded vector bundle $K=E\oplus sTM$ as follows: first, we set $D^{(0)}=\partial$, $D^{(1)}=d^{\overline{\nabla}}$ and %, starting from $D^{(0)}+D^{(1)}$, we build by induction the term of order $k$: $D^{(k)}\colon\Omega^\bullet\big(E,K)_\blacklozenge\to\Omega^{\bullet+k}\big(E,K)_{\blacklozenge+1}$, in order to obtain a total degree $+1$ map $D=\sum_{i\geq0} D^{(i)}$ satisfying the homological condition $D^2=0$. 
%One can then check that by setting,
 for every $k\ge2$ we define the connection $k$-form $\omega^{(k)}$ as:
\begin{equation}\label{definitionconnection}
\omega^{(k)}(a_1,\ldots,a_k)= l_{k+1}(a_1,\ldots,a_k,\,.\,)+\overline{R}_{[k]}(a_1,\ldots,a_k)(\,.\,)
\end{equation}
By definition, $\omega^{(k)}(a_1,\ldots,a_k)$ is a graded endomorphism of $K=E\oplus sTM$.
The first term $l_{k+1}(a_1,\ldots,a_k,\,.\,)$ (resp. $\overline{R}_{[k]}(a_1,\ldots,a_k)$) defines the action of $\omega^{(k)}(a_1,\ldots,a_k)$ on $E$ (resp. $sTM$).
One verifies that this family $\big(\omega^{(k)}\big)_{k\geq2}$ of $\mathrm{End}(K)$-valued $k$-forms of total degree $+1$ induces the desired family of operators $D^{(k)}$ via the correspondence \eqref{eqokkk}. This adjoint representation straightforwardly generalizes that of Lie algebroids \cite{abadRepresentationsHomotopyLie2011}, and coincides with the one defined with respect to the graded symmetric convention for Lie $\infty$-algebroids appearing in Proposition~3.27 of \cite{caseiroModularClassLie2022}. The adjoint representation of $E$  presented above depends on the choice of $TM$-connection $\nabla$ so it is not unique. However, as in Section 3.2 of \cite{abadRepresentationsHomotopyLie2011}, representations of the Lie $\infty$-algebroid $E$ forms a category and the isomorphism classes of adjoint representations of $E$ (with respect to homotopy equivalences) is the `true' adjoint representation of $E$ (see also Section~5 of \cite{jotzleanModulesRepresentationsHomotopy2023}).\end{example}

\subsection{Singular foliations and their foliated cohomology}\label{singlg}

%\subsection{Universal Lie $\infty$-algebroids and singular foliations}

%Since one cannot use straightforwardly the notion of conormal bundle in the singular foliation context, we will rather adopt the strategy of using an adapted replacement of any singular foliation so that its modular class is that of the original singular foliation. 
Although most of the material of this section has been introduced in \cite{laurent-gengouxUniversalLieInfinityAlgebroid2020}, we provide it again for consistency, and develop it so that it fits our current objective: to define the modular class of a singular foliation.
Let $M$ be a smooth manifold, $\mathcal{C}^\infty\colon U\to\mathcal{C}^\infty(U)$ be the sheaf of smooth functions on $M$, and $\mathfrak{X}\colon U\to\mathfrak{X}(U)$ be the sheaf of smooth vector fields on $M$.

\begin{definition}
\label{def:sing_fol}
A \emph{singular foliation} is a $\mathcal{C}^\infty$-subsheaf ${\mathcal F}\colon {U} \mapsto  {\mathcal F}(U)$ of the sheaf of vector fields ${\mathfrak X}$, which is:
\begin{enumerate}
\item closed under the Lie bracket of vector fields,
\item locally finitely generated as a $\mathcal{C}^\infty$-module.
\end{enumerate}
\end{definition}

A singular foliation induces a \emph{generalized distribution}, i.e. the (smooth) assignment $D\colon x\mapsto D_x\subset T_xM$, for every $x\in M$, of a subspace  of the tangent space of $M$ at $x$. The rank of this distribution may not be constant over $M$. %but since the map $x\longrightarrow \mathrm{dim}(D_x)$ is a lower semi-continuous function, it is locally constant on a dense subset of $M$. %when it is, the hence defined subbundle $F\subset TM$ is called a \emph{regular foliation}. 
%For disambiguation, we will designate regular foliations by their induced distribution $F$ rather than by the sheaf of vector fields that generate it.
We define a \emph{leaf} as an immersed -- in fact, weakly embedded -- submanifold $L$ of $M$ such that $T_xL=D_x$ for every $x\in L$. By Hermann's theorem, finitely generated involutive generalized distributions are integrable \cite{hermannDifferentialGeometryFoliations1962}, i.e. leaves of a singular foliation in the sense of Definition \ref{def:sing_fol} form a partition of $M$ having the expected diffeological properties generalizing that of regular foliations \cite{miyamotoSingularFoliationsDiffeology2023}. This represents one of the generalizations of Frobenius's integrability theorem to singular foliations \cite{lavauShortGuideIntegration2018}.  

A \emph{regular leaf} is a leaf $L$ for which every leaf in the vicinity has the same dimension as $L$, i.e.  it is such that for every $x\in L$, there exists a small neighborhood of $x$ on which the map $x\longrightarrow \mathrm{dim}(D_x)$ is locally constant. The dimension of the regular leaves may not be constant over $M$. %such that the leaves intersecting $U$ have dimension $\mathrm{dim}(N)$. 
The union of all the regular leaves forms however an open dense subset in $M$, and the rank of the distribution $D$ is constant on each connected component of this subset. On the contrary, a \emph{singular leaf} is a leaf which is not a regular leaf. In particular, for any point $y$ in a singular leaf $L$, any neighborhood of $y$ admits an non-empty intersection with a regular leaf -- which has dimension stricly higher than $\mathrm{dim}(L)$, because the function $x\to\mathrm{dim}(D_x)$ taking values in the integers is lower semi-continuous.
% In this setup, an integrable regular distribution $F\subset TM$ -- corresponding to a regular foliation -- can be alternatively defined as the singular foliation  $\mathcal{F}=\Gamma(F)$ consisting of all the smooth sections of $F$. 
Most examples of singular foliations in the sense of Definition \ref{def:sing_fol} come from 
%\begin{proposition} \label{prop:fromNQtoSingularFoliations2}
 \emph{almost Lie algebroids}:

\begin{definition}
An \emph{almost Lie algebroid} is a vector bundle $A$, together with a bilinear skew-symmetric bracket $l_2$ defined on the sections of $A$ and an anchor map $\rho\colon A\longrightarrow TM$ satisfying Equations \eqref{robinson} and \eqref{jpp}. The image by the anchor map of  the sections of $A$  is a singular foliation\footnote{It is finitely generated as $A$ is a vector bundle, and it is involutive by Equation \eqref{jpp}, see \cite{laurent-gengouxUniversalLieInfinityAlgebroid2020} for details.} $\mathcal{F}_A=\rho\big(\Gamma(A)\big)$ called the \emph{characteristic foliation of $A$}.
Given a singular foliation $\mathcal{F}$, we say that an almost Lie algebroid $A$ \emph{covers $\mathcal{F}$} if $\mathcal{F}_A=\mathcal{F}$. %, for $l_2=[\,.\,,.\,]_A$.
\end{definition}

\begin{remarque}
 Contrary to Lie algebroids, the  bracket $l_2$ needs not satisfy the Jacobi identity. A typical example of an almost Lie algebroid is the zeroth component $E_0$ of a Lie $\infty$-algebroid $E=\bigoplus_{i\geq0} E_{-i}$.
 \end{remarque}

  % Lie $\infty$-algebroid $(E,l_k,\rho)$, as $\mathcal{F}_E=\rho\big(\Gamma(E_{0})\big)$.
  
%\noindent %This is a well known result when $E$ is a Lie algebroid. 
%In particular, many regular foliations emerge as the image of the anchor map of regular Lie algebroids. In this case, the anchor map has constant rank and defines a short exact sequence of vector bundles:
%\begin{center}
%\begin{tikzcd}[column sep=0.7cm,row sep=0.4cm]
%0\ar[r]&\mathrm{Ker}(\rho)\ar[r]&A\ar[r]&\rho(A)\ar[r]&0
%&\ldots\ar[r,"\delta_{-(k-2)}"]&V^*_{-k+1}\ar[r,"\delta_{-(k-1)}"]&U_{-k}^*\ar[r,"\delta_{-k}"]&\ldots
%\end{tikzcd}
%\end{center}
%and the last vector bundle on the right hand side is a regular foliation.
% Finding such a short exact sequence (at the sheaves level) is in general not possible for singular foliations. However, i
Now, it may happen that a singular foliation $\mathcal{F}$ admits a \emph{geometric resolution} \cite{laurent-gengouxUniversalLieInfinityAlgebroid2020}: %\begin{remarque}
 it is a family of non-positively graded vector bundles $E=(E_{-i})_{i\geq0}$, together with a vector bundle morphism  $d=E_\bullet\longrightarrow E_{\bullet+1}$ acting as a differential, and an anchor map $\rho:E_0\to TM$, such that for every open set $U\subset M$, the following chain complex:
\begin{center}
\begin{tikzcd}[column sep=0.7cm,row sep=0.4cm]
\ldots\ar[r,"d"]&\Gamma_U(E_{-2})\ar[r,"d"]&\Gamma_U(E_{-1})\ar[r,"d"]&\Gamma_U(E_{0})\ar[r,"\rho"]&\mathcal{F}(U)\ar[r]&0
\end{tikzcd}
\end{center}
is an exact sequence of $\mathcal{C}^\infty(U)$-modules. Not every singular foliation admits a geometric foliation (see Example 3.38 in \cite{laurent-gengouxUniversalLieInfinityAlgebroid2020}). We say that a singular foliation of $\mathcal{F}$ is \emph{solvable} when it admits a geometric resolution of finite length. Such resolutions are a natural generalization of resolutions of integrable regular distributions of the form  \eqref{cpxe} to the setup of singular foliations.
%\end{remarque}
This justifies the introduction of the following notion: %In particular the question arose if any singular foliation could be obtained as the image, through the anchor map, of a Lie algebroid. This was shown not to be true globally \cite{AndrouZamb}, but it is still conjectural locally. Notice that if it is not  from a  Lie algebroid it might be from Lie $\infty$-algebroid.
\begin{definition}\label{def:universal}
Let $\mathcal{F}$ be a singular foliation on a smooth manifold $M$, and let $(E,l_k,\rho)$ be a Lie $\infty$-algebroid over $M$ covering $\mathcal{F}$. We say that $E$ is a \emph{universal Lie $\infty$-algebroid of $\mathcal{F}$} if its linear part is a geometric resolution of $\mathcal{F}$. The set of all universal Lie $\infty$-algebroids of $\mathcal{F}$ is denoted $\mathfrak{E}_{\mathcal{F}}$.
%\begin{center}
%\begin{tikzcd}[column sep=0.7cm,row sep=0.4cm]
%\ldots\ar[r,"l_1"]&\Gamma(E_{-2})\ar[r,"l_1"]&\Gamma(E_{-1})\ar[r,"l_1"]&\Gamma(E_{0})\ar[r,"\rho"]&\mathcal{F}\ar[r]&0
%\end{tikzcd}
%\end{center}
%That is to say, it is a resolution on every open set $U\subset M$.
\end{definition}

%The resolution plays a role in defining the Lie $\infty$-algebroid structure. Indeed, in \cite{lavaulaurentgengoux}, it is shown that a singular foliation admitting a resolution induces a universal Lie $\infty$-algebroid structure on this resolution by lifting the Lie bracket of vector fields. 
%\begin{remarque}
%For locally real-analytic foliations, universal Lie $\infty$-algebroids are local generalizations of the notion of almost regular Lie algebroids: for every $x\in M$ there exists a universal Lie $\infty$-algebroid of $\mathcal{F}$ of length at most $n+1$ defined over some neighborhood of $x$.
%\end{remarque}

The problem of existence and unicity of such universal Lie $\infty$-algebroids has been answered in \cite{laurent-gengouxUniversalLieInfinityAlgebroid2020}. For the sake of completeness, we recall here the most useful and central results:

\begin{theoreme} \textbf{\emph{Existence.}}
\label{theo:existe} Let $\mathcal{F}$ be a singular foliation admitting a geometric resolution $(E,d)$. Then there exists a Lie $\infty$-algebroid structure  on $E$ such that:
\begin{enumerate}
\item its linear part is the afore-mentioned geometric resolution, and 
\item it is a universal Lie $\infty$-algebroid of $\mathcal{F}$.
\end{enumerate}
\end{theoreme}

Theorem \ref{theo:existe} is a result about existence of universal Lie $\infty$-algebroids.
As a particular case of frequent interest, we have singular foliations which are locally real analytic. A singular foliation $\mathcal{F}$ is said real analytic if there exists an open cover of $M$ such that the generators of $\mathcal{F}$ only involve real analytic functions of the local coordinates associated to each chart. Then, for every $x\in M$ there exists a universal Lie $\infty$-algebroid of $\mathcal{F}$ of length at most $n+1$ defined over some neighborhood of $x$.  The following theorem is a result about the universal property of universal Lie $\infty$-algebroids:

  %\textbf{this result is useful to define a local modular class.}

%Then the following statement has been shown in \cite{lavaulaurentgengoux}:
\begin{theoreme} \textbf{\emph{Unicity.}} \label{theo:onlyOne}
Let $\mathcal{F}$ be a singular foliation on a smooth manifold $M$ and let $E$ be  a universal Lie $\infty$-algebroid of $\mathcal{F}$. 
Then for every Lie $\infty$-algebroid $E'$ covering $ \mathcal{F}$,
there exists a Lie $\infty$-algebroid morphism over $M$ from $E'$ to $E$ and any two such Lie $\infty$-algebroid morphisms are homotopic.
\end{theoreme}

Hence, universal Lie $\infty$-algebroids of a singular foliation are in some sense unique: any two universal Lie $\infty$-algebroids $E$ and $E'$ of $\mathcal{F}$ are homotopically equivalent, in the sense of Definition~\ref{def:homtequiv}. This universality property found another formulation in category theory \cite{fregierHomotopyTheorySingular2019}: in the semi-model category of $L_\infty$-algebroids over the smooth manifold $M$ (these are mild generalizations of Lie $\infty$-algebroids), a universal Lie $\infty$-algebroid of the singular foliation $\mathcal{F}$  is a \emph{replacement of $\mathcal{F}$}. See also
\cite{laurent-gengouxLieRinehartAlgebrasAcyclic2022} for another discussion about equivalence of categories between Lie-Rinehart algebras over a commutative algebra and homotopy equivalence classes of non-positively graded Lie $\infty$-algebroid over their resolutions.
%The choice of this Lie $\infty$-algebroid structure is not unique, but one can show that any two such choices are homotopy equivalent \cite{laurent-gengouxUniversalLieInfinityAlgebroid2020}.

The second statement of Theorem \ref{theo:onlyOne} provides another important information about equivalences of universal Lie $\infty$-algebroids: that any two homotopy equivalences between universal Lie $\infty$-algebroids are homotopic.
This implies that the Lie $\infty$-algebroid cohomologies of two universal Lie $\infty$-algebroids of $\mathcal{F}$ are \emph{canonically isomorphic} as graded commutative algebras~\cite{laurent-gengouxUniversalLieInfinityAlgebroid2020}. For two such Lie $\infty$-algebroids $E$ and $E'$, we note $\widetilde{\Phi}_{E,E'}:H^\blacktriangle(E')\to H^\blacktriangle(E)$ this canonical isomorphism.  Hence, we can define an equivalence relation $\sim$ on the space $\bigcup_{E \in\mathfrak{E}_{\mathcal{F}}}H(E)$ by:
\begin{equation*}
x\sim y \hspace{1cm}\text{if and only if}\hspace{1cm}x=\widetilde{\Phi}_{E,E'}(y)
\end{equation*}
where $x\in H(E)$ and $y\in H(E')$.
This equivalence relation provides the singular foliation $\mathcal{F}$ with a cohomology:

\begin{definition} \label{defunivseral}
Let ${\mathcal F}$ be a  singular foliation admitting a geometric resolution. 
We call  \emph{universal foliated cohomology of $ {\mathcal F}$} and denote by  $H_{\mathfrak{U}}(\mathcal F)$
the set of equivalence classes of $\bigcup_{E \in\mathfrak{E}_\mathcal{F}}H(E)$ with respect to the equivalence relation $\sim$. % More precisely, for every Lie $\infty$-algebroid $E$ of $\mathcal{F}$ and every $k\geq0$, an element $x\in H^k(E)$ is a representent of $[x]\in H^k_{\mathfrak{U}}(\mathcal{F})$.
\end{definition}

This definition respects the graduation defined by the height, i.e. for every $h\geq0$, the space $H_{\mathfrak{U}}^h(\mathcal F)$ is the $\mathbb{R}$-vector space freely generated by the set of equivalence classes of $\bigcup_{E \in\mathfrak{E}_\mathcal{F}}H^h(E)$. This is well defined because for any universal Lie $\infty$-algebroids $E$ and $E'$, the canonical isomorphism $\widetilde{\Phi}_{E,E'}:H^\blacktriangle(E')\to H^\blacktriangle(E)$  preserves the height. For every $h\geq0$ and every $E\in\mathfrak{E}_\mathcal{F}$, there is a canonical isomorphism between $H_{\mathfrak{U}}^h(\mathcal F)$ and $H^h(E)$ which corresponds to picking up the unique representent $x^E \in H^h(E)$ of the equivalence class $[x]\in H_{\mathfrak{U}}^h(\mathcal F)$. Hence, $H_{\mathfrak{U}}^h(\mathcal F)$ is finite dimensional as it has the same dimension as $H^h(E)$, which  does not depend on the choice of Lie $\infty$-algebroid $E\in\mathfrak{E}_{\mathcal{F}}$.
The universal foliated cohomology $H_{\mathfrak{U}}(\mathcal{F})$  inherits a canonical graded commutative algebra structure from the one on $H(E)$. 

The modular class of a  regular foliation is an element of the first group of the foliated de Rham cohomology of the corresponding involutive regular distribution. We would like to reproduce such a statement for singular foliations, but the presence of singularities compels us to make a detour via the universal foliated cohomology. 
Let us call \emph{forms on ${\mathcal F}$} and denote by $ \Omega^\bullet({\mathcal F}) $ or simply $\Omega({\mathcal F}) $ the space of ${\mathcal{C}^\infty}$-multilinear skew-symmetric assignments from
$ {\mathcal F}$~to~$\mathcal{C}^\infty$:
 \begin{equation*}\Omega^\bullet({\mathcal F}) = \mathrm{Hom}_{\mathcal{C}^\infty}\big(\wedge^\bullet{\mathcal F},{\mathcal{C}^\infty}\big)= 
 \bigoplus_{k \geq 0} \mathrm{Hom}_{\mathcal{C}^\infty}\big(\wedge^k {\mathcal F},{\mathcal{C}^\infty}\big).\end{equation*}
 Note that $0$-forms on $\mathcal{F}$ are just functions on $M$. %Also, a $k$-form $\alpha$ on $ {\mathcal F}$ induces a $k$-form $\alpha_L$ on each regular leaf $L$, but maybe not on singular ones.
One can define a foliated de Rham differential $d_{\mathrm{dR}}$ % \emph{foliated de Rham operator} 
on $ \Omega({\mathcal F})$ from its action on any $k$-form $\alpha\in\Omega^k(\mathcal{F})$:
\begin{align}
{d}_{\mathrm{dR}} (\alpha)(X_0, \dots , X_k) &=\sum_{i=0}^k (-1)^{i} X_i\big[ \alpha (X_{0},\ldots,\widehat{X_i},\ldots,X_k)\big] \\&\hspace{1cm}+ \sum_{0\leq i < j\leq k}(-1)^{i+j}\alpha\Big([X_i,X_j],X_0,\ldots, \widehat{X_{i}},\ldots,\widehat{X_{j}},\ldots,X_k\Big) ,\nonumber
\end{align}
 with the understanding that $\widehat{X_i}$ means that the term $X_i$ is omitted.
 We call the cohomology of this operator the \emph{foliated  de Rham cohomology of $ {\mathcal F}$} and denote it by $H^\blacktriangle_{\mathrm{dR}}({\mathcal F})$. Obviously, the 0-th group of foliated de Rham cohomology corresponds to the smooth functions that are $\mathcal{F}$-invariant. When the singular foliation is regular and integrates an involutive regular distribution $F$, the foliated cohomology coincides with the foliated de Rham cohomology of $F$ or, equivalently, the  Lie algebroid cohomology of the foliation Lie algebroid $F$.

  Let $\mathcal{F}$ be a singular foliation on $M$, and let $(E,l_k,\rho)$ be a Lie $\infty$-algebroid covering $\mathcal{F}$ (i.e.~not necessarily a resolution). 
 The anchor map induces a  map $\rho^*$ from $\Omega(\mathcal{F})$ to $ \Omega(E)$ 
 given by associating to each $ \alpha \in \Omega^k(\mathcal{F})$
 the element $\rho^* {\alpha} \in  \Gamma\big(\wedge^k E_{0}^*\big) = \Omega^k_k(E) $ defined by:
  \begin{equation}\label{eq:pullbackByAnchor} 
 \rho^* \alpha (a_1, \dots,a_k) = \alpha\big(\rho(a_1), \dots,\rho(a_k)\big) 
  \end{equation}
  for every $a_1 , \dots,a_k \in \Gamma(E_{0})$.
One can check that $\alpha \mapsto \rho^* ({\alpha})$ is an injective chain map and a graded commutative algebra morphism (Lemma 4.5 in \cite{laurent-gengouxUniversalLieInfinityAlgebroid2020}), inducing therefore an algebra morphism, still denoted $\rho^*$,
  from $H_{\mathrm{dR}}(\mathcal{F})$ to 
$ H(E)$. The grading in $H_{\mathrm{dR}}(\mathcal{F})$ is the form-degree, whereas in $H(E)$ it is the height, and the map $\rho^\ast$ preserves this grading. 
The link between the foliated de Rham cohomology and the universal foliated cohomology is induced by this map:
\begin{proposition}\label{hooke}
Let $\mathcal{F}$ be a singular foliation on a smooth manifold admitting a geometric resolution. Then, there exists a canonical injective morphism of  graded commutative algebras:
\begin{equation*}
\rho^\ast_{\mathcal{F}}:H_{\mathrm{dR}}^\blacktriangle ({\mathcal F})\longrightarrow H_{\mathfrak{U}}^\blacktriangle ({\mathcal F})
\end{equation*}
that is an isomorphism in degree 1.
\end{proposition}

\begin{proof}
Let $(E,l_k,\rho)$ and $(E',l'_k,\rho')$ be two universal Lie $\infty$-algebroids of the singular foliation~${\mathcal F}$. Then by Theorem \ref{theo:onlyOne}, there exists a Lie $\infty$-morphism over $M$ from $E$ to $E'$, denoted $\Phi_{E,E'}$. Equation \eqref{ultra}, applied to functions on $M$, implies that $\rho^*=\Phi_{E,E'}\circ\rho'^*$. Since this is true for every such Lie $\infty$-morphism from $E$ to $E'$, and since they are all homotopic to one another, the canonical isomorphism $\widetilde{\Phi}_{E,E'}:H^\blacktriangle(E')\to H^\blacktriangle(E)$ that they induce at the cohomology level makes the following diagram commutative: %then we first have to show that the maps $\rho^\ast$ and $\rho'^\ast$ induce a canonical map in cohomology.
%
%\begin{proposition}\label{proposition2}
%Let $\mathcal{F}$ be a singular foliation on $M$ admitting a resolution.
%The algebra morphism $\rho^*$ from the foliated de Rham cohomology of $\mathcal{F}$ to 
%the universal foliated cohomology of $\mathcal{F}$ given by Equation (\ref{eq:pullbackByAnchor}) is canonical.
%\end{proposition} 
\begin{center}
\begin{tikzcd}[column sep=1.3cm,row sep=1.4cm]
  & H_{\mathrm{dR}}^\blacktriangle ({\mathcal F}) \ar[dl,"\rho'^*" above left] \ar[dr, "\rho^*" above right] & \\
 H^\blacktriangle (E')  \ar[rr, rightarrow ,"\widetilde{\Phi}_{E,E'}" ] &   & H^\blacktriangle(E)
\end{tikzcd} 
\end{center}
 From this diagram, we deduce that for any $\alpha\in H^\blacktriangle_{\mathrm{dR}}(\mathcal{F})$, $\rho^\ast(\alpha)$ and $\rho'^\ast(\alpha)$ induce the same element in the universal foliated cohomology $H_{\mathfrak{U}}(\mathcal{F})$.  We call this element $\rho_{\mathcal{F}}^\ast(\alpha)$. This assignment defines uniquely the height-preserving map $\rho^\ast_{\mathcal{F}}:H_{\mathrm{dR}}^\blacktriangle ({\mathcal F})\to H_{\mathfrak{U}}^\blacktriangle ({\mathcal F})$ that we are looking for.
 %Hence we call $\rho_{\mathcal{F}}$ the canonical induced morphism from the foliated de Rham cohomology to the universal foliated cohomology:
%\begin{equation*}
%\rho^\ast_{\mathcal{F}}:H_{\mathrm{dR}}^\blacktriangle ({\mathcal F})\to H_{\mathfrak{U}}^\blacktriangle ({\mathcal F})
%\end{equation*}
% Let $E$ be a universal Lie $\infty$-algebroid of $\mathcal{F}$, with anchor map $\rho$. 
Now, the map $\rho^*\colon\Omega^k(\mathcal{F})\to \Omega^k_k(E)$ is injective since $\rho$ is surjective. This result being true for every universal Lie $\infty$-algebroid of $\mathcal{F}$, we conclude that the maps $\rho_{\mathcal{F}}^\ast$ is injective.

Let us now show that it is an isomorphism at level 1. %The dual of the resolution $(E,l_1,\rho)$ is a resolution:
%\begin{center}
%\begin{tikzcd}[column sep=0.7cm,row sep=0.4cm] 
%\ldots&\ar[l,"d_E^{(0)}"]\Gamma\big((E_{-2})^*\big)&\ar[l,"d^{(0)}_E"]\Gamma\big((E_{-1})^*\big)&\ar[l,"d_E^{(0)}"]\Gamma(E_{0}^*)&\ar[l,"\rho^*"]\Omega^1(\mathcal{F})&\ar[l]0
%\end{tikzcd}
%\end{center}
The fact that $\rho^*\colon\Omega^1(\mathcal{F})\to \Omega^1_1(E)$ is injective implies that it is injective as well at the cohomology level. Let us now show that it is surjective on cohomology.
Let $u\in\Omega^1_1(E)=\Gamma(E_0^*)$ be a $d_{E}$-cocycle, then in particular it means that $d^{(0)}_E(u)=0$, that is: $u\big|_{l_1(\Gamma(E_{-1}))}=0$. Since we have the equality $\mathrm{Im}\big(l_1(\Gamma(E_{-1}))\big)=\mathrm{Ker}(\rho)$, then $u$ vanishes on the kernel of the anchor map. Then it comes from a 1-form on $\mathcal{F}$, i.e. there exists some $\alpha\in \Omega^1(\mathcal{F})$ such that:
\begin{equation}\label{alphau}
u=\rho^*(\alpha)
\end{equation}
This element is unique since $\rho^*$ is injective on $\Omega^1(\mathcal{F})$.  
Since $u$ is a $d_E$-cocycle, we also have that $d_E^{(1)}u=0$, that is: $u\big([a,b]\big)=\rho(a)[u(b)]-\rho(b)[u(a)]$ which, by virtue of Equation \eqref{alphau}, can also be written as $\alpha\big([\rho(a),\rho(b)]\big)=\rho(a)[\alpha(\rho(b))]-\rho(b)[\alpha(\rho(a))]$. This is precisely the condition that $\alpha$ is a cocycle in $\Omega^1(\mathcal{F})$. Then, since $\rho^*$ is a chain map, and since it is injective, we deduce that the closed 1-forms on $\mathcal{F}$ are in one-to-one correspondence with the closed 1-forms on $E$ of degree 0. For the same reasons, exact 1-forms are in one-to-one correspondence as well.  This means that the map $\rho^*\colon\Omega^1(\mathcal{F})\to \Omega^1_1(E)$ is bijective at the cohomology level. This implies that the canonically induced morphism $\rho^\ast_{\mathcal{F}}$ is an isomorphism in degree 1.  \end{proof}

\begin{remarque}
Actually, $\rho_{\mathcal{F}}$ is also bijective at level 0, and it is even the identity map. This is rather obvious since in both case the zero-th cohomology group consists of the functions that are $\mathcal{F}$-invariant.
\end{remarque}

%We now would like to compare the modular class of a singular foliation with the modular class of the Lie algebroids covering it. Proposition \ref{hooke} induces the following result:

The newly defined map $\rho^*_\mathcal{F}$ actually has the following universal property:
\begin{proposition}
Let $\mathcal{F}$ be a singular foliation admitting a geometric resolution on a smooth manifold~$M$, and let $E$ be a Lie $\infty$-algebroid covering $\mathcal{F}$, with anchor map $\rho$. Then the map $\rho^*:H_{\mathrm{dR}}(\mathcal{F})\to H(E)$ defined by Equation \eqref{eq:pullbackByAnchor} factors through $H_{\mathfrak{U}}(\mathcal{F})$:
\begin{center}
\begin{tikzcd}[column sep=2cm,row sep=1.6cm]
  & H_{\mathfrak{U}} ({\mathcal F}) \ar[d]  \\
 H_{\mathrm{dR}} (\mathcal{F}) \ar[ur,"\rho_{\mathcal{F}}^*" above left]  \ar[r,"\rho^\ast" above]    & H(E)
\end{tikzcd} 
\end{center}
\end{proposition}
\begin{proof}
Let $E'$ (resp. $E''$) be a universal Lie $\infty$-algebroid of $\mathcal{F}$, and let $\Phi_{E,E'}\colon\Omega(E')\to\Omega(E)$ (resp. $\Phi_{E,E''}\colon\Omega(E'')\to\Omega(E)$) be an Lie $\infty$-morphism over $M$ between $E$ and $E'$ (resp. $E''$), whose existence is guaranteed by Theorem \ref{theo:onlyOne}. The same theorem insures that any two such morphisms induce a unique graded commutative algebra morphism $\widetilde{\Phi}_{E,E'}:H(E')\to H(E)$ (resp. $\widetilde{\Phi}_{E,E''}:H(E'')\to H(E)$) at the cohomology level. They are such that they make the following diagram commutative:
\begin{center}
\begin{tikzcd}[column sep=2cm,row sep=1.6cm]
  & H (E'') \ar[d, "\widetilde{\Phi}_{E,E''}" right]  \ar[dd, bend left=60, "\widetilde{\Phi}_{E',E''}", right] \\
 H_{\mathrm{dR}} (\mathcal{F}) \ar[ur,"{\rho''}^*" above left] \ar[dr,"{\rho'}^*" below left]  \ar[r,"\rho^\ast" above]    & H(E)\\
   & H (E') \ar[u, "\widetilde{\Phi}_{E,E'}" right]  \\
\end{tikzcd} 
\end{center}
%The canonical isomorphism $\widetilde{\Phi}_{E',E''}\colon H(E'')\to H(E')$ allows to identify $H(E')$ and $H(E'')$.
Since  $\widetilde{\Phi}_{E,E''}=\widetilde{\Phi}_{E,E'}\circ\widetilde{\Phi}_{E',E''}$, when passing to the universal foliated cohomology by identifying $H(E')$ and $H(E'')$ through $\widetilde{\Phi}_{E',E''}$, the graded commutative algebra morphisms $\widetilde{\Phi}_{E,E'}$ and $\widetilde{\Phi}_{E,E''}$ induce a canonical graded commutative algebra morphism $\widetilde{\Phi}:H_{\mathfrak{U}}(\mathcal{F})\to H(E)$. %the canonical morphisms $\widetilde{\Phi}_{E,E'}$ and $\widetilde{\Phi}_{E,E''}$ commute with $\widetilde{\Phi}_{E',E''}$, i.e. they satisfy 
By commutativity of the diagram and by definition of the map $\rho_{\mathcal{F}}$, this morphism satisfies the identity $\rho^*=\widetilde{\Phi}\circ\rho_{\mathcal{F}}^*$.
\end{proof}

\section{A perspective on modular classes}\label{sec3}

\subsection{The modular class of a Lie $\infty$-algebroid}\label{sec:char}

In this section we adapt the definition of modular classes of Lie algebroids \cite{evensTransverseMeasuresModular1999} to the Lie $\infty$-algebroid context. Modular classes of Lie algebroids  have been defined  as the natural counterpart of the notion of modular vector fields of Poisson manifolds \cite{weinsteinModularAutomorphismGroup1997}. It is conjectured that it  measures the obstruction of the existence of a volume form on the differentiable stack associated to the groupoid integrating the Lie algebroid \cite{weinsteinVolumeDifferentiableStack2009, crainicMeasuresDifferentiableStacks2020}.
The generalization of the notion of modular class to the Lie $\infty$-algebroid context has already been investigated in  \cite{caseiroModularClassLie2022} under the graded symmetric convention for $L_\infty$ brackets, and we merely summarize some of their results here, with the graded skew-symmetric convention.

 Let us first apply the notion of representations up to homotopy to line bundles. Let $E$ be a Lie $\infty$-algebroid over $M$ and let $L$ be a line bundle over $M$ and a representation of $E$, that we assume to be concentrated in degree 0. Then, for degree reasons, the only sub-bundle of $E$ that acts on $L$ is the almost Lie algebroid $E_0$. %, so we need to introduce the adapted terminology to understand how it acts on $L$. 
 The notion of Lie algebroid connection straightforwardly extends to almost Lie algebroids when an additional condition is added: 
 
 \begin{definition}
Given a (possibly graded) vector bundle $K\to M$, and an almost Lie algebroid $A$ over $M$, a  (degree 0) differential operator $\nabla\colon\Gamma(A)\times \Gamma(K)\to \Gamma(K)$ satisfying axioms  \eqref{eqrep1} and \eqref{eqrep2} is called an \emph{$A$-connection on $K$}. 
Such a connection is \emph{flat} if the corresponding curvature vanishes:
\begin{equation}\label{eqrep3}
R_\nabla(a,b)=\big[\nabla_a,\nabla_b\big]-\nabla_{l_2(a,b)}=0\hspace{2cm}\text{for every $a,b\in\Gamma(A)$}
\end{equation}
%and if, additionally, the following identity holds:
%\begin{equation}\label{equaflat23}
%\nabla_{l_2(a,l_2(b,c))+l_2(b,l_2(c,a))+l_2(c,l_2(a,b))}=0 \hspace{1.5cm}\text{for every $a,b,c\in\Gamma(A)$}
%\end{equation}
Graded vector bundles $K\to M$ with flat $A$-connections are called \emph{representations of $A$} or \emph{$A$-modules}\footnote{The latter denomination may not be standard -- see e.g. \cite{mehtaLieAlgebroidModules2014} -- but we introduce and use it for its convenience.}. We say that a representation $K$ is \emph{trivial} when there exists a global frame of $K$ on which  the action of $A$ is zero.
\end{definition}

The covariant derivative associated to the connection $\nabla$ satisfies the following behavior:
\begin{align*}
(d^\nabla)^2f(a,b)&=R_\nabla(a,b)(f)\\
(d^\nabla)^2\alpha(a,b,c)&=R_\nabla(a,b)(\alpha(c))\,+\circlearrowleft+\,\alpha\big(l_2(a,l_2(b,c))+l_2(b,l_2(c,a))+l_2(c,l_2(a,b))\big)
\end{align*}
for $f\in\Gamma(K)$ and $\alpha\in\Omega^1(E,K)$. 
The last equation is expected because the 2-bracket $l_2$ does not satisfy the Jacobi identity and so there is no reason that $(d^\nabla)^2=R_\nabla$.
This curvature then satisfies the following Bianchi-like identity:
 \begin{equation}\label{bianchilike}
d^\nabla R_\nabla(a,b,c) = \big[\nabla_a,[\nabla_b,\nabla_c]\big]\,+\circlearrowleft -\nabla_{l_2(a,l_2(b,c))+l_2(b,l_2(c,a))+l_2(c,l_2(a,b))}
 \end{equation}
%The term on the left hand side plays the role of the Bianchi identity for the almost Lie algebroid connection $\nabla$. 
 The first term on the right-hand side (including the circular permutations) is necessarily zero as a Jacobiator of operators while the second term vanishes if and only if $d^\nabla R_\nabla=0$, which is the case at least when $\nabla$ is flat.

%\begin{definition}
%Let $A$ be an almost Lie algebroid.
%Vector bundles with flat $A$-connections are called \emph{representations $A$} or \emph{$A$-modules}. We say that a representation $K$ is \emph{trivial} when there exists a global frame of $K$ on which  the action of $A$ is zero.
%\end{definition}
 The notion of representations of almost Lie algebroids obviously gives back the usual notion of representations of Lie algebroids when $A$ is a Lie algebroid. %Extending this notion of representations to the realm of Lie $\infty$-algebroids requires to work up to homotopy.
Trivial representations of $A$ are characterized by the existence of a global frame which is invariant under the action of $A$. However, the action of $A$ has no reason to be trivial on another global frame, even on a constant one. It is then important to distinguish trivialness as a vector bundle and trivialness as a $A$-module.
Examples of representations of almost Lie algebroids can be induced by representations of Lie $\infty$-algebroids on line bundles:
 %More precisely, let $E$ be a Lie $\infty$-algebroid over $M$ and let $L\to M$ be a line bundle, that we assume to be concentrated in degree 0. The following Lemma makes clear that the representation of $E$ on $L$ descends to a representation of $E_0$ on $L$:

\begin{lemme}\label{refppmd}
Any representation up to homotopy of the Lie $\infty$-algebroid $E=\bigoplus_{i\geq0}E_{-i}$ over $M$ on a line bundle $L\to M$ canonically induces a representation of the almost Lie algebroid $E_0$ on $L$.
\end{lemme}

\begin{proof}
Assume that $L$ is a representation up to homotopy of $E$ concentrated in degree 0. For degree reasons, the only space of $E$ acting on $L$ is $E_0$, and the only connection $k$-form that is not zero is the connection 1-form $\omega^{(1)}$. Thus, there exists an $E_0$-connection $\nabla\colon\Gamma(E_0)\times\Gamma(L)\longrightarrow\Gamma(L)$ satisfying the following set of equations:
 \begin{align}
 \nabla_{l_1(u)}&=0\label{eqnoble1}\\
 [\nabla_a,\nabla_b]&=\nabla_{l_2(a,b)}\label{eqnoble2}
 \end{align}
for every $a,b\in\Gamma(E_0)$ and $u\in\Gamma(E_{-1})$.  The connection is flat because of Equation~\eqref{eqnoble2}.

Eventually notice that Equation \eqref{bianchilike} is satisfied since the higher Jacobi identity \eqref{superjacobi} for three elements of $E_0$ is:
\begin{equation}
l_2\big(a, l_2(b,c)\big)+l_2\big(b, l_2(c,a)\big)+l_2\big(c, l_2(a,b)\big)=l_1\big(l_3(a,b,c)\big)
\end{equation}
so that, for $u=l_3(a,b,c)$, Equation \eqref{eqnoble1} implies that the last term of Equation \eqref{bianchilike} is zero, as should be since the left-hand side is zero too.
\end{proof}

 Assume for a time that the line bundle $L$ is trivial as a vector bundle, and let $\Omega$ be a nowhere vanishing global section. 
  The action of $E$ on $\Omega$ -- in fact only $E_0$ is acting -- %induced by that of $E$
   is again proportional to $\Omega$, so there exists a 1-form $\theta_\Omega\in\Gamma(E_0^*)$,  called the \emph{modular 1-form (w.r.t. $\Omega$)}, satisfying:
\begin{equation}\label{definitionmodular30}
\nabla_a(\Omega)=\theta_\Omega(a)\,\Omega
\end{equation}
for every $a\in\Gamma(E_0)$. Equation \eqref{eqnoble1} implies that:
\begin{equation}\label{vanishmodularbis}
\theta_\Omega\big|_{\Gamma(l_1(E_{-1}))}=0
\end{equation}
%This equation is equivalent to the fact that $d_E^{(0)}\theta_\Omega=0$.
 Furthermore,  Equation \eqref{eqnoble2} is equivalent to:
\begin{equation}\label{vanishmodularter}
\rho(a)\big(\theta_\Omega(b)\big)-\rho(b)\big(\theta_\Omega(a)\big)-\theta_\Omega\big(l_2(a,b)\big)=0
\end{equation}
Using the definition of the differential $d_E$ as given in Equations \eqref{eqdiff1} and \eqref{eqdiff2}, Equations \eqref{vanishmodularbis} and \eqref{vanishmodularter} are equivalent, respectively, to:
 \begin{equation}
d_E^{(0)}\theta_\Omega=0 \qquad\text{and}\qquad d_E^{(1)}\theta_\Omega=0
 \end{equation}%A cumbersome but straightforward calculation shows that $d_E^{(0)}\theta_\Omega=0$.

 By definition of the differential $d_E$, these two identities  are in turn equivalent to the fact that $\theta_\Omega$ is $d_E$-closed. In particular it means that this 1-form is a cocycle and defines a class in $H^1(E)$. %Following \cite{caseiroModularClassLie2022} we call the 1-form $\theta_\Omega$ the \emph{cocycle function of $(L,\Omega)$}.
As in the Lie algebroid case, this cohomology class is invariant under the choice of section $\Omega$ of $L$; % -- that would necessarily have been expressed as $\Omega'=f\,\Omega$, for $f$ a nowhere vanishing function on $M$ -- then we would have:
%\begin{equation}
%\theta_{\Omega'}-\theta_\Omega=d_E\big(\mathrm{ln}(|f|)\big)\end{equation}
 %Thus, the cohomology class of $\theta_\Omega$ in $H^1(E)$ is the same as that of $\theta_{\Omega'}$;
  we denote it by $\theta_L$. When $L$ is not a trivial line bundle, we cannot define $\theta_L$ as above. However, following \cite{evensTransverseMeasuresModular1999}, the line bundle $L^2=L\otimes L$ is trivial so that one can associate the following distinguished cohomology class to $L$:
 \begin{equation}
 \theta_L=\frac{1}{2}\theta_{L^2}
 \end{equation}
Hopefully, both definitions coincide  when $L$ is trivial.

\begin{definition}\label{def:char}
We call \emph{characteristic class of the line bundle $L$} the  class $\theta_L\in H^1(E)$. %of \emph{any} one-form $\theta_\Omega$ associated to the action of $E$ on a $L$ via Equation \eqref{definitionmodular30}.
\end{definition} 

This characteristic class is thus attached to a particular representation of $E$ on $L$. %Let us now explain its meaning. %We say that a representation $V$ of an almost Lie algebroid $E_0$ (resp. a Lie $\infty$-algebroid $E$) is \emph{trivial} when there exists a nowhere vanishing section of $V$ which vanishes when acted upon by $E_0$ (resp. $E$). 
%It means in particular that trivial representations are trivial vector bundles. Moreover, it is worth emphasizing that a constant frame may \emph{not} be vanishing under the action of $E_0$ or $E$. %It measures the trivialness of the representation $L$, not as a vector bundle, but as a $E$-module (or, equivalently, as an $E_0$-module).
The following proposition shows that it captures only the action of the almost Lie algebroid $E_0$, as no other vector bundle $E_{-i}$ acts on $L$.

\begin{proposition}\label{propiteratif}
Let $L$ be a trivial line bundle and a representation up to homotopy of a Lie $\infty$-algebroid $E$. Then $L$ is a trivial $E_0$-module if and only if $\theta_L=0$.
\end{proposition}

\begin{proof}
Assume that the line bundle $L$ is trivial as a vector bundle and that it admits a nowhere vanishing global section $\Omega$. Moreover suppose that the characteristic class $\theta_L$ is zero, i.e. that there exists a function $f\in\mathcal{C}^\infty$ such that $\theta_\Omega=d_E(f)$. Then, one can check from Equation  \eqref{definitionmodular30} that the nowhere vanishing section $\mathrm{e}^{-f}\Omega$ is invariant under the action of any section $a$ of $E_0$: % evaluating Equation \eqref{definitionmodular30} on a section of $E_0$ gives:
\begin{equation}
\nabla_a\big(\mathrm{e}^{-f}\Omega\big)=\rho(a)(\mathrm{e}^{-f})\,\Omega+\mathrm{e}^{-f}\nabla_a(\Omega)=\big(-\rho(a)(f)\,+d_E(f)(a)\big)\mathrm{e}^{-f}\Omega=0
\end{equation}
%This global section of $L$ thus generates all other sections of $L$. 
The action of $E_0$ being zero on this section, we conclude that the trivial line bundle $L$ is a trivial representation of $E_0$. 
Conversely, if $L$ is a trivial $E_0$-module, then there exists a global section $\Omega$ of $L$ invariant under the action of $E_0$, meaning that the left-hand side of Equation~\eqref{definitionmodular30} vanishes. This implies that the modular cocycle $\theta_\Omega$ vanishes, hence the result.
\end{proof}

Thus, the characteristic class of the line bundle $L$ measures the trivialness of the representation $L$, not only as a vector bundle, but as a $E_0$-module.  %  (or, equivalently, $E$). 
In particular a trivial line bundle $L$ which is \emph{not} a trivial representation of $E_0$ admits nowhere vanishing global sections, but none of them is zero under the action of $E_0$.
We will now apply this result to a line bundle of particular interest, which is associated to any Lie $n$-algebroid $E$:

\begin{definition} Let $E$ be a Lie $n$-algebroid, for some $n\geq1$. The determinant line bundle of $E$ is called the \emph{Berezinian line bundle of $E$} and is defined, depending on the parity of $n$, as:
%\begin{equation}\label{eq:Berezinian}
%\mathrm{Ber}(E)= \wedge^{\mathrm{top}}(TM)^*\otimes\wedge^{\mathrm{top}}E_{0}\otimes \wedge^{\mathrm{top}}E^*_{-1}\otimes\ldots\otimes\wedge^{\mathrm{top}}E^*_{-n+1}
%\end{equation}
\begin{align}
&\text{$n$ even}&&\mathrm{Ber}(E)=\wedge^{\mathrm{top}}T^*M\otimes \wedge ^{\mathrm{top}}E_0\otimes \wedge^{\mathrm{top}} (E_{-1})^*\otimes \wedge^{\mathrm{top}} E_{-2}\otimes \ldots\otimes \wedge^{\mathrm{top}} (E_{-n+1})^{*}&\label{eq:Berezinian}\\
&\text{$n$ odd}&&\mathrm{Ber}(E)=\wedge^{\mathrm{top}}T^*M\otimes \wedge ^{\mathrm{top}}E_0\otimes \wedge^{\mathrm{top}} (E_{-1})^*\otimes \wedge^{\mathrm{top}} E_{-2}\otimes \ldots\otimes \wedge^{\mathrm{top}} E_{-n+1}&\label{eq:Berezinian1}
\end{align}
Here, we consider the spaces $E_{-i}$ as vector spaces and do not keep track of the graduation. 
\end{definition}

In full rigor, the original formula of the Berezinian uses symmetric products for odd graded vector spaces, and alternating products for even graded vector spaces. Once we do not  take graduation into consideration and use only alternating products, the original formula of the Berezinian is canonically isomorphic to the right hand side of Equations \eqref{eq:Berezinian}, \eqref{eq:Berezinian1}. We decided not to keep track of the graduation because it is not used further in the definition of modular classes as well as in computations.

\begin{proposition}\label{bronol}
The Berezinian of the Lie $n$-algebroid $E$ comes with a canonical structure of representation up to homotopy of $E$ (and hence, by Lemma \ref{refppmd}, with a $E_0$-module structure).
\end{proposition}
\begin{proof}
Given a frame $e^{(i)}_{1},\ldots, e^{(i)}_{\mathrm{dim}(E_{-i})}$ of $E_{-i}$, and a section $a\in \Gamma(E_0)$, we define the \emph{Lie derivative of $a$} as the unique differential operator $\mathcal{L}_a\colon\Gamma(E_{-i})\to \Gamma(E_{-i})$ which coincides with the 2-bracket:
\begin{equation}\label{exserv}
\mathcal{L}_a\big(e^{(i)}_{k}\big)=l_2\big(a,e^{(i)}_{k}\big)
\end{equation}
It induces a dual differential operator on $(E_{-i})^*$, satisfying the usual formula $\mathcal{L}_a=d_E\iota_a+\iota_ad_E$. The action of the Lie derivative straightforwardly extends as a derivation to the line bundle $\wedge^{\mathrm{dim}(E_{-i})}E_{-i}$. Setting $\mu^{(i)}=e^{(i)}_{1}\wedge \ldots\wedge e^{(i)}_{\mathrm{dim}(E_{-i})}$, we define the smooth function $\mathrm{div}_{\mu^{(i)}}(a)\in \mathcal{C}^\infty(M)$ as the following proportionality coefficient:
\begin{equation}\label{bracos3}
\mathcal{L}_{a}\big(\mu^{(i)}\big)=-\mathrm{div}_{\mu^{(i)}}(a)\cdot \mu^{(i)}
\end{equation}  %It indeed coincides with the cohomology class in $H^1(E)$ of the supertrace of the action of $E_0$ on $\mathrm{Ber}(E)$ \textbf{pas clair ici car c pas la cohomology class qui doit apparaitre + y a une propo de ca dan raquelcamille  / on doit en fait avoir le modular cocycle}:
By duality, the action of $a$ on the dual line bundle $\wedge^{\mathrm{dim}(E_{-i})}(E_{-i})^*$ -- with dual volume form $\mu^{(i)*}$ -- defines a proportionality coefficient which is minus the latter:
\begin{equation}\label{bracos4}
\mathcal{L}_{a}\big(\mu^{(i)*}\big)=\mathrm{div}_{\mu^{(i)}}(a)\cdot \mu^{(i)*}
\end{equation}
The convention has been chosen so that if $E=E_0=TM$ then we find the usual divergence of a vector field.
Moreover, $E_0$ naturally acts on $\wedge^nT^*M$ with the (usual) Lie derivative $\mathcal{L}_{\rho(a)}$:
\begin{equation}\label{diverg}
\mathcal{L}_{\rho(a)}(\omega)=\mathrm{div}_\omega\big(\rho(a)\big)\cdot \omega
\end{equation}
for any section $\omega$ of $\wedge^nT^*M$.

Let $\Omega=\omega\otimes \mu^{(1)}\otimes \mu^{(2)*}\otimes \mu^{(3)}\otimes\ldots $ be a section of $\mathrm{Ber}(E)$. Given the above discussion, the action of $E_0$ on the Berezinian of $E$ via the Lie derivative amounts to:
\begin{equation}\label{gluttony4}
\mathcal{L}_a(\Omega)=\left(\mathrm{div}_\omega\big(\rho(a)\big)-\sum_{i=0}^{n-1}(-1)^i\mathrm{div}_{\mu^{(i)}}(a)\right)\cdot \Omega
\end{equation}
The sums stops at $i=n-1$ because $E$ is a Lie $n$-algebroid.
The term in parenthesis on the right-hand side of Equation \eqref{gluttony4} corresponds to the supertrace of the operator $\mathcal{L}_a\colon\Gamma(E\oplus sTM)\to\Gamma(E\oplus sTM)$, where $\mathcal{L}_a$ is understood to act as $\mathcal{L}_{\rho(a)}$ on $sTM$.   More generally, any vector bundle endomorphism $T\colon\Gamma(E\oplus sTM)\to \Gamma(E\oplus sTM)$ induces an endomorphism $\widetilde{T}$ of $\mathrm{Ber}(E)$ and is such that:
 \begin{equation}\label{gluttony5}
 \widetilde{T}(\Omega)=\mathrm{Str}_\Omega(T)\cdot\Omega
 \end{equation}
 As expected, the supertrace, when evaluated on the commutator of two operators, is zero:
 \begin{equation}\label{flow}
\mathrm{Str}_\Omega\big([S,T]\big)=0
 \end{equation}
Here, the commutator is understood in the space of vector bundle endomorphisms of the graded vector bundle $E\oplus sTM$.

The 1-form $\theta_\Omega$ associated to the action of $E_0$ on $\mathrm{Ber}(E)$ (see Equation \eqref{definitionmodular30}) thus satisfies:
\begin{equation}
\theta_\Omega(a)= \mathrm{Str}_\Omega\big(\mathcal{L}_a\big)
\end{equation}
%Equation \eqref{flow} has two main consequences.
First, by the higher Jacobi identity \eqref{leibnizzzz}, we have that,  for any $u\in \Gamma(E_{-1})$, $\mathcal{L}_{l_1(u)}=[l_1,\mathcal{L}_{u}]$. Then, by Equation \eqref{flow}, we deduce that:
\begin{equation}\label{gluttony7}
d^{(0)}_E\theta_\Omega(u)=\theta_\Omega\big(l_1(u)\big)= \mathrm{Str}_\Omega\big([l_1,\mathcal{L}_{u}]\big)=0
\end{equation}
Second, by the Jacobi identity \eqref{jacobiiiii}  we deduce that the Lie derivative satisfies the following identity on $\Gamma(E\oplus sTM)$:
 \begin{equation}\label{gluttony6}
 \mathcal{L}_{l_2(a,b)}-[\mathcal{L}_a,\mathcal{L}_b]=-[l_1,l_3(a,b,\,.\,)]
 \end{equation}
By Equation \eqref{gluttony5}, together with the trace property \eqref{flow} of the supertrace, we deduce that:
\begin{equation}\label{miklouho}
\left(\mathcal{L}_{l_2(a,b)}-[\mathcal{L}_a,\mathcal{L}_b]\right)(\Omega)=-\mathrm{Str}_\Omega\big([l_1,l_3(a,b,\,.\,)]\big)\cdot \Omega=0
\end{equation}
 But, if one expands the left-hand side by using the definition of the Lie derivative, then one obtains:
\begin{equation}
\left(\mathcal{L}_{l_2(a,b)}-[\mathcal{L}_a,\mathcal{L}_b]\right)(\Omega)=\left(\rho(a)\big(\theta_\Omega(b)\big)-\rho(b)\big(\theta_\Omega(a)\big)-\theta_\Omega\big(l_2(a,b)\big)\right)\cdot \Omega
\end{equation}
So we conclude that:
\begin{equation}\label{gluttony8}
d^{(1)}_E\theta_\Omega(a,b)=-\mathrm{Str}_\Omega\big([l_1,l_3(a,b,\,.\,)]\big)=0
\end{equation}
Equations \eqref{gluttony7} and \eqref{gluttony8} show that $\theta_\Omega$ is $d_E$-closed, so that $\mathrm{Ber}(E)$ is indeed a representation (up to homotopy) of the Lie $\infty$-algebroid $E$, and then also of $E_0$ by Lemma \ref{refppmd}.
\end{proof}

The authors in \cite{caseiroModularClassLie2022} define the Berezinian line bundle of \emph{any} representation $K$ of a Lie $\infty$-algebroid $E$, so that the representation of $E$ on $K$ canonically descends to this Berezinian. This allows them to define the characteristic class of such a representation $K$ as the characteristic class of the associated Berezinian bundle. In particular, when $K=E\oplus sTM$, and the corresponding action is the adjoint action defined in Example \ref{exadjoint}, the Berezinian of $E$ inherits a representation of $E$. A priori, this representation would differ from the canonical representation of Proposition~\ref{bronol} by some contribution coming from the last terms of Equations~\eqref{basicc1} and~\eqref{basicc2}. But one can show that in both cases the summation \eqref{gluttony4} implies that these contributions cancel out and we are left with the canonical representation of $E$ on $\mathrm{Ber}(E)$ induced by Equations~\eqref{exserv} and~\eqref{diverg}. In other words, 
the representation of $E$ on $\mathrm{Ber}(E)$ induced by \emph{any} adjoint representation of $E$ coincides with the canonical representation of $E$ on $\mathrm{Ber}(E)$ defined in Proposition \ref{bronol}.
Accordingly, the modular class of the Lie $n$-algebroid  $E$ is the characteristic class of any adjoint representation of $E$ or, equivalently, the characteristic class of the Berezinian of $E$ equipped with the canonical representation defined in Proposition \ref{bronol}:

\begin{definition} \cite{caseiroModularClassLie2022} \label{defmodular}
Let $E$ be a Lie a $n$-algebroid over $M$, then the characteristic class $\theta_{\mathrm{Ber}(E)}$ associated to the canonical representation of $E$ on  the Berezinian line bundle $\mathrm{Ber}(E)$, as defined in Proposition \ref{bronol}, is called the \emph{modular class} of $E$ and is denoted $\theta^E$.
%Let $E$ be a Lie $n$-algebroid over $M$ (for some $n\geq1$) whose Berezinian $\mathrm{Ber}(E)$ is a trivial line bundle, and let $\Omega$ be any nowhere vanishing section of $\mathrm{Ber}(E)$. The cohomology class of the 1-form $\theta_\Omega$ defined by Equation \eqref{definitionmodular3} is called \emph{the modular class of the Lie $n$-algebroid $E$} and is noted~$\theta^E$. We say that $E$ is \emph{unimodular} when $\theta^E=0$.
\end{definition}

\begin{remarque}
By Proposition \ref{propiteratif}, as a characteristic class of a line bundle, the modular class measures the trivialness of the Berezinian as a representation of $E_0$: i.e. $\theta^E=0$ if and only if there exists a nowhere vanishing section of $\mathrm{Ber}(E)$, which is zero under the action of $E_0$.  The notion of modular class of Lie $n$-algebroids coincides with that of Lie algebroids when $n=1$.
\end{remarque}

\begin{example}\label{ex:foliationNYC}
By Frobenius' integrability theorem, a regular foliation on a smooth manifold can be understood as an involutive regular distribution $F$ of the tangent bundle $TM$. By abuse of denomination and by analogy with singular foliations, we call $F$ regular foliation too.  %to say:
%\begin{equation}
%\big[\Gamma(F),\Gamma(F)\big]\subset\Gamma(F)
%\end{equation}
%where $[\,.\,,.\,]$ is the Lie bracket of vector fields on $M$.
  This sub-bundle becomes a Lie algebroid when equipped with the inclusion $\iota\colon F\to TM$ as anchor map.  %The Lie algebroid structure equipping a regular foliation is called a \emph{foliation Lie algebroid}. % and Lie bracket the Lie bracket of vector fields.
The modular class of $F$  of this \emph{foliation Lie algebroid} is a class in the first cohomology group of the foliated cohomology and is precisely the modular class of the regular foliation. That is to say,  the obstruction to the existence of transverse measures to the leaves, invariant under the flow of vector fields tangent to the leaves (see Section \ref{subsecc3} and \cite{kamberFoliatedBundlesCharacteristic1975}). \end{example}

%To apply this result to singular foliations, w
We now need to be able to compare the modular classes of homotopy equivalent Lie $\infty$-algebroids (see Definition \ref{def:homtequiv}). The notion of modular class of a morphism of Lie algebroids~\cite{kosmann-schwarzbachRelativeModularClasses2005} straightforwardly generalizes to Lie $\infty$-algebroids. Indeed, given a Lie $\infty$-morphism $\Phi\colon \Omega_\blacktriangle(E')\to\Omega_\blacktriangle(E)$, it induces a morphism of algebras $\widetilde{\Phi}\colon H^\blacktriangle(E')\to H^\blacktriangle(E)$  at the level of cohomologies. Then the modular class of $\Phi$ is given by:
 \begin{equation}
 \theta^\Phi=\theta^E-\widetilde{\Phi}(\theta^{E'})
 \end{equation}
In the Lie algebroid context, if $\varphi\colon A\longrightarrow A'$ is a Lie algebroid isomorphism, the modular class $\theta^\varphi$ is zero.
  However, for Lie $\infty$-algebroids, the notion of isomorphism is weakened to that of homotopy equivalence. In that context, the morphism $\widetilde{\Phi}$ is an isomorphism, which turns out to be canonical: it does not depend on the homotopy equivalence, see  Section \ref{singlg}. Then the following result holds \cite{caseiroModularClassLie2022}:
   %However, even then, a similar result has been proven in  \cite{caseiroModularClassLie2022}:
\begin{proposition}\label{interestingprop}
Let $E$ and $E'$ be two Lie $\infty$-algebroids of finite length. %over some manifold $M$ and $M'$, respectively. 
If $\Phi$ is a homotopy equivalence between $E$ and $E'$, then $\theta^\Phi=0$, i.e. $\Phi$ intertwines the modular classes of $E$ and $E'$.\end{proposition}

%Any homotopy equivalence $\Phi\colon\Omega_\blacktriangle(E')\to\Omega_\blacktriangle(E)$ between $E$ and $E'$ induces an isomorphism $\widetilde{\Phi}_{E,E'}$

%Since homotopy equivalences between Lie $\infty$-algebroids are quasi-isomorphisms, the following result is straightforward, and will be useful in the next section 

\begin{corollaire}\label{prop29}
Let $E$ and $E'$ be two homotopy equivalent Lie $\infty$-algebroids of finite length. Then $E$ and $E'$ are simultaneously unimodular, i.e.  $E$ is unimodular if and only if $E'$ is unimodular.
\end{corollaire}

\begin{example}\label{examplereg}
Let $A$ be a regular Lie algebroid over $M$, i.e. such that the anchor map $\rho$ has constant rank. In that case its kernel is a vector subbundle of $A$ denoted $\mathrm{Ker}(\rho)$ and called \emph{isotropy Lie algebra bundle}. The image of the anchor map defines a regular foliation $F=\rho(A)$.
One can equip ${E}= A\oplus s^{-1}\mathrm{Ker}(\rho)$ with a (strict) Lie 2-algebroid structure, compatible with the Lie algebroid structure on $A$. The 1-bracket on $E$ is the inclusion of $s^{-1}\mathrm{Ker}(\rho)$ into $A$:
\begin{equation*}
l_1=\iota\circ s:s^{-1}\mathrm{Ker}(\rho) \xrightarrow{\hspace*{0.8cm}} A
\end{equation*}
and it induces an exact sequence of vector bundles:
\begin{center}
\begin{tikzcd}[column sep=0.7cm,row sep=0.4cm]
0\ar[r]&s^{-1}\mathrm{Ker}(\rho)\ar[r,"l_1"]&A\ar[r,"\rho"]&F\ar[r]&0\end{tikzcd}
\end{center}
The 2-bracket between two elements of $\Gamma(A)$ is the Lie algebroid bracket on $A$, and
the 2-bracket between a section $a$ of $A$ and a section $b$ of $s^{-1}\mathrm{Ker}(\rho)$ is defined as:
\begin{equation}
l_2(a,b)= s^{-1}[a,sb]_A
\end{equation}
%By Equation \eqref{jpp}, the right-hand side indeed takes values in $s^{-1}\mathrm{Ker}(\rho)$.
%This bracket coincides (modulo a shifting) with the action $\nabla^{\mathfrak{g}(A)}$ of $A$ on $\mathfrak{g}(A)$ defined in Equation \eqref{operation2}.  
The 2-bracket $l_2$ satisfies the usual Leibniz rule:
\begin{equation}
l_2(a,fb)=f\,l_2(a,b)+\rho(a)[f]\,b
\end{equation}
for every $a\in\Gamma(A)$, $b\in \Gamma({E})$ and $f\in\mathcal{C}^\infty$. There is no 3-bracket because $A$ is a Lie algebroid.

The quotient vector bundle $\bigslant{A}{\mathrm{Ker}(\rho)}$ is a well-defined vector bundle over $M$, canonically isomorphic to $F$.
Then, the map $\varphi\colon E\to F$ acting as the quotient map $A\longrightarrow\bigslant{A}{\mathrm{Ker}(\rho)}$ on $E_0=A$ and as the zero map on $E_{-1}=s^{-1}{\mathrm{Ker}(\rho)}$ induces a canonical homotopy equivalence between $E$ and the foliation Lie algebroid $F$:
\begin{center}
\begin{tikzcd}[column sep=0.7cm,row sep=0.7cm]
0\ar[r]&s^{-1}\mathrm{Ker}(\rho)\ar[d,"\varphi"]\ar[r,"l_1"]&A\ar[d,"\varphi"]\ar[r,"\rho"]&F\ar[d,"\mathrm{id}"]\ar[r]&0\\
0\ar[r]&0\ar[d]\ar[r]&\bigslant{A}{\mathrm{Ker}(\rho)}\ar[d,"\sim"]\ar[r,"\sim"]&F\ar[d,"\mathrm{id}"]\ar[r]&0\\
0\ar[r]&0\ar[r]&F\ar[r,"\mathrm{id}"]&F\ar[r]&0
\end{tikzcd}
\end{center}
Then Corollary \eqref{prop29} applies so the Lie 2-algebroid $E$ and the Lie algebroid $F$ are simultaneously unimodular: the modular class of $E$ vanishes if and only if that of $F$ vanishes. Two reasons explain this:
first, %the modular class of $E$ is the characteristic class of
 the Berezinian Line bundle $\mathrm{Ber}(E)=\wedge^{n}T^*M\otimes \wedge^{\mathrm{top}}A\otimes \otimes \wedge^{\mathrm{top}}\mathrm{Ker}(\rho)$ is canonically isomorphic to the line bundle $Q_F=\wedge^{n}T^*M\otimes \wedge^{\mathrm{top}}F$. The respective action of $A$ on $\mathrm{Ber}(E)$ and that of $F$ on $Q_F$ give the same value of the parenthesis in Equation \eqref{gluttony4}.

Second, there is a one-to-one correspondence between closed one-forms on $E$ of height 1 and closed one-forms on $F$.
 %$\varphi$ induces a canonical isomorphism between $H^1(E)$ and the first group of cohomology of the Lie algebroid $F$: \begin{equation}H^1(E)\simeq H^1(F)\end{equation}
 More precisely, a closed one-form on $E$ of height 1, say $\theta$, is a section of $A^*$, that satisfies the following two conditions:
\begin{align}
&d^{(0)}_E(\theta)=0 &\Longleftrightarrow&&\theta\big|_{\mathrm{Ker}(\rho)}&=0\label{id:1}\\
&d^{(1)}_E(\theta)=0&\Longleftrightarrow&&\theta\big([a,b]\big)&=\rho(a)\big(\theta(b)\big)-\rho(b)\big(\theta(a)\big)\label{id:2}
\end{align}
for every $a,b\in \Gamma(A)$.
  Identity \eqref{id:2} illustrates the fact that $\theta$ is closed in the Lie algebroid cohomology of $A$. Notice that not every closed form in the Lie algebroid cohomology of $A$ vanish on $\mathrm{Ker}(\rho)$: only those that are closed with respect to the differential on $E$ do so.  Identity \eqref{id:1} shows that such a one form canonically defines a differential one-form $\theta'$ on $F\simeq \bigslant{A}{\mathrm{Ker}(\rho)}$ because it vanishes on the kernel of the anchor map. Since exact one-forms are also in one to one correspondence, we deduce that $H^1(E)$ and $H^1(F)$ are isomorphic, proving the desired correspondence between modular classes. %The modular class of $E$ is the characteristic class of the Berezinian Line bundle $\mathrm{Ber}(E)=\wedge^{n}T^*M\otimes \wedge^{\mathrm{top}}A\otimes \otimes \wedge^{\mathrm{top}}\mathrm{Ker}(\rho)$, which is canonically isomorphic to the line bundle $Q_F=\wedge^{n}T^*M\otimes \wedge^{\mathrm{top}}F$. %Since $\varphi$ is a homotopy equivalence between $E$ and $F$, we deduce that $\theta'$ is closed as well. %By identity \eqref{id:1}, one concludes that  the induced one-form $\theta'\in\Omega^1(F)$ satisfies:
 % \begin{equation}
 % \theta\big(\varphi([a,b])\big)&=\rho(\varphi(a))\big(\theta(\varphi(b))\big)-\rho(\varphi(b))\big(\theta(\varphi(a))\big)
 % \end{equation}
%  is closed as well.
   %The converse is not necessarily true: a closed form in the Lie algebroid cohomology of $A$ may not vanish on $\mathrm{Ker}(\rho)$. Since it is however the case for $\theta$, the induce one-form $\theta'\in\Omega^1(F)$ is closed as well. %Since, as a Lie algebroid, $\bigslant{A}{\mathrm{Ker}(\rho)}$ is isomorphic to  the regular foliation $F$, it implies that the Lie 2-algebroid  $E$ and the foliation $F$ are simultaneously unimodular.  %This canonical isomorphism gives another characterization of the modular class of a regular foliation: as the modular class of the Lie 2-algebroid $E=A\oplus s^{-1}\mathrm{Ker}(\rho)$ constructed from any regular Lie algebroid covering $F$, the foliation Lie algebroid being the particular case where the rank of the anchor map is null.

\end{example}

%Not only regular foliation can be seen as foliation Lie algebroids, but more generally regular foliations are images through the anchor map of regular Lie algebroids. As seen in Example \ref{examplereg}, any regular Lie algebroid $A$ defines a Lie 2-algebroid $E=A\oplus s^{-1}\mathrm{Ker}(\rho)$, which is homotopy equivalent to the foliation Lie algebroid $F$. The homotopy equivalence induces an isomorphism between the Berezinians $\mathrm{Ber}(F)$ and $\mathrm{Ber}(E)$, and intertwines the corresponding modular classes (see Proposition \ref{interestingprop}). So, the foliation Lie algebroid $F$ and the Lie 2-algebroid $E=A\oplus s^{-1}\mathrm{Ker}(\rho)$ are simultaneously unimodular. 

The argument presented in Example \ref{examplereg} can be generalized to  Lie $n$-algebroids defining resolutions of the regular foliation $F$:
\begin{proposition}\label{propastute}
Let $F$ be a regular foliation on $M$, and let $E$ be any Lie $n$-algebroid such that its linear part  %1. the 1-bracket $l_1$ together with the anchor map $\rho$ have constant rank, and 2. 
defines the following exact sequence of vector bundles:
\begin{equation}\label{cpxe}
\begin{tikzcd}[column sep=0.7cm,row sep=0.4cm]
0\ar[r]&E_{-n+1}\ar[r,"l_1"]&\ldots\ar[r,"l_1"]&E_{-1}\ar[r,"l_1"]&E_{0}\ar[r,"\rho"]&F\ar[r]&0
%&\ldots\ar[r,"\delta_{-(k-2)}"]&V^*_{-k+1}\ar[r,"\delta_{-(k-1)}"]&U_{-k}^*\ar[r,"\delta_{-k}"]&\ldots
\end{tikzcd}
\end{equation}
Then $F$ and $E$ are simultaneously unimodular.
\end{proposition}
The observation made in Proposition \ref{propastute} is the central idea behind the extension of the notion of modular class to the context of singular foliations. Indeed, if one finds a replacement of a singular foliation $\mathcal{F}$ by a Lie $\infty$-algebroid $E$ generalizing the Lie 2-algebroid $A\oplus s^{-1}\mathrm{Ker}(\rho)$, one may define (a representent of) the modular class of $\mathcal{F}$ as the modular class of $E$. %This will make computations much more affordable. 
The natural candidates of such structures are the universal Lie $\infty$-algebroids of $\mathcal{F}$.

%This implies that there cohomologies are isomorphic, and that they are simultaneously unimodular (see Corollary \ref{interestingprop}). %simultaneously unimodular. % More generally, for any regular foliation $F$ and  for any regular Lie algebroid $A$ such that the chain complex:
%\begin{center}
%\begin{tikzcd}[column sep=0.7cm,row sep=0.4cm]
%0\ar[r]&\mathrm{Ker}(\rho)\ar[r]&A\ar[r]&F\ar[r]&0
%&\ldots\ar[r,"\delta_{-(k-2)}"]&V^*_{-k+1}\ar[r,"\delta_{-(k-1)}"]&U_{-k}^*\ar[r,"\delta_{-k}"]&\ldots
%\end{tikzcd}
%\end{center}
%is a short exact sequence of vector bundles, 
%In fact, the foliation Lie algebroid  $F$ is homotopy equivalent to the Lie 2-algebroid $E=A\oplus^{-1}\mathrm{Ker}(\rho)$.

%\begin{example} For any two regular Lie algebroids $A$ and $A'$, any Lie 2-algebroid morphism $\Phi\colon\Omega_\blacktriangle({E}')\to \Omega_\blacktriangle({E})$ admits an inverse up to homotopy, by Theorem \ref{theo:onlyOne}, and due to the fact that both ${E}$ and ${E}'$ are universal Lie $\infty$-algebroid of $F$ (see Definition \ref{def:universal}). By the above Proposition and Corollary \ref{corolle}, this implies that the modular class of $\Phi$ vanishes. 
%\end{example}

\subsection{The modular class of a singular foliation}\label{subsecc3}

 %The next section will present such a candidate.
%To conclude this section, let us explain how one may generalize the notion of modular class to singular foliations. We cannot follow the same line of argument as in the present section since the dimension of the leaves of singular foliations is not constant, so that the notion of conormal bundle is not so well defined. 
%We have seen  that a regular foliation $F$ is a Lie algebroid, with anchor map the inclusion $\iota\colon F\to TM$ and Lie bracket the Lie bracket of vector fields. In that case the Lie algebroid cohomology coincides with the foliated de Rham cohomology. This implies that the modular class of the foliation Lie algbroid, as defined in Definition \ref{defmodular}, is precisely the modular class of the regular foliation as defined in Definition \ref{defmodular2}. 
%Then, as the modular class of $\mathcal{F}$ shall be an element of the first group of the foliated cohomology associated to $\mathcal{F}$. %The next subsection is devoted to introduce the latter notion.

Let us first describe the geometric meaning of the modular class of a regular foliation, in order to make sense later of that of a singular foliation.
A Lie algebroid representation of a regular foliation $F$ is the \emph{annihilator bundle} $F^\circ\subset T^*M$, which  is generated by every covector which is zero when evaluated on $F$.
The annihilator bundle is canonically isomorphic to the \emph{conormal bundle} of $F$, which is the vector bundle dual to the normal bundle $\nu(F)$. %In the paper we will arbitrarily use one or the other denomination.
 %In most example we will then work with $F^\circ$ instead of $N(F)$, because we possess explicit local coordinates on the latter. Moreover,
 Then there exists a distinguished flat connection on $F^\circ$,
%\begin{proposition}
%The normal bundle of $F$ admits a canonical flat $F$-connection, %on n'a pas dfini ce que c'tait cf kamber Tondeur page 20-25
called \emph{the Bott connection} \cite{kamberFoliatedBundlesCharacteristic1975}, %, defined by:
%\begin{equation}\label{bottconne}
%\nabla^{\mathrm{Bott}}_u\big(\pi(X)\big)=\pi\big([u,X]\big)
%\end{equation}
%where $u\in \Gamma(F)$, $X\in\mathfrak{X}(M)$ and $\pi:TM\longrightarrow\nu(F)$ is the canonical quotient map.
%\end{proposition}
%This turns $\nu(F)$ into a representation of the foliation Lie algebroid $F$.
% The Bott connection on $\nu(F)$ induces a dual flat connection on the annihilator bundle $F^\circ$, also denoted $\nabla^{\mathrm{Bott}}$ 
 and defined~by:
\begin{equation}\label{definitionbottdual}
\nabla^{\mathrm{Bott}}_u(\xi)=\mathcal{L}_u(\xi)=\iota_u{d}\xi
\end{equation}
where $u\in\Gamma(F)$, $\xi\in\Gamma(F^\circ)$ and $\mathcal{L}_u$ is the Lie derivative. %Obviously, if $\omega$ is a closed differential form, it is an invariant measure.
Here, flatness can be seen as a consequence of the fact that the Lie derivative $\mathcal{L}\colon\mathfrak{X}(M)\to\mathfrak{X}(M)$ is a Lie algebra morphism. Using the derivation property of the Lie derivative, the Bott connection can be extended to $\Omega^\bullet(F^\circ)=\bigoplus_{1\leq i\leq \mathrm{codim}(F)}\Gamma(\wedge^iF^\circ)$. 
%Flatness implies that every vector bundle $\wedge^i F^\circ$ is a $F$-module, and for later convenience we denote 
We denote $\wedge^\mathrm{top} F^\circ$ the top power $\wedge^{\mathrm{codim}(F)} F^\circ$, which is then a one-dimensional representation of $F$. % so one should read  $\wedge^{\mathrm{codim}(F)} F^\circ$. %, where the codimension of $F$ is the difference $\mathrm{dim}(M)-\mathrm{rk}(F)$. 

\begin{definition}\label{defmodular2}
We say that $F$ is \emph{transversely orientable} if $\wedge^\mathrm{top} F^\circ$ is a trivializable vector bundle or, equivalently, if it admits a nowhere vanishing global section.
The modular class $\theta^F$ of a transversely orientable regular foliation $F$ is the characteristic class of the $F$-module $\wedge^{\mathrm{top}}F^\circ$.
%Let $F$ be a\theta_ regular foliation on $M$, and let $\omega$ be any nowhere vanishing section of the line bundle $\wedge^{\mathrm{top}}F^\circ$. The cohomology class of the horizontal 1-form $\theta_\omega$ defined by Equation \eqref{definitionmodular} is called \emph{the modular (or Reeb) class of the regular foliation $F$} and is noted $\theta^F$.
\end{definition}

The definition of the modular class of a regular foliation as given in Definition \ref{defmodular2} is consistent with the definition of the modular class presented in Example \ref{ex:foliationNYC}.
This is a consequence of the fact that a short exact sequence of vector bundles:
\begin{center}
\begin{tikzcd}[column sep=0.7cm,row sep=0.4cm] 
0\ar[r]&A\ar[r]&B\ar[r]&C\ar[r]&0
\end{tikzcd}
\end{center}
implies that there is a canonical isomorphism between $\wedge^{\mathrm{top}}B$ and $\wedge^{\mathrm{top}}A\otimes \wedge^{\mathrm{top}}C$. From the following short exact sequence:
 \begin{center}
\begin{tikzcd}[column sep=0.7cm,row sep=0.4cm] 
0\ar[r]&F\ar[r,"\iota"]&TM\big|_U\ar[r]&\nu(F)\ar[r]&0
\end{tikzcd}
\end{center}
we deduce that $\wedge^{\mathrm{top}}TM\big|_U$ is canonically isomorphic to $\wedge^{\mathrm{top}}F\otimes\wedge^{\mathrm{top}}(\nu(F))$. Multiplying both sides by $\wedge^{\mathrm{top}}T^*M\big|_U\otimes\wedge^{\mathrm{top}}(\nu(F))^*$ implies in turn that  the line bundle $\wedge^{\mathrm{top}}(\nu(F))^*\simeq\wedge^{\mathrm{top}}F^\circ$ is canonically isomorphic to the Berezinian bundle of the foliation Lie algebroid $F$: % (see the discussion preceding Equation \eqref{sympaleq}):
\begin{equation}\label{sympaleq2}
\wedge^{\mathrm{top}}F^\circ\simeq\wedge^{\mathrm{top}}T^*M\big|_U\otimes \wedge^{\mathrm{top}}F
\end{equation}

 Obviously, $F$  will be transversely orientable when $F^\circ$ is a rank zero vector bundle, i.e. if $F=TM$, for in that case the line bundle $\wedge^\mathrm{top} F^\circ=\wedge^0 F^\circ=M\times \mathbb{R}$ is trivial.  In other cases, the top exterior power of $F^\circ$ is a rank one subbundle of $\wedge^{\mathrm{codim}(F)} TM$, and even if the latter might be a trivial vector bundle, it may well happen that $\wedge^\mathrm{top} F^\circ$ is not trivial as a vector bundle, but only trivializable.
The modular class of $F$ measures the trivialness --  as a representation of $F$ -- of the line bundle $\wedge^\mathrm{top}(F^\circ)$.  In general, the existence of  a nowhere vanishing global section of the determinant bundle $\wedge^{\mathrm{top}}F^\circ$ does not necessarily entail that this section is invariant under the flow of vector fields of $F$.  At least, we deduce from Equation~\eqref{definitionbottdual} that if a chosen transverse volume form is exact, then it is necessarily $F$-invariant.
The vanishing of the modular class of the foliation is precisely equivalent to the existence of an \emph{invariant} volume form transverse to the regular foliation $F$,  this is why we originally require that $F$ is transversely orientable.
%In the case where $F$ is not transversely orientable, there is no nowhere vanishing transverse measure so the modular class cannot vanish.  %one may work with the square root of the density bundle, that is: $\frac{1}{2}\big|\wedge^{\mathrm{top}}F^\circ\big|\otimes \big|\wedge^{\mathrm{top}}F^\circ\big|$. \textbf{ MAIS DANS CE CAS Y A PAS DE TRANSVERSE MEASURE PTDR DONC OSEF.}
 %Indeed, a given  section $\omega\in\Omega^{\mathrm{top}}(F^\circ)$ being $F$-invariant means that the differential form $\omega$ satisfies $\mathcal{L}_u(\omega)=0$ for every $u\in\Gamma(F)$. 
%However, assuming that it is,  not every  nowhere vanishing section of the line bundle $\wedge^\mathrm{top}(F^\circ)$ may be invariant under the action of $F$, not even the constant ones!
%Indeed, constant global sections of a trivial line bundle $L$ are not necessarily $F$-invariant. Conversely, when $L$ is a trivial $F$-module, there exists a global section on which the action of $F$ is trivial, but it may not be a constant section. 
The following examples are providing some situations illustrating the notion of invariant transverse volume forms and the subsequent vanishing of the modular classes:

\begin{example}\label{ex3} The regular foliation consisting of concentric circles on the punctured plane $\mathbb{R}^2-\{(0,0)\}$ is induced by the irrotational vector field $v=-y\partial_x+x\partial_y$. %$v=-\frac{y}{x^2+y^2}\partial_x+\frac{x}{x^2+y^2}\partial_y$. 
It admits a transverse measure $\omega=d(x^2+y^2)$ which is actually invariant because it is exact.
\end{example}

\begin{example}
The regular foliation of $\mathbb{R}^2$ corresponding of horizontal lines in the lower half-plane, and of parabolas of equations $y=a(x^2+1)$ (for $a\geq0$) in the upper half-plane, is generated by the following vector field:
\begin{equation}
v(x,y)=\begin{cases}\partial_x &\text{whenever $y\leq0$}\\
\partial_x+2ax\partial_y &\text{whenever there exists $a\geq0$ such that $y=a(x^2+1)$}
\end{cases}
\end{equation}
Notice that for $a=0$ both definitions coincide. Then it turns out that the differential one-form $\omega=d\big(y-a(x^2+1)\big)$ is an invariant transverse volume form to the foliation.
\end{example}

\begin{example}
In $\mathbb{R}^2\backslash B^2$, where  $B^2$ is the closed 2-ball, let us define the following set of parametrized curves:
\begin{equation}
\forall\ \varphi\in[0,2\pi[\qquad L_{\varphi}=\big\{\big(e^t\mathrm{cos}(\varphi+t),e^t\mathrm{sin}(\varphi+t)\big)\ \big|\ t>0\big\}
\end{equation}
These curves foliate the open set $\mathbb{R}^2\backslash B^2$ into spirals. The tangent vector to a leaf at a given point $(x,y)$ is the one-dimensional subspace $F_{(x,y)}$ of $T_{(x,y)}M$ generated by the element $v=(x-y)\partial_x+(x+y)\partial_y$, which is the velocity vector tangent to the leaf at $(x,y)$. The fiber of the conormal bundle $F^\circ\subset T^*M$ at the same point $(x,y)$ is generated by the covector $\omega=(x+y)dx-(x-y)dy$. This plays the role of a transverse measure to the foliation since $F^\circ$ is one-dimensional. Notice that $\omega$ is not exact hence the Bott connection evaluated on the vector field $v$ does not vanish on it:
\begin{equation}\label{exe1}
\nabla^\mathrm{Bott}_v(\omega)=-2\, \iota_v dx\wedge dy=2\, \omega
\end{equation}
Since the vector bundle $F$ is of rank one, then every vector field tangent to the leaves is colinear to $v$, so that the modular 1-form $\theta_\omega$ can be directly read on Equation \eqref{exe1}:
\begin{equation}
%\theta_\omega(fv)=2f
\theta_\omega(v)=2
\end{equation}

%for every $f\in\mathcal{C}^\infty(\mathbb{R}^2\backslash B^2)$. 
The differential 1-form $\theta_\omega$ is not zero, but it may occur that its cohomology class in $H^1(F)$ is zero, i.e.  that it could be written as an exact 1-form: $\theta_\omega= d_F h$ for some function $h$. In particular this function would have to satisfy the following first-order partial differential equation:
\begin{equation}
(x-y)\frac{\partial h}{\partial x}+(x+y)\frac{\partial h}{\partial y}=2
\end{equation}
A solution of this equation on $\mathbb{R}^2\backslash B^2$ is $h=\frac{1}{2}\mathrm{ln}(x^2+y^2)$. Hence the differential one-form $\theta_\omega$ is exact 
%\begin{equation}
%\theta_\omega=d_F\mathrm{ln}(x^2+y^2)
%\end{equation} 
so the modular class $\theta^F$ of the regular foliation $F$ vanishes, which means that there exists an invariant transverse measure to the foliation. Letting $r^2=x^2+y^2$, one can check that the Bott connection indeed vanishes on the following measure:
\begin{equation}
\omega_{\mathrm{inv}}=\frac{1}{r}\omega
\end{equation}
%1. l'exemple de Debord-Lescure avec les spirales : feuilletage lineaire classe modulaire non nulle et oui les spirales

%2. un hamiltonien sur une variete symplectique : il preserve la forme la forme volume (comme tout champ symplectique) par contre la classe modulaire est nulle (pour une raison a determiner) on retrouve le theoreme de conservation des volumes dans le flot hamiltonien.

%3. variete de poisson classe modulaire des poissons et la relier au feuilletage ??
\end{example}

%\textbf{etendre plutot la notion de characteristic class of a vector bundle a un feuilletage singulier, ca permet de mettre l'accent sur la notion de representation et d'illustrer l'utilite des universal lie infty alg dans un cas + general. En fait on peut definir la notion de characteristic class directement dans la cohom de F, mais pour definir la modular class il faut un fibre vectoriel. Celui correspondant a la prop la en question peut le faire. Ma ensuite demontrer que la modular equal au pull back par rho de la modular class de any of its universal Lie infty algebroid (ca fait une jolie prop). Et surtout en fait on sait pas trop comment le feuilletage singulier agit sur $L_F$ car differemment aux singularites ou non?}
% \textbf{dire que la modular class de TM c'est zero, que l'importance des champs de vecteurs dans un feuilletage singulier fait qu'on a plus le meme resultat et donc qu'il faut utiliser les Lie infty algebroids}

 We cannot straightforwardly reproduce Definition \ref{defmodular2} for singular foliations because the normal bundle and the conormal bundle is not well-defined everywhere in that case. 
Indeed, since the leaves of a singular (non-regular) foliation $\mathcal{F}$ have various dimensions -- and thus so have their transversals -- the counterpart of the normal bundle is not a vector bundle. The wise reader would notice that defining the normal bundle is not necessary, as %made in Example \ref{ex4}
 one only needs to find the counterpart of the line bundle  $\wedge^{\mathrm{top}}F^\circ$ for singular foliations in order to define modular classes of the latter.  The simplest idea would be to pick up the line bundle $\wedge^{\mathrm{top}}F^\circ$ generated by the regular distribution $F$ induced from $\mathcal{F}$, but then we are confronted with the problem of extending this line bundle at singularities.
  More precisely, assume that the singular $\mathcal{F}$ has regular leaves of the same, maximal dimension. % (see Proposition 2.5 in \cite{laurent-gengouxUniversalLieInfinityAlgebroid2020}).
 It is a fact that the union of the regular leaves form a dense open subset $U$ of $M$.  Thus, on $U$, the singular foliation $\mathcal{F}$ induces a regular distribution $F$ admitting a conormal bundle canonically isomorphic to $F^\circ\subset T^*M|_U$. However, this vector subbundle does not necessarily extends to the singular leaves of $\mathcal{F}$, implying that there is no straightforward way of defining the modular class of the singular foliation $\mathcal{F}$ in the sense of Definition \ref{defmodular2}, as the following discussion shows.

\begin{example}\label{ex4} %One can legitimately extend the vector field of Example \ref{ex3} at the origin as $v=-y\partial_x+x\partial_y$ is defined there and vanishes. We set $\mathcal{F}$ to be the singular foliation on $\mathbb{R}^2$ generated by the one vector field $v$. The differential form $\omega=d(x^2+y^2)$ is still defined at the origin and vanishes as well, and $\omega(\mathcal{F})=0$. However, t
 On the punctured plane $U=\mathbb{R}^2-\{(0,0)\}$, the vector bundle $F^\circ$ of Example \ref{ex3} is a subbundle of $T^*\mathbb{R}^2\big|_U$ which cannot be extended at the origin. Indeed,  at each point $(x,y)$ of the punctured plane, $F^{\circ}_{(x,y)}$ is the one dimensional subspace of $T^*_{(x,y)}\mathbb{R}^2\simeq \mathbb{R}^2$ colinear to the vector $(x,y)$, so it has no unequivocal definition at the origin. %As $F^\circ$ is one-dimensional, the line bundle $\wedge^{\mathrm{top}}F^\circ$ coincides with  %Then, the singular foliation on $\mathbb{R}^2$ generated by $v$ does not possess a conormal bundle at the origin.
Although the vector bundle $F^{\circ}$ -- or equivalently $\wedge^{\mathrm{top}}F^\circ$, because $F^\circ$ has rank one -- may not be extended at the origin, it admits a nowhere vanishing global section $\omega=d(x^2+y^2)$ on the punctured plane.
%It is then trivializable, and the corresponding trivial line bundle $U\times\mathbb{R}$ extends to the origin, forming a line bundle $L$ over $\mathbb{R}^2$.  Now, one can also legitimately extend the vector field of Example \ref{ex3} at the origin as $v=-y\partial_x+x\partial_y$ is defined there and vanishes. We set $\mathcal{F}$ to be the singular foliation on the whole space $\mathbb{R}^2$ generated by the one vector field $v$. %%%%%%%%%The differential form $\omega=d(x^2+y^2)$ is still defined at the origin and vanishes as well, and $\omega(\mathcal{F})=0$. However, t
%Since the sections of the trivial line bundle $L$ can be canonically identified with the smooth functions on $\mathbb{R}^2$, the vector field $v$ canonically acts on $L$ via derivation, turning it into a $\mathcal{F}$-module. Any constant section of $L$ is zero under this action, so $L$ is a trivial $\mathcal{F}$-representation. However, this representation does not faithfully represent $\wedge^{\mathrm{top}}F^\circ$ and cannot be used to define the modular class of $\mathcal{F}$ because $L$ is defined at the origin while $\wedge^{\mathrm{top}}F^\circ$ is not.
\end{example}

Although this is a problem to make sense of invariant transverse volume at singularities, we can nonetheless get inspired by Definition \ref{defmodular2} to define the modular class of $\mathcal{F}$.
From the notion of Bott connection, we will keep the idea that a singular foliation acts on a distinguished line bundle. We then need only define what is a representation of a singular foliation, by extending the notion of connection of almost Lie algebroids.

\begin{definition}\label{defkitani} Let $\mathcal{F}$ be a singular foliation, then an $\mathcal{F}$-connection on a (graded) vector bundle $K$ is an operator $\nabla\colon\mathcal{F}\times \Gamma(K)\longrightarrow \Gamma(K)$, satisfying the usual axioms \eqref{eqrep1}-\eqref{eqrep2}. The connection is said  \emph{flat} if its curvature vanishes. In that case $K$ is said to be a \emph{representation of $\mathcal{F}$}, or a \emph{$\mathcal{F}$-module}. %Representations of  $\mathcal{F}$ -- also called $\mathcal{F}$-modules -- are those vector bundles $K$ admitting a flat $\mathcal{F}$-connection. 
We say that a section $s\in\Gamma(K)$ is \emph{$\mathcal{F}$-invariant} when 
\begin{equation}\text{$\nabla_{u}(s)=0$, for every $u\in\mathcal{F}$.}\end{equation} A representation $K$ of $\mathcal{F}$ is \emph{trivial}  when there exists a global frame of $K$ which is $\mathcal{F}$-invariant. 
\end{definition}

\begin{example}
Let $\mathcal{F}=\mathfrak{X}(M)$ and $K=M\times\mathbb{R}$, so sections of $K$ are smooth functions on $M$. Let $\nabla$ be the $TM$-connection on $K$ defined as the standard action of vector fields:
\begin{equation}
\nabla_X(f)=X(f)
\end{equation}
for every vector field $X$ and smooth function $f$. This is a flat connection and the constant functions are $\mathcal{F}$-invariant so $K$ is a trivial $\mathcal{F}$-module.
\end{example}

%Possessing a well-defined notion of de Rham foliated cohomology for the singular foliation $\mathcal{F}$ is well defined, Definition \ref{defkitani} allows to extend %the notion of Following the discussion in Section \ref{sec:char} before Definition \ref{def:char},
 %the notion of characteristic classes of line bundles with respect to the action of a Lie $\infty$-algebroid (see Definition \ref{def:char}) to line bundles that are representations of $\mathcal{F}$.
 
 Given a trivial line bundle $L$, which is additionally a $\mathcal{F}$-module, the characteristic class of $L$ with respect to the action of $\mathcal{F}$  is the cohomology class -- in the foliated de Rham cohomology of $\mathcal{F}$ -- of the foliated one form $\theta_\Omega\in\Omega^1(\mathcal{F})$, defined as in identity \eqref{definitionmodular30}. %, where now $\nabla$ denotes the action of $\mathcal{F}$ on $L$.
  This may provide us with a definition of modular class for singular foliations once we find the correct line bundle.   Proposition \ref{propastute} has shown that the modular class of a regular foliation $F$ can be equivalently computed from any Lie $n$-algebroid forming a geometric resolution of $F$. It is thus natural to extend this notion to singular foliations. Since we want to work with Lie $n$-algebroids in order to define their Berezinian, from now on we will only consider  \emph{solvable} singular foliations. 
  By Theorem \ref{theo:existe}, such a foliation  admits at least one universal Lie $\infty$-algebroid $E$ of finite length.
The fact that the singular foliation admits a universal Lie $\infty$-algebroid of finite length implies that all regular leaves of $\mathcal{F}$ have the same, maximal dimension (see Proposition 2.5 in \cite{laurent-gengouxUniversalLieInfinityAlgebroid2020}).  We set $U$ to be the dense open subset of $M$ consisting of the union of such regular leaves, and the regular distribution induced by $\mathcal{F}$ on $U$  is denoted $F$.  Since $\mathcal{F}\big|_U\subset \Gamma(F)\big|_U$ as sheaves over $U$, the Bott connection associated to the action of $F$ on $F^\circ$ induces an $\mathcal{F}\big|_U$-connection on $F^\circ$ (over~$U$):
 \begin{equation}\label{botte}
 \nabla^{\mathrm{Bott}}_X(\xi)=\iota_X{d}\xi
\end{equation}
for every $X\in\mathcal{F}\big|_U$ and $\xi\in\Gamma(F^\circ)$. This action canonically extends to $\wedge^{\mathrm{top}}F^\circ$, turning it into a $\mathcal{F}\big|_U$-module.

 \begin{proposition}\label{propirenee}
 %Let $E$ be a universal Lie $\infty$-algebroid of finite length of a solvable singular foliation $\mathcal{F}$ on $M$, and let $U$ be the  open dense subset  of $M$ consisting of the union of the regular leaves. 
 %Then t
 Given the above assumptions and notations, the Berezinian line bundle of $E$ satisfies the following two criteria:
 \begin{enumerate}
\item it is canonically isomorphic to $\wedge^{\mathrm{top}}F^\circ$ on $U$, via an isomorphism $\varphi\colon\mathrm{Ber}(E)\big|_U\to\wedge^{\mathrm{top}}F^\circ$;
\item it is a $\mathcal{F}$-module, and $\varphi$ is a morphism of $\mathcal{F}\big|_U$-modules. %, in particular $\wedge^{\mathrm{top}}F^\circ$ would be a trivial $\mathcal{F}\big|_U$-module.
\end{enumerate}
 \end{proposition}
 
 \begin{proof} Our goal is to replace $\wedge^{\mathrm{top}}F$ by the alternating powers of the Lie $n$-algebroid $E$, following the same argument leading to Equation \eqref{sympaleq2}.
On the open dense subset $U$, the linear part of the Lie $n$-algebroid $E$ is a geometric resolution of $F$ (for brevity, we omit to write the restriction to $U$ in the following):
\begin{center}
\begin{tikzcd}[column sep=0.7cm,row sep=0.4cm]
0\ar[r]&E_{-n+1}\ar[r,"l_1"]&\ldots\ar[r,"l_1"]&E_{-1}\ar[r,"l_1"]&E_{0}\ar[r,"\rho"]&F\ar[r]&0
%&\ldots\ar[r,"\delta_{-(k-2)}"]&V^*_{-k+1}\ar[r,"\delta_{-(k-1)}"]&U_{-k}^*\ar[r,"\delta_{-k}"]&\ldots
\end{tikzcd}
\end{center}
Since the rank of each map is constant, one can split this long exact sequence of vector bundles into various short exact sequences of vector bundles:
\begin{center}
\begin{tikzcd}[column sep=0.7cm,row sep=0.4cm] 
0\ar[r]&E_{-n+1}\ar[r]&E_{-n+2}\ar[r]&l_1(E_{-n+2})\ar[r]&0\\
0\ar[r]&l_1(E_{-n+2})\ar[r]&E_{-n+3}\ar[r]&l_1(E_{-n+3})\ar[r]&0\\
&&\ldots&&\\
0\ar[r]&l_1(E_{-2})\ar[r]&E_{-1}\ar[r]&l_1(E_{-1})\ar[r]&0\\
0\ar[r]&l_1(E_{-1})\ar[r]&E_{0}\ar[r]&F\ar[r]&0
\end{tikzcd}
\end{center}

Then one has $\wedge^{\mathrm{top}}E_0\simeq \wedge^{\mathrm{top}}l_1(E_{-1})\otimes \wedge^{\mathrm{top}}F$, which in turn implies that $\wedge^{\mathrm{top}}F\simeq \wedge^{\mathrm{top}}E_0\otimes \wedge^{\mathrm{top}}\big(l_1(E_{-1})\big)^*$. Let us find a replacement for $\wedge^{\mathrm{top}}\big(l_1(E_{-1})\big)^*$ now. We have that $\wedge^{\mathrm{top}}E_{-1}\simeq \wedge^{\mathrm{top}}l_1(E_{-2})\otimes \wedge^{\mathrm{top}}l_1(E_{-1})$ so finally $\wedge^{\mathrm{top}}\big(l_1(E_{-1})\big)^*\simeq \wedge^{\mathrm{top}}l_1(E_{-2})\otimes \wedge^{\mathrm{top}}(E_{-1})^*$. This implies that:
\begin{equation}\label{intertween}
\wedge^{\mathrm{top}}F\simeq \wedge^{\mathrm{top}}E_0\otimes \wedge^{\mathrm{top}}(E_{-1})^*\otimes\wedge^{\mathrm{top}}l_1(E_{-2})
\end{equation}
where the right-hand side is assumed to be restricted to $U$ of course.
Continuing the process up to $E_{-n+1}$, and plugging $\wedge^{\mathrm{top}}F$ into Equation \eqref{sympaleq2}, we deduce that there is canonical isomorphism of line bundles:
\begin{equation}
\wedge^{\mathrm{top}}F^\circ\simeq\mathrm{Ber}(E)\big|_U
\end{equation}
We denote $\varphi\colon\mathrm{Ber}(E)\big|_U\to\wedge^{\mathrm{top}}F^\circ$ the corresponding isomorphism.

By Proposition \ref{bronol}, the Berezinian is a $E_0$-module. This structure canonically induces a $\mathcal{F}$-module structure via the anchor map. Indeed, for a given $X\in\mathcal{F}$, let $a\in\Gamma(E_0)$ be a preimage of $X$ through $\rho$, i.e. such that $\rho(a)=X$, so we define the action of $X$ on a section $\Omega$ of $\mathrm{Ber}(E)$ as:
\begin{equation}\label{eqsmash}
\nabla^{\mathcal{F}}_X(\Omega)=\mathcal{L}_a(\Omega)
\end{equation}
where the right-hand side is defined in Equation \eqref{gluttony4}. The connection $\nabla^{\mathcal{F}}$ does not depend on the preimage of $X$, for the following reason: if one had chosen another preimage $a'\in\Gamma(E_0)$ of $X$, then $\rho(a-a')=0$ so there exists $b\in \Gamma(E_{-1})$ such that $l_1(b)=a-a'$ and $\mathcal{L}_{l_1(b)}(\Omega)=0$ by Equation \eqref{gluttony7}.
 The reason why $l_1(b)=a-a'$ is explained by the fact that $E$ is a universal Lie $\infty$-algebroid of $\mathcal{F}$, which means that its linear part is a geometric resolution of $\mathcal{F}$, at the sheaf level. The fact that $\mathrm{Ber}(E)$ is a $E_0$-module implies that the connection $\nabla^{\mathcal{F}}$ is flat, turning $\mathrm{Ber}(E)$ into a $\mathcal{F}$-module.
 
% \textbf{ici}
 %Assume that $M$ is orientable and that et $\omega$ be a section of 
 %Let us define $\mathrm{Tr}_{TM}(\mathcal{L}_X)$ (resp. $\mathrm{Tr}_{F}(\mathcal{L}_X)$) to be the smooth function induced by the canonical action of $F$ on $\wedge^{\mathrm{top}}TM$ (resp. $\wedge^{\mathrm{top}}F$):
 %\begin{equation}
 %\mathcal{L}_X(\omega)=\mathrm{Tr}_{TM}(\mathcal{L}_X)\omega \text{}
 %\end{equation}

 The line bundle $\wedge^{\mathrm{top}}T^*M\big|_U\otimes \wedge^{\mathrm{top}}F$ is a representation of $F$ under the canonical action of the Lie derivative of vector fields. 
 Since the Bott connection on the normal bundle $\nu(F)$ is the quotient representation induced by the canonical action of $F$ on $TM$ vie the Lie derivative, it is straightforward to check that the line bundles $\wedge^{\mathrm{top}}\nu(F)^*$ and $\wedge^{\mathrm{top}}T^*M\big|_U\otimes\wedge^{\mathrm{top}}F$ are not only isomorphic as line bundles, but also as $F$-modules, and hence as $\mathcal{F}\big|_U$-modules as well. From this, we deduce that Equation \eqref{sympaleq2} also holds at the level of $\mathcal{F}\big|_U$-modules.
 
 Now, the regular foliation $F$ is canonically isomorphic to the quotient vector bundle $\left.\bigslant{E_0}{l_1(E_{-1})}\right|_U$ and the identification can be made via the anchor map $\rho$. Since $l_1(E_0)=0$, the Jacobi identity~\eqref{leibnizzzz} implies that $l_1(E_{-1})$ is stable under  the adjoint action of $E_0$ on itself: $l_2(E_0,l_1(E_{-1}))\subset l_1(E_{-1})$. This action then passes to the quotient and, by Equation \eqref{jpp}, the induced action of $E_0$ on $F\simeq\left.\bigslant{E_0}{l_1(E_{-1})}\right|_U$ coincides with the canonical action of $F$ on itself:
 \begin{equation}
 a\cdot \rho(b)=\rho(l_2(a,b))=[\rho(a),\rho(b)]
 \end{equation}
 It means in particular that the canonical isomorphism between $\wedge^{\mathrm{top}}F$ and $\wedge^{\mathrm{top}}E_0\otimes \wedge^{\mathrm{top}}\big(l_1(E_{-1})\big)^*$ intertwines the action of $E_0$ on both sides. 
 One step further, we know that the vector bundle $l_1(E_{-1})$ is canonically isomorphic to the quotient vector bundle $\left.\bigslant{E_{-1}}{l_1(E_{-2})}\right|_U$, and the identification can be made via $l_1$. By the identities $l_1(E_0)=0$ and $l_1\circ l_2=l_2\circ l_1$, all terms of the short exact sequence:
 \begin{center}
\begin{tikzcd}[column sep=0.7cm,row sep=0.4cm] 
0\ar[r]&l_1(E_{-2})\ar[r]&E_{-1}\ar[r]&l_1(E_{-1})\ar[r]&0
\end{tikzcd}
\end{center}
 are stable under the adjoint action of $E_0$. %, and moreover this action commutes with every arrow on the diagram.
  From this we deduce that the isomorphism $\wedge^{\mathrm{top}}E_{-1}\simeq \wedge^{\mathrm{top}}l_1(E_{-2})\otimes \wedge^{\mathrm{top}}l_1(E_{-1})$ intertwines the action of $E_0$. In turn this implies that the isomorphism of line bundles presented in Equation \eqref{intertween} intertwines the action of $E_0$.

  By going up the ladder of short exact sequences we deduce that the canonical isomorphism:
 \begin{equation}\label{intertween2}
\wedge^{\mathrm{top}}F\simeq \wedge^{\mathrm{top}}E_0\otimes \wedge^{\mathrm{top}}(E_{-1})^*\otimes\wedge^{\mathrm{top}}E_{-2}\otimes \wedge^{\mathrm{top}}(E_{-3})^*\otimes\ldots\big|_U
\end{equation}
 intertwines the action of $E_0$ on both sides. Since the action of $E_0$ is a lift of the action of $\mathcal{F}$, we deduce that the action of $\mathcal{F}\big|_U$ of both sides of Equation \eqref{intertween2} commute with the corresponding canonical isomorphism. Then, the isomorphism between $\wedge^{\mathrm{top}}T^*M\big|_U\otimes\wedge^{\mathrm{top}}F$ and $\mathrm{Ber}(E)\big|_U$ intertwines the action of $\mathcal{F}\big|_U$. Together with the observation that $\wedge^{\mathrm{top}}\nu(F)^*$ and the Berezinian $\wedge^{\mathrm{top}}T^*M\big|_U\otimes\wedge^{\mathrm{top}}F$ are isomorphic $\mathcal{F}\big|_U$-modules, we deduce that $\wedge^{\mathrm{top}}\nu(F)^*$ and $\mathrm{Ber}(E)\big|_U$ are not only canonically isomorphic as line bundles, but also as $\mathcal{F}\big|_U$-modules.
% For each $1\leq i\leq n-2$, the singular foliation $\mathcal{F}|_U$ acts on the short exact sequence:
% \begin{center}
%\begin{tikzcd}[column sep=0.7cm,row sep=0.4cm] 
%0\ar[r]&l_1(E_{-i-1})\ar[r,"l_1"]&E_{-i}\ar[r,"l_1"]&l_1(E_{-i})\ar[r]&0
%\end{tikzcd}
%\end{center}
%through the action of $E_0$ on $E$ via Lie derivative $\mathcal{L}_a=l_2(a,.\,)$. This action is \emph{not} a representation of $E_0$  because Equations \eqref{eqrep3} and \eqref{equaflat23} need not be satisfied.
 %Since $l_1(E_0)=0$, not only this action leaves invariant each term of the sequence, but it also commutes with the differential $l_1$. Hence, we deduce that  $\wedge^{\mathrm{top}}E_{-i}\simeq \wedge^{\mathrm{top}}l_1(E_{-i-1})\otimes \wedge^{\mathrm{top}}l_1(E_{-i})$ not only as line bundles but also as representations of $E_0$. 
 \end{proof}

%\textbf{je suis ici}

%Let $E$ be a universal Lie $\infty$-algebroid of the solvable singular foliation $\mathcal{F}$. Then, by Equation \eqref{eqsmash}, it inherits a $\mathcal{F}$-module structure. The 
%{interestingprop}
%{defmodular}

% we denote by $N(\mathcal{F})$ the unique vector subbundle of $T^*M$ whose restriction to $U$ is $F^\circ$. By a dimensional argument, $N(\mathcal{F})$ does not correspond to the conormal bundle of the singular leaves (singular leaves have lower dimensions than regular ones). However, the top exterior power of $N(\mathcal{F})$ is a line bundle $L_\mathcal{F}=\wedge^{\mathrm{top}}N(\mathcal{F})$ defined everywhere, and is such that its restriction to $U$ coincides with $\wedge^{\mathrm{top}}F^\circ$. 
 % This action extends at singularities, thus canonically inducing an action of $\mathcal{F}$ on $N(\mathcal{F})$, and then on $L_\mathcal{F}$. Additionally, this connection is flat since the Bott connection is. This turns $L_\mathcal{F}$ into an $\mathcal{F}$-module.
 %\begin{definition}
%(Definition 1) The modular class of a solvable singular foliation $\mathcal{F}$ is the characteristic class of the line bundle $L_\mathcal{F}$, with respect to the action of $\mathcal{F}$.
% \end{definition}

%Then, the Berezinian of $E$ is a line bundle which, when restricted to $U$, is canonically isomorphic to $\wedge^{\mathrm{top}}F^\circ$. 
 %Any other universal  Lie $\infty$-algebroid $E'$ of finite length of $\mathcal{F}$, being homotopically  equivalent to $E$ (see Theorem \ref{theo:onlyOne}), would . 
 %Even better, a
 
Item 2. of Proposition \ref{propirenee} states that the Berezian of any universal Lie $\infty$-algebroid of a solvable singular foliation $\mathcal{F}$ inherits a $\mathcal{F}$-module structure. %This representation is directly related to the $E_0$-module structure on the Berezinian by Equation \eqref{eqsmash}. 
We use this result  to define the modular class of the singular foliation $\mathcal{F}$.  

\begin{definition}\label{def:modclasssing}
%Let $\mathcal{F}$ be a solvable singular foliation on $M$.  %Let $\theta_{\mathfrak{U}}^{\mathcal{F}}$ be the unique element in the  universal foliated cohomology group $H^1_{\mathfrak{U}}(\mathcal{F})$ whose representent in $H^1(E)$ is the modular class $\theta^E$ of the Lie $\infty$-algebroid~$E$.
Let $\mathcal{F}$ be a solvable singular foliation and let $E$ be a universal Lie $\infty$-algebroid of $\mathcal{F}$ of finite length. Let $\theta_{\mathfrak{U}}^{\mathcal{F}}$ be the cohomology class induced  by the modular class of $E$ in the universal foliated cohomology of $\mathcal{F}$.
Then the modular class of $\mathcal{F}$ is the unique element $\theta^{\mathcal{F}}\in H^1_{\mathrm{dR}}(\mathcal{F})$ defined by:
\begin{equation}
\theta^\mathcal{F}=(\rho_{\mathcal{F}}^*)^{-1}\theta_{\mathfrak{U}}^{\mathcal{F}}
\end{equation}
% Then we define the \emph{modular class of $\mathcal{F}$} as the unique element $\theta^{\mathcal{F}}\in H^1_{\mathfrak{U}}(\mathcal{F})$ whose representent in $H^1(E)$ is the modular class $\theta^E$ of the Lie $\infty$-algebroid~$E$.
%Then the modular class $\theta^{\mathcal{F}}$ of $\mathcal{F}$ is the unique element of $H^1_{\mathrm{dR}}(\mathcal{F})$ defined by:
%\begin{equation}
%\theta^\mathcal{F}=(\rho_{\mathcal{F}}^*)^{-1}\theta^{\mathcal{F}}_{\mathfrak{U}}
%\end{equation}
% the modular class of any universal Lie $\infty$-algebroid of $\mathcal{F}$, and is noted $\theta_{\mathfrak{U}}^{\mathcal{F}}$.
%Let $\mathcal{F}$ be a singular foliation on $M$ admitting a resolution. Then the modular class of $\mathcal{F}$ is the modular class of any universal Lie $\infty$-algebroid of $\mathcal{F}$, and is noted $\theta_{\mathfrak{U}}^{\mathcal{F}}$.
\end{definition}

\begin{comment}
\begin{remarque}This result is the direct  application of Proposition \ref{hooke}. Indeed, when the singular foliation $\mathcal{F}=\Gamma(F)$ coincides with the sections of a regular foliation, then for every $k\geq0$, $\wedge^k \Gamma(F)=\Gamma\big(\wedge^k F\big)$. This implies that the foliated de Rham cohomology of $\mathcal{F}$ coincides with the horizontal de Rham cohomology of $F$. Moreover, the foliation Lie algebroid $A_F$ of $F$ is a universal Lie $\infty$-algebroid of $\mathcal{F}=\Gamma(F)$. Then $\theta_{\mathfrak{U}}^{\mathcal{F}}$ can be represented by $\theta^{A_F}$ and $\theta^{\mathcal{F}}=\theta^F$. Since the anchor map is the inclusion $\iota\colon A_F\to TM$, they coincide.
\end{remarque}
\end{comment}

%\begin{remarque}
%Notice that we required $\mathcal{F}$ to be solvable so that we know that it admits a universal Lie $\infty$-algebroid of finite length $E$, so that we can define its Berezinian, and its modular class.
%\end{remarque}

The choice of cohomology -- the foliated de Rham cohomology instead of the universal foliated cohomology -- comes from the fact that we want the modular cocycle to be applied to elements of $\mathcal{F}$, rather than sections of any of its universal Lie $\infty$-algebroids.  Doing so, we mimick the definition of the modular class for regular foliations, which is defined on the foliated de Rham cohomology.
However, in practice, it is often more useful to find a universal Lie $\infty$-algebroid $E$ and to work with the modular class of $E$, which then appears to be a particular representant of~$\theta^\mathcal{F}$. Indeed,  Definition \ref{def:modclasssing} makes sense only if the modular class does not depend on the choice of  the universal Lie $\infty$-algebroid:

\begin{proposition}\label{propimposs}
The modular class of the solvable singular foliation $\mathcal{F}$ does not depend on the choice of universal Lie $\infty$-algebroid from which it is defined.
\end{proposition}

\begin{proof}
Assume that $E$ (resp. $E'$) is a universal Lie $\infty$-algebroid of $\mathcal{F}$ of finite length, and let $\theta^E$ (resp. $\theta^{E'}$) be its modular class (see Definition \ref{defmodular}).
We know by Theorem \ref{theo:onlyOne} that  $E$ and $E'$ are homotopically equivalent, and that any two choices of homotopy equivalences are homotopic. Let us pick one, say $\Phi_{E,E'}\colon\Omega(E')\to\Omega(E)$, then by Proposition \ref{interestingprop} it has vanishing modular class. This implies that, at the level of cohomologies, the canonical isomorphism $\widetilde{\Phi}_{E,E'}$ induced by the homotopy equivalence $\Phi_{E,E'}$ intertwines the respective modular classes of $E$ and $E'$. %respective modular classes $\theta^E$ and $\theta^{E'}$ satisfy:
%\begin{equation}
%\theta^{E}=\Phi_{E,E'}\big(\theta^{E'}\big)
%\end{equation}
By definition of the universal foliated cohomology, 
these modular classes define the same element in~$H^1_{\mathfrak{U}}(\mathcal{F})$, that we note $\theta_{\mathfrak{U}}^{\mathcal{F}}$. Pulling back this cohomology class in the de Rham foliated cohomology using Proposition \ref{hooke} gives the modular class associated to $\mathcal{F}$.
\end{proof}
 %Pulling back this cohomology class in the de Rham foliated cohomology gives the modular class associated to $\mathcal{F}$: 

By construction, the vanishing or non-vanishing of the modular class of a singular foliation has a straightforward meaning:

\begin{proposition}\label{plop} Let $\mathcal{F}$ be a solvable singular foliation. Then
 $\theta^{\mathcal{F}}=0$ if and only if, for \emph{any} universal Lie $\infty$-algebroid of finite length $E$ of $\mathcal{F}$, the Berezinian $\mathrm{Ber}(E)$ is a trivial $\mathcal{F}$-module.
\end{proposition}

\begin{remarque}
The main motivation behind Definition \ref{def:modclasssing} is the following observation: defining the modular class of a singular foliation from that of a Lie algebroid from which it descends has several drawbacks. First, assume that one finds several different, non-isomorphic, Lie algebroids covering the same singular foliation: then we do not have a unique way of defining a modular class for the latter. This problem is solved in the Lie $\infty$-algebroid realm by Proposition \ref{propimposs}. Moreover, even for Lie algebroids for which the anchor has constant rank on an open dense subset, the modular class of the Lie algebroid 
 may not coincide with the modular class of the induced regular foliation (see Examples \ref{exaeuler} and \ref{charaa}). This is somewhat problematic.
  Even worse, there are examples of singular foliations which we do not know if they descend from a Lie algebroid or not (this is still a conjecture \cite{androulidakisSmoothnessHolonomyCovers2013}), but of which we know some universal Lie $\infty$-algebroids (see Example \ref{quadra}). %These various cases, together with the next example, show that
Thus, relying on Lie algebroids raises many problems and justifies that we use the notion of universal Lie $\infty$-algebroids to define modular class of singular foliations. 
\end{remarque}

We will now understand the relationship between the modular class of singular foliations and that of regular foliations.
Recall that the regular leaves of a solvable singular foliation have the same, maximal dimension (see Proposition 2.5 in \cite{laurent-gengouxUniversalLieInfinityAlgebroid2020}) on a dense open set $U\subset M$, hence forming a regular foliation $F$. 
From Example \ref{ex4} we know that the conormal bundle $F^\circ$ may, however, not be extendable to the singular leaves $M\backslash U$. If it is extendable, we call $\mathcal{F}^\circ$ the vector bundle on $M$ whose restriction to the open dense subset $U$ is $F^\circ$. The action of $\mathcal{F}\big|_U$ on $F^\circ$ extends at the singularities as an action of $\mathcal{F}$ on $\mathcal{F}^\circ$.
 By construction, the line bundle $\wedge^{\mathrm{top}}\mathcal{F}^\circ$ restricts to $\wedge^{\mathrm{top}}F^\circ$ on $U$. Moreover, it inherits a $\mathcal{F}$-module structure from that of $\mathcal{F}^\circ$, which extends that of $\mathcal{F}\big|_U$ on $\wedge^{\mathrm{top}}F^\circ$.
 By a similar argument as in item 1. of Proposition \ref{propirenee}, the line bundle $\wedge^{\mathrm{top}}\mathcal{F}^\circ$ is isomorphic, as a vector bundle, to the Berezinian of \emph{any} universal Lie $\infty$-algebroid of finite length $E$ of $\mathcal{F}$. While item 2. of Proposition \ref{propirenee} states that the restriction on $U$ of these two line bundles are isomorphic as $\mathcal{F}\big|_U$-modules, \emph{it does not mean} that over $M$ they are isomorphic as $\mathcal{F}$-modules, precisely because the action of $\mathcal{F}\big|_U$ coincides with that of the regular foliation $F$ and forgets everything about the behavior of the vector fields at the singularities. To better understand the situation, let us explore the following example:
 
\begin{example}\label{humok}

Let $M=\mathbb{R}^n$ with standard coordinates $x_1,\ldots, x_n$. Let $f$ be a smooth function on $\mathbb{R}$ vanishing only at $0$ and let $g$ be a smooth function on $\mathbb{R}^{n-1}$ vanishing only at the origin. We use them to define two vector fields:
\begin{equation*}
X_1=g(x_2,\ldots, x_n)\partial_{x_1}\qquad \text{and} \qquad X_2=f(x_1)\partial_{x_1}
\end{equation*}
Let $\mathcal{F}_i$ be the singular foliation generated by the vector field $X_i$ ($i=1,2$). The regular foliation $F_1$ induced by $X_1$ is defined over the open dense subset  $U_1=\mathbb{R}^n-\{(x_1,\ldots, x_n)\,|\,x_2=\ldots=x_n=0\}$  while the regular foliation $F_2$ induced by $X_2$ is defined over $U_2=\mathbb{R}^n-\{(x_1,\ldots, x_n)\,|\,x_1=0\}$. Both regular leaves consist of parallel lines along the $x_1$-direction, but in the first case the singular leaves are the points forming the $x_1$-axis, while in the second case  the singular leaves consist of points on the $x_1=0$ plane

The annihilator bundle $F^\circ_i$ is the rank $n-1$ subbundle of $T^*\mathbb{R}^n\big|_{U_i}$spanned by the one-forms $dx_2,dx_3,\ldots, dx_n$, and the determinant line bundle $\wedge^{n-1}F^\circ_i$  is the rank one subbundle of $\wedge^{n-1}T^*\mathbb{R}^n\big|_{U_i}$  generated by the $n-1$-form $dx_2\wedge\ldots\wedge dx_n$. 
The vector bundles $F^\circ_1$ and $\wedge^{n-1}F^\circ_1$ (resp. $F_2$ and $\wedge^{n-1}F^\circ_2$) straightforwardly extend along the $x_1$-axis  (resp. at the plane $x_1=0$), and so does the section $dx_2\wedge\ldots\wedge dx_n$. We denote by $L$ the rank one trivial subbundle  of $\wedge^{n-1} T^*\mathbb{R}^n$ generated by the global section $dx_2\wedge\ldots\wedge dx_n$ (over the whole space $\mathbb{R}^n$ then).
The vector fields $X_1$ and $X_2$ act on $L$ via the Bott connection (see Equation \eqref{botte}) and their action on the global section $dx_2\wedge\ldots\wedge dx_n$ is trivial. Thus,  $L$ is both a trivial $\mathcal{F}_1$ and a trivial $\mathcal{F}_2$ line bundle. However, we will now see that these singular foliations have different modular classes.

A universal Lie $\infty$-algebroid of $\mathcal{F}_i$ can be chosen to be the rank one trivial Lie algebroid $A=\mathbb{R}^n\times \mathbb{R}$ such that the anchor map $\rho_i$ sends the constant section $\textbf{1}$ to $X_i$. By construction it is injective on sections so there is no additional term in the resolution.
The Berezinian of $A$ is $\mathrm{Ber}(A)=\wedge^nT^*\mathbb{R}^n\otimes A$, and admits as global section the following form: $\Omega=dx_1\wedge \ldots\wedge dx_n\otimes\textbf{1}$. The action of the constant section $\textbf{1}\in\Gamma(A)$ on  $\Omega$ then gives:
\begin{equation}\label{eq1}
\nabla_\textbf{1}(\Omega)=\mathcal{L}_{X_i}(dx_1\wedge \ldots\wedge dx_n)\otimes \textbf{1}=\mathrm{div}(X_i)\,\Omega
\end{equation}
In the first case, $\mathrm{div}(X_1)=0$, while in the second case, $\mathrm{div}(X_2)=f'(x_2)$. So, in the first case, the modular class $\theta^{\mathcal{F}_1}=0$. However, for the second case, the modular 1-form $\theta_\Omega$ from Equation~\eqref{eq1} is $\theta_\Omega(\textbf{1})=f'(x_1)$. Although it is not zero as a cochain element, it can be written as a $d_A$-exact one-form on the open dense subset $U_2$:
\begin{equation}
\theta_\Omega\big|_{U_2}=\frac{1}{2}d_A\mathrm{ln}(f^2)
\end{equation}

Then, $\theta_\Omega$ is \emph{not} exact on the whole of $\mathbb{R}^n$ (or equivalently, on any neighborhood of the hyperplane $x_1=0$) because the function $x_1\mapsto \frac{1}{2}\mathrm{ln}\big(f(x_1)^2\big)$ diverges at $x_1=0$. It implies in turn that the modular class $\theta^{\mathcal{F}_2}$ is \emph{not} zero (at least when $f'(0)\neq0$, see Example \ref{exaeuler}).
To conclude: $\mathrm{Ber}(A)$ is a trivial $\mathcal{F}_1$-module, but not a trivial  $\mathcal{F}_2$-module, while in both cases $L$ is a trivial module.
We interpret this observation as follows: as line bundles, $L$ and  $\mathrm{Ber}(A)$ are isomorphic. However, as $\mathcal{F}_1$-modules they are isomorphic, while as $\mathcal{F}_2$-modules they are not. This statement extends to the Berezinian of any universal Lie $\infty$-algebroid of finite length of the singular foliations $\mathcal{F}_i$.
 %he canonical isomorphism (of vector bundle) between $L$ and  $\mathrm{Ber}(A)$ is a morphism of $\mathcal{F}_2$-modules, while   %Indeed, assume that it is exact, then there exists a smooth function $h$ such that $\theta_\Omega=d_ah$. In particular, over the open and dense subset $U_1$, we have $d_A\big(h-\frac{1}{2}\mathrm{ln}(f^2)\big)=0$. By density it is zero everywhere so $h=\frac{1}{2}\mathrm{ln}(f^2)$ (modulo some constant).
%Then the modular class of the universal Lie $\infty$-algebroid $A$ of $\mathcal{F}$ is zero.
%It implies that on the punctured space the modular class of the regular foliation induced by the vector field $X$ vanishes, confirming the existence of an $X$-invariant transverse volume form, which we already found to be $dy\wedge dz$. It moreover confirms that this volume form can be extended to the origin, as a global section of the extension at the origin of the line bundle $\wedge^2F^\circ$.
\end{example}

From Example \ref{humok}, we understand that, when the regular foliation $F$ induced from the singular foliation $\mathcal{F}$ has a vanishing modular class (as a regular foliation, on the open dense subset of $M$ made of regular leaves), the modular class of the singular foliation $\mathcal{F}$ tells us to what extent item 2. of Proposition \ref{propirenee} extends to the singular leaves. More precisely,  
%In particular, %the same discussion applies for $L_\mathcal{F}$, so that we obtain the following result, that extends the notion of modular class of regular foliations that we saw in Section \ref{sec:boot}:
%Since on regular points, the singular foliation admits a conormal bundle $\nu(\mathcal{F})$, and since on those points, the Berezinian of any universal Lie $\infty$-algebroid is canonically isomorphic to this canonical bundle, we reach the following result, which extends the notion of modular class of regular foliations:
assume that the modular class of this regular foliation is zero on $U$, and that %let $\mathcal{F}^\circ$ be the unique vector bundle which restricts to 
the vector bundle $F^\circ$ extends at singularities. We denote by $\mathcal{F}^\circ$ this vector bundle and by $\wedge^{\mathrm{top}}\mathcal{F}^\circ$ the unique line bundle which restricts to the determinant line bundle $\wedge^{\mathrm{top}}F^\circ$ on $U$. Then, drawing on Proposition \ref{plop}, we have the following statement:

\begin{proposition}\label{prop:ok} Let $\mathcal{F}$ be a solvable singular foliation. Assume that the regular foliation $F$ induced by $\mathcal{F}$ admits a global invariant transverse measure, and that $F^\circ$ extends at singularities. Then, 
$\theta^\mathcal{F}=0$ if and only if there exists  an isomorphism of $\mathcal{F}$-modules  between $\wedge^{\mathrm{top}}\mathcal{F}^\circ$ and the Berezinian of \emph{any} universal Lie $\infty$-algebroid of finite length of $\mathcal{F}$. % If the line bundle $L_\mathcal{F}=\wedge^{\mathrm{top}}$ is also defined on singular leaves then:
%\begin{enumerate}
%\item the action of $\mathcal{F}$ on $L_\mathcal{F}$ canonically extends at singularities, and
%\item
% the modular class $\theta^\mathcal{F}$  measures the obstruction to the existence of an invariant global section of $L_\mathcal{F}$ or, equivalenty, the fact that $L_{\mathcal{F}}$ is a trivial $\mathcal{F}$-module. % satisfies the following identity:
%\begin{equation}
%\nabla^\mathcal{F}_X(\Omega)=\theta^\mathcal{F}(X)\,\Omega
%\end{equation}
%for any $X\in\mathcal{F}$.
%\end{enumerate}
\end{proposition}

%\textbf{Not even true car pour gln et son + euler vf on a le line bundle qui est trivial donc la section global existe c'est 1, tandis que le Berezinien dans le second cas n'admet PAS de section globale.}

%\textbf{EN FAIT ! L'action du feuilletage singulier sur la section constante $\textbf{1}$ n'est pas forcement triviale (au sens des rep.) !!!!! cf p36 du cours sur la quantization de la modular class. Donc le fait qu'un fibre en droite est trivial au sens geometrique n'induit PAS qu'il est additionellement trivial au sens de l'action de F sur lui. Mettre l'accent sur la trivialite de l'action plutot que l'invariance qui est une facon de le dire.}

%\begin{remarque}Notice that the existence of such a section does not imply that the vector bundle $L_\mathcal{F}$ is trivial, for the section can vanish at singularities. On the converse, if the vector bundle is trivial, then any constant section is invariant and thus the modular class should vanish.

%In other words, in the case where $F$ is unimodular, and $F^\circ$ is extendable at the singularities, the modular class $\theta^\mathcal{F}$ measures the obstruction of extending item 

 We assumed that the modular class of the regular foliation $F$ on $U$ vanishes, because otherwise that of the singular foliation $\mathcal{F}$ will not vanish either, not telling us anything that is not already captured by the modular class  of $F$.
For the same reason, if the vector bundle $F^\circ$ cannot extend at singular leaves, the modular class of the singular foliation $\mathcal{F}$ still does exist, but does not say anything on $\mathcal{F}$ that is not already known about $F$. On the contrary, if those two conditions are met, Proposition \ref{prop:ok} establishes that the modular class of the singular foliation $\mathcal{F}$ is an obstruction to the existence of an invariant global section of a particular family of line bundles -- the Berezinians of the universal Lie $\infty$-algebroids of finite length -- which, on the regular leaves, can be interpreted as the existence of an invariant transverse measure to $F$.

 The use of universal Lie $\infty$-algebroids to define modular classes of singular foliations illustrates how to extend other characteristic classes or geometric notions from the regular foliations realm to the singular foliations one. For prospective researches, it would be important to investigate the behaviour of the modular class (or other characteristic classes) under Hausdorff Morita equivalences of singular foliations \cite{garmendiaHausdorffMoritaEquivalence2019}. Also, based on the idea that the modular class of a Lie algebroid can be seen as an obstruction to the existence of a measure on the differentiable stack associated to this Lie algebroid \cite{weinsteinVolumeDifferentiableStack2009, crainicMeasuresDifferentiableStacks2020}, it would be certainly fruitful to investigate the relationship between the modular class of a singular foliation and the existence of an invariant measure on its associated holonomy groupoid \cite{androulidakisHolonomyGroupoidSingular2009}.

\section{Examples}\label{exxxon}

\begin{comment}
\subsection{Two examples}\label{expremis}

Take a family of examples: sing fol generated by one vector field $X$. Then the modular class is just the divergence of $X$. In some cases it is exact on the whole of $R^n$ (case $X=x\partial_x-y\partial_y$ qui admet comme forme normale $ydx+xdy$) on sometimes it is not (euler vf)...what does it mean regarding the modular class???

Le cas general ou $X$ est donne par $\sum_if(x_i)\partial_i$ avec la meme fonction $f(x)$ qui s'annule uniquement en $x=0$, alors $Div(X)$ peut etre obtenu a partir d'une forme exacte $dg$ suivantes:
$$g=\frac{1}{2}\mathrm{ln}\left(\Pi_jf(x_j)^2\right)=\sum_j\mathrm{ln}(|f(x_j)|)$$ 
%ou bien encore:
%$$g=\frac{1}{2}\mathrm{ln}\left(\sum_jf(x_j)^2\right)$$
% $\sum_if(x_i)\partial_ig=\sum_i f'(x_i)$. On peut rendre un a un tous les $x_j=0$ sauf $x_i$ pour obtenir $f(x_i)\partial_ig(0,...,0,x_i,0,...,0)=(n-1)f'(0)+f'(x_i)$. On a donc $\partial_ig=\frac{f'(x_i)}{f(x_i)}=\frac{f(x_i)f'(x_i)}{(f(x_i))^2}$ donc $dg=\sum_j\frac{f(x_j)f'(x_j)}{(f(x_j))^2}dx_j= \sum_jd\mathrm{ln}(f(x_j)^2)=d\mathrm{ln}\big(\Pi_jf(x_j)^2\big)$. Ce cas pratique empeche d'avoir une forme modulaire nulle a l'origine

\end{comment}

The first example \ref{exaeuler} and the second example \ref{charaa} are raw calculations of modular classes of two solvable singular foliations: the Euler vector field and the characteristic foliation of a Poisson manifold. In both cases, the annihilator bundle $F^\circ$ of the induced regular foliation $F$ \emph{cannot} be extended at the singular leaves. The modular classes of these singular foliations then do not bring any new knowledge that is not already captured by the modular class of the induced regular foliations. Examples \ref{quadra} and \ref{exempleglau}  study solvable singular foliations whose conormal bundles can be extended at the origin. In both cases the modular class vanishes, proving that the line bundle $\wedge^{\mathrm{top}}F^\circ$ is a trivial module of the singular foliation.
Example \ref{exampleson} presents a case of a  solvable singular foliation whose conormal bundle \emph{cannot} be extended at the origin, but if we add the Euler vector field, it can. However, even if the action of the newly defined singular foliation on the determinant line bundle $\wedge^{\mathrm{top}}F^\circ$ is trivial, we show that its modular class is not.  This case should then be interpreted under the light of Proposition \ref{prop:ok}.

\subsection{The Euler vector field} \label{exaeuler} %\textbf{est-ce qu'on veut vraiment voir l'addition du champ d'euler a so(n) ou pas juste parler du champ d'euler ? et dans ce cas rappatrier l'exemple \eqref{exaeuler}. On peut reduire juste a un paragraphe l'ajout du champs d'euler a so(n). Au final ca veut dire que l'action de euler est triviale SAUF EN ZERO.}

%The present example shows a totally opposite situation than Example \ref{expremis}: although the line bundle $\wedge^{\mathrm{top}}F^\circ$ will be trivializable, it will not be extendable at the origin, nor will the modular class vanish.
On $\mathbb{R}^n$, for $n\geq2$, the Euler vector field $\epsilon=\sum_{i=1}^nx_i\partial_{x_i}$ defines a singular foliation $\mathcal{F}_\epsilon$ whose leaves are straight lines (regular leaves) escaping the origin (singular leaf).  The punctured space $\mathbb{R}^n\backslash\{0\}$ is the union of regular leaves, whose transversals consist of the $(n-1)$-spheres centered at $0$ (and of any non-zero radius). Let us denote by $F$ the regular distribution induced by $\mathcal{F}_\epsilon$ on the punctured space $\mathbb{R}^n\backslash\{0\}$. %, and let $F^\circ$ be its associated  conormal bundle. 
Sections of the annihilator vector bundle $F^\circ\subset T^*(\mathbb{R}^n\backslash\{0\})$ are generated by elements of the form $x_idx_j-x_jdx_i$ for $1\leq i<j\leq n$. The codistribution they generate consists of concentric spheres, so the annihilator bundle does not extend at the origin.

 A distinguished global section of the line bundle $\wedge^{\mathrm{top}}F^\circ$ on the punctured space is the $n-1$-dimensional spherical volume form:
\begin{equation}
\omega_{F^\circ}=\sum_{i=1}^n(-1)^{i-1}x_idx_1\wedge\ldots\wedge dx_{i-1}\wedge dx_{i+1}\wedge\ldots\wedge dx_n ,
\end{equation}
showing that $\wedge^{\mathrm{top}}F^\circ$ is a trivializable line bundle over the punctured space. The action of the regular foliation $F$ on $\omega_{F^\circ}$ is made through the Bott connection defined as in Equation \eqref{definitionbottdual}. Since the restriction of  $\epsilon$ to the punctured space generates all the sections of $F$, it is sufficient to compute the action of the Bott connection on $\epsilon$, so we obtain:
\begin{equation}\label{eqezio}
\nabla^{\mathrm{Bott}}_\epsilon(\omega_{F^\circ})=n\,\omega_{F^\circ}
\end{equation}

One can straightforwardly  read the modular 1-form $\theta_{\omega_{F^\circ}}\in\Omega^1(F)$ from this equation; evaluated on the Euler vector field it reads: $\theta_{\omega_{F^\circ}}(\epsilon)=n$. Although the modular 1-form $\theta_{\omega_{F^\circ}}$ is not zero as a cochain element of $\Omega^1(F)$, it is exact since it can be written as: %\footnote{Notice that another choice of function could have been the following: 
%\begin{equation}
%\theta_{\omega_{F^\circ}}=\frac{n}{2}d_{F}\big(\mathrm{ln}(x_1^2+\ldots+x_n^2)\big)
%\end{equation}}:
\begin{equation}\label{super}
%\theta_\Omega\Big|_{\mathbb{R}^n\backslash\{0\}}=\frac{n}{2}d_E\big(\mathrm{ln}(x_1^2+\ldots+x_n^2)\big)
%\theta_{\omega_{F^\circ}}=\frac{1}{2}d_F\big(\mathrm{ln}(x_1^2\cdots x_n^2)\big)
\theta_{\omega_{F^\circ}}=n\,d_{F}\left(\mathrm{ln}\sqrt{x_1^2+\ldots+x_n^2}\right)
\end{equation}
It is well defined because the sum in the square root never vanishes on the punctured space $\mathbb{R}^n\backslash\{0\}$.
Then, the modular class of the regular foliation $F$ induced by the Euler vector field vanishes. %vector bundle $E\big|_{\mathbb{R}^n\backslash\{0\}}$ vanishes,
It means that there exists a transverse measure invariant under the action of the Euler vector field, which can be chosen to be: %\footnote{Another choice could be the following: \begin{equation}\omega_{\mathrm{inv}}=\frac{1}{\sqrt{x_1^2+\cdots +x_n^2}}\omega_{F^\circ}\end{equation}}:
\begin{equation}\label{super2}
\omega_{\mathrm{inv}}=\frac{1}{\sqrt{x_1^2+\cdots +x_n^2}}\omega_{F^\circ}
%\omega_{\mathrm{inv}}=\frac{1}{x_1\cdots x_n}\omega_{F^\circ}=\sum_{i=1}^n(-1)^{i-1}\frac{dx_1}{x_1}\wedge\ldots\wedge \frac{dx_{i-1}}{x_{i-1}}\wedge \frac{dx_{i+1}}{x_{i+1}}\wedge\ldots\wedge \frac{dx_n}{x_n}
\end{equation}
 One can check that the action of the Bott connection is trivial on $\omega_{\mathrm{inv}}$, showing that $\wedge^{\mathrm{top}}F^\circ$ is a trivial $F$-module.

%Its extension to the origin as a  line bundle over $\mathbb{R}^n$ is denoted  $L_\mathcal{F}$. The section $\omega_{F^\circ}$ also extends to the origin as a section of the line bundle $L_\mathcal{F}$, although vanishing at the origin.
Now let us turn to the modular class of the singular foliation $\mathcal{F}_\epsilon$.
A universal Lie $\infty$-algebroid of $\mathcal{F}_\epsilon$ is the rank 1 trivial vector bundle $E=\mathbb{R}^n\times \mathbb{R}$, whose anchor map $\rho$ sends the constant section $\textbf{1}$ to the Euler vector field $\epsilon$.   The Berezinian $\mathrm{Ber}(E)$ is a trivial vector bundle and by Proposition \ref{propirenee}, its restriction to the  punctured space can be canonically identified with $\wedge^{\mathrm{top}}F^\circ$. The latter line bundle does not extend to the origin however, so the identification cannot be extended there.
% The action of  $\mathcal{F}$ on $L_\mathcal{F}$  can then alternatively be understood from the action of $E$ on its Berezinian. 
 A global non-vanishing section of the Berezinian is $\Omega=\omega\otimes \textbf{1}$, where $\omega=dx_1\wedge\ldots dx_n$ is the standard volume form on $\mathbb{R}^n$. Since the Lie algebroid bracket on $E$ is zero, the action of $E$ on $\Omega$ is only given by the Lie derivative of the Euler vector field on $\omega$:
\begin{equation}\label{eqauditore}
\nabla_{\textbf{1}}(\Omega)=\mathcal{L}_{\epsilon}(\omega)\otimes \textbf{1}=\mathrm{div}(\epsilon)\,\Omega=n\,\Omega
\end{equation}

One obtains an equation similar to Equation \eqref{eqezio}. Then we know that the modular 1-form $\theta_\Omega\in\Omega^1(E)$ satisfying $\theta_\Omega(\textbf{1})=n$ is exact when restricted to the punctured space, and
\begin{equation}
\theta_\Omega\big|_{\mathbb{R}^n\backslash\{0\}}= \theta_{\omega_{F^\circ}}%n\, d_{E}\left(\mathrm{ln}\sqrt{x_1^2+\ldots+x_n^2}\right)
\end{equation}
However, as seen from Equation \eqref{super}, the right-hand side cannot be extended at the origin. One can show that in fact, there is no smooth function $g$ such that $\theta_\Omega=d_Eg$ on any neighborhood of the origin. Indeed, if it were the case, by Equation \eqref{ancre} we would have:
\begin{equation}\label{espresso}\theta_\Omega(\textbf{1})=d_Eg(\textbf{1})=\epsilon(g)=\sum_{i=1}^nx_i\partial_ig\end{equation} Since the derivative of $g$ is continuous at the origin it is bounded and then the right-hand side of Equation \eqref{espresso} vanishes. However, by Equation \eqref{eqauditore}, the left-hand side of Equation \eqref{espresso} should be equal to $n$  everywhere, in particular at the origin. This is a contradiction.%for whatever choice of non-zero although small value of the $x_i$. Then by continuity it should be $n$ as well at the origin, which is not the case.

Then, the modular 1-form $\theta_\Omega$ is not an exact 1-form on any neighborhood of the origin, and so we conclude that the modular class of the universal Lie $\infty$-algebroid $E$ is not zero. This implies in turn that the modular class of the singular foliation $\mathcal{F}_\epsilon$ is not zero. This not only means that the Berezinian of $E$ is \emph{not} a trivial $\mathcal{F}_\epsilon$-module, but that this statement holds for the Berezinian of \emph{any} universal Lie $\infty$-algebroid of $\mathcal{F}_\epsilon$. This result is clearly related to the behavior of the Euler vector field at the origin. Since in any case the conormal bundle $F^\circ$ could not be extended at the origin, the non-vanishing modular class of $\mathcal{F}_\epsilon$ does not tell us anything more about the regular foliation $F$ that we already knew from its own modular class. %, because on the punctured space the res 

\subsection{The characteristic foliation of a degenerate Poisson manifold}\label{charaa}

Poisson manifolds form a nice class of foliated manifolds. Indeed, for any Poisson manifold $M$ with Poisson bivector $\pi\in\mathfrak{X}^2(M)$, the cotangent bundle $T^*M$ is a Lie algebroid -- called the \emph{cotangent Lie algebroid} -- with anchor the Poisson bivector $\pi^{\musSharp}:T^*M\longrightarrow TM$, and bracket the following:
\begin{equation}
[\alpha,\beta]_{\pi}=\mathcal{L}_{\pi^{\musSharp}(\alpha)}(\beta)-\mathcal{L}_{\pi^{\musSharp}(\beta)}(\alpha)-d\big(\pi(\alpha,\beta)\big)
\end{equation}
for any two differential 1-forms $\alpha,\beta\in\Omega^1(M)$. Thus, the image of $\Omega^1(M)$ through the anchor map defines a singular foliation $\mathcal{F}_\pi$, called the \emph{characteristic foliation} of the Poisson manifold. In particular, a property of such a foliation is that its leaves are symplectic manifolds, hence necessarily even dimensional.

For the present example -- due to Vladimir Rubstov --  let us pick up $M=\mathbb{R}^3$ and the following Poisson bivector:
\begin{equation}
\pi=(x\partial_x+y\partial_y)\wedge \partial_z
\end{equation}
The cotangent bundle $T^*\mathbb{R}^3$ is trivial and admit three canonical generators, the constant sections $dx,dy,dz$. The Lie algebroid structure on $E_0=T^*\mathbb{R}^3$ is such that the anchor map reads:
\begin{equation}\label{bracos1}
X_x=\pi^{\musSharp}(dx)=x\partial_z,\quad X_y=\pi^{\musSharp}(dy)=y\partial_z\quad \text{and}\quad X_z=\pi^{\musSharp}(dz)=-x\partial_x-y\partial_y
\end{equation}
while the Lie bracket reads, because of the identity $[df,dg]_\pi=d\{f,g\}$:
\begin{equation}\label{bracos2}
[dx,dy]_\pi=0, \quad [dx,dz]_\pi=dx\quad \text{and}\quad   [dy,dz]_\pi=dy
\end{equation}
The characteristic foliation $\mathcal{F}_\pi$ induced by the three vector fields $X_x, X_y,X_z$ consists of the union of 2-dimensional half-planes escaping radially from the vertical axis $(x=0,y=0,z)$, itself constituted of the singular leaves: points.

The union of regular leaves forms an open dense subset $U=\mathbb{R}^3\backslash\{z-\text{axis}\}$, on which $\mathcal{F}_\pi$ induces a regular foliation $F$. The conormal bundle is  the rank one subbundle $F^\circ$ of $T^*\mathbb{R}^3$ generated by the differential form $ydx-xdy$, so it does not extend at the singular leaves. Another way of seeing this is to notice that the transversals to the regular leaves consist of concentric circles around the vertical axis. The determinant line bundle $\wedge^{\mathrm{top}}F^\circ$, being of rank one, coincides with $F^\circ$, so a transversal volume form to the regular leaves is given by:
\begin{equation}
\omega_{F^\circ}=ydx-xdy
\end{equation}
The action of $X_x, X_y$ and $X_z$ on this volume form is computed via the Lie derivative as in Equation \eqref{definitionbottdual}. One can check that only $X_z$ has a non trivial action on $\omega_{F^\circ}$ since
\begin{equation}
\nabla^{\mathrm{Bott}}_{X_z}(\omega_{F^\circ})=-2\,\omega_{F^\circ}\end{equation}

The modular one-form $\theta_{\omega_{F^\circ}}$ can be read directly from this equation:
\begin{equation}
\theta_{\omega_{F^\circ}}(X_x)=\theta_{\omega_{F^\circ}}(X_y)=0\quad \text{and}\quad  \theta_{\omega_{F^\circ}}(X_z)=-2
\end{equation}
It turns out to be exact on $U$, because one can write:
\begin{equation}
\theta_{\omega_{F^\circ}}=d_F\left(\mathrm{ln}(x^2+y^2)\right)
\end{equation}
From this we deduce that the modular class of the regular foliation $F$ is zero. Indeed, the latter admits an invariant transverse measure to the regular leaves given by:
\begin{equation}
\omega_{\mathrm{inv}}=\frac{1}{x^2+y^2}(ydx-xdy) 
\end{equation}
Notice that it does not extend to the vertical axis because the function at the denominator diverges. 

Now let us turn to the modular class of the singular foliation $\mathcal{F}_\pi$.
The kernel of the anchor map $\pi^{\musSharp}$ is generated by the differential one-form $ydx-xdy$.
A universal Lie $\infty$-algebroid of the characteristic foliation $\mathcal{F}_\pi$ can then be built on the following geometric resolution:
\begin{center}
\begin{tikzcd}[column sep=0.7cm,row sep=0.4cm]
0\ar[r]&\Gamma(\mathbb{R}^3\times \mathbb{R})\ar[r,"{\delta}"]&\Omega^1(\mathbb{R}^3)\ar[r,"\pi^{\musSharp}"]&\mathcal{F}_\pi
\end{tikzcd}
\end{center}
%so that $E_{0}=T^*\mathbb{R}^3$ and $E_{-1}$ is the trivial line bundle $\mathbb{R^3}\times \mathbb{R}$. The map $\delta$ is such that 
where the vector bundle morphism $\delta$ is such that it sends the constant section $\textbf{1}$ of the trivial line bundle $E_{-1}=\mathbb{R}^3\times \mathbb{R}$ to the one-form $ydx-xdy$. Then, we indeed have $\mathrm{Im}(\delta)=\mathrm{Ker}(\pi^{\musSharp})$. The 2-bracket $l_2$ on $E_0$ coincides with the Lie algebroid bracket $[\,.\,,.\,]_\pi$. The 2-bracket $l_2$ between a section of $E_0=T^*\mathbb{R}^3$ and a section of $E_{-1}$ should satisfy Equation \eqref{leibnizzzz}, for $l_1=\delta$. So, we deduce that:
\begin{equation}\label{actionEE}
l_2(dx,\textbf{1})=0,\quad l_2(dy,\textbf{1})=0\quad \text{and}\quad l_2(dz,\textbf{1})=-2\,\textbf{1}
\end{equation}
The Berezinian of $E$ is given by:
\begin{equation}\mathrm{Ber}(E)=\wedge^3T^*\mathbb{R}^3\otimes \wedge^3T^*\mathbb{R}^3\otimes E_{-1}^*\end{equation}
where the second slot corresponds to $\wedge^3 E_0$.
The Berezinian admits a global section given by $\Omega=dx\wedge dy\wedge dz\otimes dx\wedge dy\wedge dz\otimes \textbf{1}^*$. The constant section $\textbf{1}^*\in \Gamma(E_{-1}^*)$ is the section dual to $\textbf{1}$.

We will compute the modular class of $E$ by using Formula \eqref{gluttony4}, i.e. by computing the action of $E_0$ on each of the three terms of the tensor product in $\mathrm{Ber}(E)$. Having Formulas~\eqref{bracos1}, \eqref{bracos2} and \eqref{actionEE} at hands, one straightforwardly notices that  the only a priori  non-trivial action on $\mathrm{Ber}(E)$ is that of $dz$.
To compute the explicit action of the element $dz$ on $E^*_{-1}$, we will in fact compute it on $E_{-1}$ and revert the sign, as is explained in Equations~\eqref{bracos3} and~\eqref{bracos4}.
The action of $dz$ on the first term $\wedge^3T^*\mathbb{R}^3$ is given by the divergence of $X_z$, which is:
\begin{equation}
\mathrm{div}(X_z)=-2
\end{equation}
The action of $dz$ on $\wedge^3E_0$ is given via the Lie algebroid bracket \eqref{bracos2} and gives another $-2$ contribution.
%Since there are two terms $\wedge^3T^*\mathbb{R}^3$ the contribution of the first two terms to the modular class is $-4$.
The action of $dz$ on the constant section $\textbf{1}$ of $E_{-1}$ gives $-2\,\textbf{1}$. Taking the opposite sign, we deduce that  the section $dz$ acts on $E_{-1}^*$ as a multiplication by $+2$. Summing all the numbers, we obtain that the modular one-form $\theta_\Omega\in\Omega^1(E)$ satisfies:
\begin{equation}
\theta_\Omega(X_x)=\theta_\Omega(X_y)=0\quad \text{and}\quad  \theta_\Omega(X_z)=-2
\end{equation}
%\theta^{\mathcal{F}_\pi}
%Notice that, because of the compensation between the contribution of $E_{-1}^*$ and $\wedge^3E_0$, the modular one form of the singular foliation $\mathcal{F}_\pi$ coincides with that of the Poisson manifold $(\mathbb{R}^3,\pi)$ -- the later being half that of the associated cotangent Lie algebroid  \cite{evensTransverseMeasuresModular1999}.
Following the same lines of argument as in the second part of  Example \ref{exaeuler}, we deduce that the modular one-form is exact on $U$, but not on any neighborhood of the singular leaves, so not on $\mathbb{R}^3$ as a whole. It means that the modular class of the singular foliation of $\mathcal{F}_\pi$ is not zero, i.e. that the Berezinian of \emph{any} universal Lie $\infty$-algebroid is \emph{not} a trivial $\mathcal{F}_\pi$-module. %The fact that the conormal bundle $F^\circ$ does not extend to the singularities entails that we do not obtain more informations on the regular foliation $F$ from the modular class of $\mathcal{F}_\epsilon$ that we already knew from the modular class of $F$.

\begin{comment}

Since $X_z$ is minus the Euler vector field on the plane $z=0$, we know that the modular one-form $\theta_\omega$ can be shown to be an exact form outside of the vertical axis:
\begin{equation}
\theta_\Omega=d_E\big(\mathrm{ln}(x^2+y^2)\big)
\end{equation}
The function in parenthesis indeed explodes when either $x=y=0$ thus the modular cocycle $\theta^{\mathcal{F}_\pi}$ is not zero in $H^1(\mathbb{R}^3)$.  One can indeed show as in Example \ref{exaeuler} that the modular one-form $\theta_\Omega$ cannot be exact on any neighborhood of the vertical axis. \textbf{so what en fait ?}

\end{comment}

\subsection{Vector fields that quadratically vanish at 0}\label{quadra}
This example illustrates that a solvable singular foliation of which we do not know if it descends from a Lie algebroid, can nonetheless be attributed a modular class because we know a universal Lie $\infty$-algebroid associated to it. Let $\mathcal{F}$ be the sheaf of vector fields on $\mathbb{R}^2$ that quadratically vanish at the origin. A typical element of $\mathcal{F}$ can be written as $X=(ax^2+2bxy+cy^2)(u\,\partial_x+v\,\partial_y)$, where $a,b,c$ are functions that are possibly locally defined, and $u,v$ are smooth functions. The leaves of $\mathcal{F}$ consist of the punctured plane (the regular leaf) and the origin (singular leaf). Since on the punctured plane the conormal bundle is the zero vector bundle, it is extendable at the origin and is denoted $\mathcal{F}^\circ$. We deduce that the determinant line bundle $\wedge^{\mathrm{top}}\mathcal{F}^\circ$ actually involves the zeroth power of $\mathcal{F}^\circ$ and is the  trivial vector bundle: $\wedge^{0}\mathcal{F}^\circ=\mathbb{R}^n\times \mathbb{R}$.  The sections of this line bundle consists of the algebra of smooth functions on $\mathbb{R}^n$ and any constant function is a global nowhere-vanishing section on which any vector field acts trivially. It means that $\wedge^{0}\mathcal{F}^\circ$ is a trivial $\mathcal{F}$-module.
 %Since on the punctured plane the conormal bundle is the zero vector bundle, we deduce  that  the line bundle $L_{\mathcal{F}}=\wedge^{\mathrm{top}}N(\mathcal{F})$ is a trivial vector bundle.
  %Let us show that the modular class of $\mathcal{F}$ is zero, meaning that there exists a morphism of $\mathcal{F}$-modules between t.

Following Examples 3.35 and 3.100 in \cite{laurent-gengouxUniversalLieInfinityAlgebroid2020}, there exists a universal Lie $\infty$-algebroid of $\mathcal{F}$ that has length~2:
\begin{center}
\begin{tikzcd}[column sep=0.7cm,row sep=0.4cm]
0\ar[r]&\Gamma(E_{-1})\ar[r,"{d}"]&\Gamma(E_{0})\ar[r,"\rho"]&\mathcal{F}
\end{tikzcd}
\end{center}
where the fiber of $E_{0}$ is $S^2(\mathbb{R}^2)\otimes \mathbb{R}^2$ and the fiber of $E_{-1}$ is $\mathbb{R}^2\otimes \mathbb{R}^2$. At the point $(x,y)\in \mathbb{R}^2$, the anchor map is defined as:
\begin{equation}
\begin{pmatrix}
a & b \\
b & c
\end{pmatrix}\otimes (u,v)\xmapsto{\hspace{1cm}} \big(ax^2+2bxy+cy^2\big)\,\big(u\,\partial_x+v\,\partial_y\big)
\end{equation}
whereas the map $d$ is defined as:
\begin{equation}
\begin{pmatrix}
 \alpha \\
\beta
\end{pmatrix}\otimes (u,v)\xmapsto{\hspace{1cm}}\begin{pmatrix}
2y\alpha & -x\alpha+y\beta \\
-x\alpha+y\beta & -2x\beta
\end{pmatrix}\otimes (u,v)
\end{equation}

For clarity, in the following, we will denote a generic constant section of $E_0$ as $A\otimes X$ where one has to understand that $A=\begin{pmatrix}
a & b \\
b & c
\end{pmatrix}$ is a $2\times2$ symmetric matrix, and where $X=(u,v)$ is a vector of $\mathbb{R}^2$. The quadratic function $ax^2+2bxy+cy^2$ will be denoted  $A_{(x,y)}$, and the vector field $u\,\partial_x+v\,\partial_y$ which corresponds to $X$ is denoted $\overline{X}$. Then, at any point $(x,y)\in\mathbb{R}^2$, the 2-brackets and 3-bracket defining the $L_\infty$-algebra are given (on constant sections) by:
\begin{align}
l_2\big(A\otimes X,B\otimes Y\big)&=\overline{X}\big(B_{(x,y)}\big)\,A\otimes Y-\overline{Y}\big(A_{(x,y)}\big)\,B\otimes X\label{bracket1}\\
l_2\left(A\otimes X,\begin{pmatrix}
 \alpha \\
\beta
\end{pmatrix}\otimes Y\right)&=\begin{pmatrix}
\alpha(axu+2bxv+cyv)+ \beta(ayu-axv)\\
\beta(axu+2byu+cyv)+ \alpha(cxv-cyu)
\end{pmatrix}\otimes Y\label{bracket2}\\
&\hspace{4cm}-\overline{Y}\big(A_{(x,y)}\big)\,\begin{pmatrix}
 \alpha  \\
\beta
\end{pmatrix}\otimes X\nonumber\\
l_3\Big(A_1\otimes X_1,A_2\otimes X_2,A_3\otimes X_3\Big)&=\overline{X_1}\circ\overline{X_2}\big(A_{3,(x,y)}\big)\begin{pmatrix}
 (b_1a_2-a_1b_2)x+\frac{1}{2}(c_1a_2-a_1c_2)y\\
(c_1b_2-b_1c_2)y+\frac{1}{2}(c_1a_2-a_1c_2)x
\end{pmatrix}\otimes X_3+\circlearrowleft\label{bracket3}
%\alpha\,\begin{pmatrix}
% axu+2bxv+cyv\\
% cxv-cyu
%\end{pmatrix}\otimes Y\\
%&\hspace{2cm}+\beta\,\begin{pmatrix}
% ayu-axv\\
% axu+2byu+cyv
%\end{pmatrix}\otimes Y
\end{align} 
where the circle arrow $\circlearrowleft$ means that we proceed to a circular permutation. Then the 2-brackets are extended to every section by applying the Leibniz rule \eqref{robinson}. The brackets have been precisely defined so that they satisfy the three equations:$\rho\circ l_2=l_2\circ \rho,\,
d\circ l_2=l_2\circ d$ and 
$l_2\circ l_2=-d\circ l_3$.

The Berezinian of $E$ is given by:
\begin{equation}
\mathrm{Ber}(E)=\wedge^2T^*\mathbb{R}^2\otimes\wedge^6 E_0\otimes \wedge^4 E^*_{-1}
\end{equation}
We will compute the modular class of $E$ by using Formula \eqref{gluttony4}.
The action of an element $A\otimes X$ on the first factor is given by the divergence of the vector field $\rho(A\otimes X)$. On the second factor, it is given by the adjoint action $l_2(A\otimes X,-)$ as given in Equation \eqref{bracket1}, and on the third factor by the dual of action \eqref{bracket2}. One can check that the action of every basis element of $E_0$:
\begin{align}
e_1=\begin{pmatrix}
1 & 0 \\
0 & 0
\end{pmatrix}\otimes (1,0)&,\qquad e_2=\begin{pmatrix}
0 & 1 \\
1 & 0
\end{pmatrix}\otimes (1,0),\qquad e_3=\begin{pmatrix}
0 & 0 \\
0 & 1
\end{pmatrix}\otimes (1,0)\\
e_4=\begin{pmatrix}
1 & 0 \\
0 & 0
\end{pmatrix}\otimes (0,1)&,\qquad e_5=\begin{pmatrix}
0 & 1 \\
1 & 0
\end{pmatrix}\otimes (0,1),\qquad e_6=\begin{pmatrix}
0 & 0 \\
0 & 1
\end{pmatrix}\otimes (0,1)
\end{align}
is zero on $\mathrm{Ber}(E)$. For example, by taking $e_1$, we have $\mathrm{div}\big(\rho(e_1)\big)=\mathrm{div}\big(x^2\,\partial_x\big)=2x$, $\mathcal{L}_{e_1}\big|_{\wedge^6E_0}=-2x\,\mathrm{id}_{\wedge^6E_0} $ and $\mathcal{L}_{e_1}\big|_{\wedge^4(E_{-1})^*}=0$. Let us show the two last results in more detail since this may not be obvious to the reader.

Let $\lambda=e_1\wedge\ldots\wedge e_6$ be the canonical section of $\wedge^6E_0$, then the action of $e_1$ on $\lambda$ is performed through the adjoint action \eqref{bracket1}, and it defines a factor $a_1$: $\mathcal{L}_{e_1}(\lambda)= a_1\, \lambda$. Thus we have to compute the bracket of $e_1$ with every basis elements of $E_0$, by using Equation \eqref{bracket1}, and sum the contributions to obtain $a_1$:
\begin{align}
[e_1,e_1]&=0,\quad
&[e_1,e_2]&=2y\,e_1-2x\,e_2,\quad
&[e_1,e_3]&=-2x\,e_3,\\
[e_1,e_4]&=2x\,e_4,\quad
&[e_1,e_5]&=2y\,e_4,\quad
&[e_1,e_6]&=0.
\end{align}
Not every terms on the right hand side will contribute to $a_1$ since for example the $e_1$ contribution in $[e_1,e_2]$ will be cancelled out because $e_1\wedge e_1=0$ so that $e_1\wedge[e_1,e_2]\wedge e_3\wedge \ldots\wedge e_6=-2x\,e_1\wedge\ldots\wedge e_6$. Hence we have: $a_1=-2x-2x+2x=-2x$, so that $a_1$ and $\mathrm{div}(e_1)$ cancel one another.

To conclude the calculation, we have to show that the action of $e_1$ on $\wedge^4 E^*_{-1}$ is trivial. To do this, we will instead compute the action of $e_1$ on the dual $\wedge^4E_{-1}$ and show that it is zero. Let $\mu=f_1\wedge\ldots\wedge f_4$ be the canonical section of $\wedge^4 E_{-1}$, where $f_1,\ldots,f_4$ are the basis elements of $E_{-1}$:
\begin{align}
f_1&=\begin{pmatrix}
1\\0
\end{pmatrix}\otimes (1,0)
,\qquad f_2=\begin{pmatrix}
0\\1
\end{pmatrix}\otimes (1,0)\\
f_3&=\begin{pmatrix}
1\\0
\end{pmatrix}\otimes (0,1)
,\qquad f_4=\begin{pmatrix}
0\\1
\end{pmatrix}\otimes (0,1)
\end{align}
Then the action of $e_1$ on $\mu$ defines a function $a_\mu$ by: $\mathcal{L}_{e_1}(\mu)=a_\mu\, \mu$. We compute $a_\mu$ by summing all contributions of the $\mathcal{L}_{e_1}$ on the $f_i$. We have, by using Equation \eqref{bracket2}:
\begin{align}
[e_1,f_1]&=x\, f_1-2x\,f_1=-x\,f_1,&[e_1,f_2]&=x\,f_2-2x\,f_2=-x\,f_2,\\
[e_1,f_3]&=x\, f_3,&[e_1,f_4]&=x\,f_4.
\end{align}
Then the action of $e_1$ on $\mu$ consists of the sum of the contributions on each basis vector, which sums up to zero. Hence the result: the modular one form $\theta_{\Omega}$, when applied on $e_1$, is zero. One can check that this result is true for the other basis vectors $e_i$, proving the nullity of the modular class of $E$, and thus that of the singular foliation $\mathcal{F}$.  Then, by Proposition \ref{prop:ok}, the Berezinian $\mathrm{Ber}(E)$ is isomorphic -- as a $\mathcal{F}$-module -- to the trivial line bundle $\wedge^{0}\mathcal{F}^\circ=\mathbb{R}^n\times\mathbb{R}$. In particular, the standard constant section $\Omega$ of the former can be sent to the constant section $\textbf{1}$ of the latter.

As a side remark, notice that although $\wedge^4 E_{-1}$ does not contribute to the modular class, the vector bundle $E_{-1}$ is nonetheless necessary to make the Jacobi identity on $E_0$ close up to homotopy. As a matter of fact, it is not known yet if $\mathcal{F}$ descends from a Lie algebroid. This illustrates the importance of relying on universal Lie $\infty$-algebroids to define modular classes of singular foliations. 
%As a side remark, notice that, due to the absence of contribution coming from $\wedge^4(E_{-1})^*$, the modular class of the Lie algebroid $E_0$ is zero as well. This is not always the case, as is shown in the next example.

%we reach the following result that explains the role of the modular class of a singular foliation:
%\begin{proposition}\label{prop:ok}
%Let $\mathcal{F}$ be a singular foliation and let $E$ be a universal Lie $\infty$-algebroid of $\mathcal{F}$. Then the action of $\mathcal{F}$ on $\mathrm{Ber}(E)$ canonically extends at singularities, and the modular class $\theta^\mathcal{F}$  measures this action:% satisfies the following identity:
%\begin{equation}
%\nabla^\mathcal{F}_X(\Omega)=\theta^\mathcal{F}(X)\,\Omega
%\end{equation}
%for any $X\in\mathcal{F}$ and volume form $\Omega\in\Gamma\big(\mathrm{Ber}(E)\big)$.
%\end{proposition}

% Let $\phi:E\to E'$ be any chain map that intertwines the anchor maps: $\rho=\rho'\circ\phi$. Any other choice of chain map having this property induces the same diffeomorphism $\varphi$. Let $\Omega$ be any section of $\mathrm{Ber}(E)$, then it uniquely defines a section $\Omega'=\varphi(\Omega)$ of $\mathrm{Ber}(E')$. Then we have \textbf{POURQUOI ?}:
%\begin{equation}
%\varphi\big(\nabla_a(\Omega)\big)=\nabla'_{\phi(a)}(\Omega')
%\end{equation}
%

\subsection{The action of $\mathfrak{gl}_n(\mathbb{R})$ on $\mathbb{R}^n$}\label{exempleglau}
We study the singular foliation $\mathcal{F}_{\mathfrak{gl}_n}$ induced by the action of $\mathfrak{gl}_n(\mathbb{R})$ on $\mathbb{R}^n$, for $n\geq2$. 
As for Example \ref{quadra},  the leaves of $\mathcal{F}_{\mathfrak{gl}_n}$ consist of the punctured plane (the regular leaf) and the origin (singular leaf). Since on the punctured plane the conormal bundle is the zero vector bundle, it is extendable at the origin and is denoted $\mathcal{F}_{\mathfrak{gl}_n}^\circ$. We deduce that the determinant line bundle $\wedge^{\mathrm{top}}\mathcal{F}_{\mathfrak{gl}_n}^\circ$ actually involves the zeroth power of $\mathcal{F}_{\mathfrak{gl}_n}^\circ$ and is the  trivial vector bundle: $\wedge^{0}\mathcal{F}_{\mathfrak{gl}_n}^\circ=\mathbb{R}^n\times \mathbb{R}$. The sections of this line bundle consists of the algebra of smooth functions on $\mathbb{R}^n$ and any constant function is a global nowhere-vanishing section on which any vector field acts trivially. It means that $\wedge^{0}\mathcal{F}_{\mathfrak{gl}_n}^\circ$ is a trivial $\mathcal{F}_{\mathfrak{gl}_n}$-module. Computing the modular class of $\mathcal{F}_{\mathfrak{gl}_n}$ will show us if the trivial line bundle $\wedge^{0}\mathcal{F}_{\mathfrak{gl}_n}^\circ$ is isomorphic to the Berezinian of \emph{any} universal Lie $\infty$-algebroid of $\mathcal{F}_{\mathfrak{gl}_n}$ (see Proposition \ref{prop:ok}). % We will show that the modular class of the singular foliation  $\mathcal{F}_{\mathfrak{gl}_n}$ is zero. 
%%%However this should not be taken as a general rule:
%This example is important in the sense that the foliation  $\mathcal{F}_{\mathfrak{gl}_n}$   can be obtained as descending from the action Lie algebroid $\mathbb{R}^n\rtimes \mathfrak{gl}_n(\mathbb{R})$, whose modular class is not vanishing. However, we will show that the modular class of a particular universal Lie $\infty$-algebroid of  $\mathcal{F}_{\mathfrak{gl}_n}$ is zero. 
%%%at the end of  Example \ref{exampleson} we will have the same determinant line bundle but the modular class of the singular foliation there would not vanish.  This proves that  the generators of the foliation carry more informations than the mere leaves. %and that this latter result (rather than the former) is consistent with the notion of invariant transverse measure to the foliation.

Let $x_1, \ldots, x_n$ be the canonical coordinates on $\mathbb{R}^n$ and let $\big(e^{(0)}_{i,j}\big)_{1\leq i,j\leq n}$ be the canonical basis of $\mathfrak{gl}_n(\mathbb{R})$, i.e. $e^{(0)}_{i,j}$ is the $n\times n$-matrix with zero everywhere and a 1 at line $i$ and column $j$. Let $E_{0}=\mathbb{R}^n\rtimes \mathfrak{gl}_n$ be the action Lie algebroid associated to the action of $\mathfrak{gl}_n(\mathbb{R})$ on $\mathbb{R}^n$. The anchor and brackets are defined on constant sections by:
\begin{align}
\rho\big(e^{(0)}_{i,j}\big)&=x_i\partial_{x_j}\\
\big[e^{(0)}_{i,j},e^{(0)}_{k,l}\big]&=\delta_{jk}\,e^{(0)}_{i,l}-\delta_{li}\, e^{(0)}_{k,j}\label{crochetgln}
\end{align}
for every $1\leq i,j,k,l\leq n$. The bracket is the usual bracket on $\mathfrak{gl}_n(\mathbb{R})$, that we extend to every sections of $E_0$ by the Leibniz rule \eqref{robinson}.
%\begin{equation}\label{eq:leibnizrr}
%\big[e^{(0)}_{i,j},fe^{(0)}_{k,l}\big]=f\big[e^{(0)}_{i,j},e^{(0)}_{k,l}\big]+x_i\partial_{x_j}(f)e^{(0)}_{k,l}
%\end{equation}
%for any smooth function $f\in\mathcal{C}^\infty(\mathbb{R}^n)$.
 The anchor map has a non trivial kernel, spanned by elements of the following kind:
\begin{equation}\label{relation1}
x_ie^{(0)}_{j,k}-x_je^{(0)}_{i,k}
\end{equation}
Notice that the $i,j$-indices on the left hand side have a skew-symmetric property, hence showing that the minimal number of generators of $\mathrm{Ker}(\rho)$ is $\frac{n(n-1)}{2}$.
Let us set $E_{-1}=\mathbb{R}^n\times \wedge^2\mathbb{R}^n\oplus\ldots\oplus \wedge^2\mathbb{R}^n$, where the tensor product is made over $n$ copies of $\wedge^2\mathbb{R}^n$ and where, for clarity, we omitted the suspension operator that shifts the degree of the fiber by $-1$. Let $\big(e^{(1)}_{ij,k}\big)_{1\leq i<j\leq n}$ be the canonical basis of the $k$-th factor in the product $\wedge^2\mathbb{R}^n\oplus\ldots\oplus \wedge^2\mathbb{R}^n$. For $j<i$ we set $e^{(1)}_{ij,k}=-e^{(1)}_{ji,k}$ so that we have a set of vectors $\big(e^{(1)}_{ij,k}\big)_{1\leq i,j\leq n}$ that have this anti-symmetry property on the $i,j$-indices. % (in particular $e^{(1)}_{ij,k}=-e^{(1)}_{ij,k}$).
 Then, define a vector bundle map $d^{(0)}\colon E_{-1}\to E_0$ by the following relations, for every $1\leq i,j,k\leq n$:
\begin{equation}
d^{(0)}\big(e^{(1)}_{ij,k}\big)=x_ie^{(0)}_{j,k}-x_je^{(0)}_{i,k}
\end{equation}
%In that case, we obtain the following resolution at the level of sections:
%\begin{center}
%\begin{tikzcd}[column sep=0.7cm,row sep=0.4cm]
%0\ar[r]&\Gamma(E_{-1})\ar[r,"{d}"]&\Gamma(E_{0})\ar[r,"\rho"]&\mathcal{F}_{\mathfrak{gl}_n}
%\end{tikzcd}
%\end{center}
The bracket between a section of $E_0$ and a section of $E_{-1}$ can be first defined on constant sections as:
\begin{equation}\label{crochets}
\big[e^{(0)}_{i,j},e^{(1)}_{kl,m}\big]=\delta_{jk}\,e^{(1)}_{il,m}+\delta_{jl}\, e^{(1)}_{ki,m}-\delta_{im}\,e^{(1)}_{kl,j}
\end{equation}
Then it is extended to every sections of $E_{-1}$ by the Leibniz rule \eqref{robinson}.  %This bracket closes the Lie 2-algebroid structure on $E_0\oplus E_{-1}$, since there is no 3-bracket on $E_0$.

At the level of sections, the kernel of the map $d^{(0)}\colon E_{-1}\to E_0$ is not zero, since elements of the form:
\begin{equation}\label{relation4}
x_ie^{(1)}_{jk,l}+x_{j}e^{(1)}_{ki,l}+x_{k}e^{(1)}_{ij,l}
\end{equation}
span this kernel. Let us set $E_{-2}=\mathbb{R}^n\times \wedge^3\mathbb{R}^n\oplus\ldots\oplus \wedge^3\mathbb{R}^n$, where the tensor product is made over $n$ copies of $\wedge^3\mathbb{R}^n$ and where, for clarity, we omitted the suspension operator that shifts the degree of the fiber by $-1$. Let $\big(e^{(2)}_{ijk,l}\big)_{1\leq i<j<k\leq n}$ be the canonical basis of the $l$-th factor in the tensor product $\wedge^3\mathbb{R}^n\oplus\ldots\oplus \wedge^3\mathbb{R}^n$. Then for any $1\leq i,j,k,l\leq n$ we set $e^{(2)}_{ijk,l}=(-1)^{\sigma}e^{(2)}_{\sigma(i)\sigma(j)\sigma(k),l}$ where $\sigma$ is the unique permutation such that $\sigma(i)<\sigma(j)<\sigma(k)$. We define a vector bundle map $d^{(1)}\colon E_{-2}\to E_{-1}$ by:
\begin{equation}
d^{(1)}\big(e^{(2)}_{ijk,l}\big)=x_ie^{(1)}_{jk,l}+x_{j}e^{(1)}_{ki,l}+x_{k}e^{(1)}_{ij,l}
\end{equation}
The bracket between a section of $E_0$ and a section of $E_{-2}$ can be first defined on constant sections as:
\begin{equation}\label{crochets11}
\big[e^{(0)}_{i,j},e^{(2)}_{klm,r}\big]=\delta_{jk}\,e^{(2)}_{ilm,r}+\delta_{jl}\, e^{(2)}_{kim,r}+\delta_{jm}\, e^{(2)}_{kli,r}-\delta_{ir}\,e^{(2)}_{klm,j}
\end{equation}
Then it is extended to every sections of $E_{-2}$ by the Leibniz rule \eqref{robinson}. We do not need to define the bracket between two elements of $E_{-1}$ because is is not used to compute the modular class.

The same process repeats itself up to order $n$, since the dimension of $\wedge^n\mathbb{R}^n$ is one. In the end, we will have a geometric resolution:
\begin{center}
\begin{tikzcd}[column sep=0.7cm,row sep=0.4cm]
0\ar[r]&\Gamma(E_{-n+1})\ar[r,"{d}^{(n-2)}"]&\ldots\ar[r, "d^{(2)}"]&\Gamma(E_{-2})\ar[r,"{d}^{(1)}"]&\Gamma(E_{-1})\ar[r,"{d}^{(0)}"]&\Gamma(E_{0})\ar[r,"\rho"]&\mathcal{F}_{\mathfrak{gl}_n}
\end{tikzcd}
\end{center}
where, for every $0\leq p\leq n-1$, $E_{-p}=\mathbb{R}^n\times \wedge^{p+1}\mathbb{R}^n\oplus\ldots\oplus \wedge^{p+1}\mathbb{R}^n$, such that there are $n$ copies of $\wedge^{p+1}\mathbb{R}^n$. The differential is defined as:
\begin{equation}
d^{(p-1)}\big(e^{(p)}_{i_1\ldots i_{p+1},k}\big)=\sum_{\sigma\in Un(1,p)}\ (-1)^{\sigma}\,x_{\sigma(i_1)}e^{(p-1)}_{\sigma(i_2)\ldots \sigma(i_{p+1}),k}
\end{equation}
where $Un(1,p)$ designates the $(1,p)$-unshuffles, i.e. the permutations of $p+1$ elements such that $\sigma(i_2)<\ldots<\sigma(i_{p+1})$. The sign $(-1)^\sigma$ is the signature of the permutation. The bracket between $E_0$ and $E_{-p}$ is given by the following formula:
\begin{equation}
\big[e^{(0)}_{i,j},e^{(p)}_{i_1\ldots i_{p+1},k}\big]=\sum_{l=1}^{p+1}\ \delta_{j i_{l}}\,e^{(p)}_{i_1\ldots i\ldots i_{p+1},k}-\delta_{ik}\,e^{(p)}_{i_1\ldots i_{p+1},j}
\end{equation}
where the index $i$, in the term under the sum, is replacing the $l$-th index $i_l$. There are other brackets, such as the two bracket between two elements of degree lower than 0, but these are not necessary to compute the modular class.

The  Berezinian associated to the geometric resolution is
\begin{align*}
&\text{for $n$ even}&\qquad&\mathrm{Ber}(E)=\wedge^nT^*\mathbb{R}^n\otimes \wedge ^{n^2}E_0\otimes \wedge^{n\binom{n}{2}} (E_{-1})^*\otimes \wedge^{n\binom{n}{3}} E_{-2}\otimes \ldots\otimes \wedge^n (E_{-n+1})^{*}&\\
&\text{for $n$ odd}&\qquad&\mathrm{Ber}(E)=\wedge^nT^*\mathbb{R}^n\otimes \wedge ^{n^2}E_0\otimes \wedge^{n\binom{n}{2}} (E_{-1})^*\otimes \wedge^{n\binom{n}{3}} E_{-2}\otimes \ldots\otimes \wedge^n E_{-n+1}&
\end{align*}
For $n\geq1$ even (resp. odd),  let $\Omega=\omega\otimes \lambda\otimes \mu^{(1)*}\otimes \mu^{(2)}\otimes\ldots\otimes \mu^{(n-1)*}$ (resp. $\omega\otimes \lambda\otimes \mu^{(1)*}\otimes \mu^{(2)}\otimes\ldots\otimes \mu^{(n-1)}$)  be a section of the Berezinian, where $\mu^{(p)}$ is a nowhere vanishing section of $\wedge^{n\binom{n}{p}} E_{-p}$, and where $\mu^{(p)*}$ is a  nowhere vanishing section of  $\wedge^{n\binom{n}{p}} (E_{-p})^*$. Since the modular class does not depend on the chosen volume form, we could pick up the one defined by the following constant sections:
\begin{align}
\omega&=dx_1\wedge\ldots\wedge dx_n\\
\lambda&= \underset{1\leq k\leq n}{\bigwedge} e^{(0)}_{1,k}\wedge \underset{1\leq k\leq n}{\bigwedge} e^{(0)}_{2,k}\wedge\ldots \wedge  \underset{1\leq k\leq n}{\bigwedge}e^{(0)}_{n,k}\\
%e^{(0)}_{1,1}\wedge\ldots\wedge e^{(0)}_{1,n}\wedge e^{(0)}_{2,1}\wedge\ldots\wedge e^{(0)}_{(n-1),n}\wedge e^{(0)}_{n,1}\wedge \ldots\wedge e^{(0)}_{n,n}\\
\mu^{(p)}&=\mu^{(p)}_1\wedge\ldots\wedge\mu^{(p)}_n
\end{align}
where, for any $1\leq k,p\leq n$, we set:
\begin{equation}
\mu^{(p)}_k=\underset{2\leq i_2<\ldots <i_{p+1}\leq n}{\bigwedge} e^{(p+1)}_{1i_{2}\ldots i_{p+1},k}\wedge\underset{3\leq i_2<\ldots <i_{p+1}\leq n}{\bigwedge} e^{(p+1)}_{2i_{2}\ldots i_{p+1},k}\wedge \ldots \wedge e^{(p+1)}_{n-p\ldots n,k}%e^{(p+1)}_{1\ldots p+1,k}\wedge \ldots \wedge u_{1n,k}^*\wedge u_{23,k}^*\wedge\ldots \wedge u_{(n-2)(n-1),k}^*\wedge u_{(n-2)n,k}^*\wedge u_{(n-1)n,k}^*
\end{equation}
% is a constant section of $\wedge^{p+1}\mathbb{R}^n$. 
 We will then compute the modular class of $E$ by using Formula \eqref{gluttony4}. %, and where $e^{(1)*}_{ij,k}$ is the dual element to $e^{(1)}_{ij,k}$. 
The computation of the value of $\theta_\Omega$ is performed in several steps: we pick up a constant section $e^{(0)}_{i,j}$ of $E_0$, and we compute independently the value of the action of $e^{(0)}_{i,j}$ on each term $\omega, \lambda$ and $\mu^{(p)}$, for every $1\leq p\leq n-1$. Then we add up all the contributions to obtain the value of $\theta_\Omega\big(e^{(0)}_{i,j}\big)$.

First, the action of $e^{(0)}_{i,j}$ on $\wedge^nT^*\mathbb{R}^n$ gives the divergence of $\rho(e^{(0)}_{i,j})$:
\begin{equation}
\mathcal{L}_{e^{(0)}_{i,j}}(\omega)=\mathrm{div}\big(x_i\partial_{x_j}\big)\,\omega=\delta_{ij}\,\omega
\end{equation}
Second, the action of $e^{(0)}_{i,j}$ on $\wedge^{n^2}E_0$ is obtained from the adjoint action $a\to [e^{(0)}_{i,j},a]$ defined in Equation \eqref{crochetgln}. Since $\lambda$ is a totally alternating tensor, the only part of $[e^{(0)}_{i,j},e^{(0)}_{k,l}]$ that will contribute to the value of $\theta_\Omega\big(e^{(0)}_{i,j}\big)$ is the part that is again proportional to $e^{(0)}_{k,l}$. For this reason, if $i\neq j$, one can check that the action of $e^{(0)}_{i,j}$ on $\lambda$ is null. Now if $i=j$, it is still trivial, but for another reason: the only elements that will contribute will be $(e^{(0)}_{i,l})_{1\leq l\leq n}$ and $(e^{(0)}_{l,i})_{1\leq l\leq n}$, but for every $l$, $[e^{(0)}_{i,i}, e^{(0)}_{i,l}]=e^{(0)}_{i,l}$
and $[e^{(0)}_{i,i}, e^{(0)}_{l,i}]=-e^{(0)}_{l,i}$, so that the sum of the contributions of $e^{(0)}_{i,l}$ and of $e^{(0)}_{l,i}$ will always cancel. Hence:
\begin{equation}\label{contribution1}
\mathcal{L}_{e^{(0)}_{i,j}}(\lambda)=0
\end{equation}

Now, let us compute the contribution of the action of $e^{(0)}_{i,j}$ on $\mu^{(1)*}\in \Gamma\big(\wedge^{n\binom{n}{2}} E_{-1}^*\big)$ to the value of the modular class. We will proceed as follows: we will compute the action of $e^{(0)}_{i,j}$ on $\mu^{(1)}$, and then, by duality,  we will revert the sign of the result to obtain the action of $e^{(0)}_{i,j}$ on $\mu^{(1)*}$. Applying the same argument as above to the bracket defined in Equation \eqref{crochets}, we notice that if $i\neq j$ then the action of $e^{(0)}_{i,j}$ on $\mu^{(1)}$ is trivial. In the case that $i=j$, the only basis elements that will contribute are $e^{(1)}_{il,m}, e^{(1)}_{ki,m}$ and $e^{(1)}_{kl,i}$, for some values of $k,l\neq i$ and $m\neq i$, for if we have simultaneously e.g. $k=i$ and $m=i$, from Equation \eqref{crochets}, we obtain:
\begin{equation}
[e^{(0)}_{i,i},e^{(1)}_{il,i}]=e^{(1)}_{il,i}-e^{(1)}_{il,i}=0
\end{equation}
Then one can suppose that, when they appear in the subsequent formulas, both $k,l$ and $m$ are different than $i$.
Then we have:
\begin{align}
[e^{(0)}_{i,i},e^{(1)}_{il,m}]&=e^{(1)}_{il,m}\\
[e^{(0)}_{i,i},e^{(1)}_{ki,m}]&=e^{(1)}_{ki,m}\\
[e^{(0)}_{i,i},e^{(1)}_{kl,i}]&=-e^{(1)}_{kl,i}
\end{align}
The first equation is valid for $i<l\leq n$ and $m\neq i$, and hence contribute with a factor of $(n-i)(n-1)$ to the modular class. The second equation is valid for $1\leq k< i$ and $m\neq i$, thus contributing to a factor $(i-1)(n-1)$ to the modular class.
Finally the last equation is valid for those $k,l$ that are not equal to $i$.
Since we have the strict inequality $k<l$, there are $\frac{n(n-1)}{2}-(i-1)-(n-i)=\frac{n(n-1)}{2}-(n-1)$ such terms. Due to the minus sign, these terms would contribute to a factor $(n-1)-\frac{n(n-1)}{2}$ to the modular class. Adding all these contributions  together we deduce how $e^{(0)}_{i,i}$ acts on $\mu^{(1)}$:
\begin{equation}
\mathcal{L}_{e^{(0)}_{i,i}}(\mu^{(1)})=(n-i)(n-1)+(i-1)(n-1)+(n-1)-\frac{n(n-1)}{2}=\frac{n(n-1)}{2}
\end{equation}
Reverting the sign, we deduce that $\mathcal{L}_{e^{(0)}_{i,i}}\big(\mu^{(1)*}\big)=-\binom{n}{2}\,\mu^{(1)*}$. 

Now turning to the action of $e^{(0)}_{i,j}$ on $\mu^{(2)}$. The above discussion implies that $\mathcal{L}_{e^{(0)}_{i,j}}\big(\mu^{(2)}\big)=0$ if $i\neq j$. Then, when $i=j$, several terms contribute to the value of the modular class, for Equation \eqref{crochets11} gives:
\begin{align}
[e^{(0)}_{i,i},e^{(2)}_{ilm,r}]&=e^{(2)}_{ilm,r}\\
[e^{(0)}_{i,i},e^{(2)}_{kim,r}]&=e^{(2)}_{kim,r}\\
[e^{(0)}_{i,i},e^{(2)}_{kli,r}]&=e^{(2)}_{kli,r}\\
[e^{(0)}_{i,i},e^{(2)}_{klm,i}]&=-e^{(2)}_{klm,i}
\end{align}
when $k,l,m,r\neq i$. More precisely, terms of the kind $e^{(2)}_{ilm,r}$ contribute for a value of $(n-1)\frac{(n-i)(n-i-1)}{2}$, terms of the kind $e^{(2)}_{kli,r}$ contribute for a value of $(n-1)\frac{(i-1)(i-2)}{2}$, terms of the kind $e^{(2)}_{kim,r}$ contribute for a value of $(n-1)(i-1)(n-i)$ times, and those of the kind $e^{(2)}_{klm,i}$ contribute for a value of $-\Big(\binom{n}{3}-\frac{(n-i)(n-i-1)}{2}-\frac{(i-1)(i-2)}{2}-(i-1)(n-i)\Big)$. A short calculation then shows that adding them up gives $\mathcal{L}_{e^{(0)}_{i,i}}\big(\mu^{(2)}\big)=2\binom{n}{3}\mu^{(2)}$. 
More generally, one can check that for $1\leq p\leq n-1$, we have:
\begin{equation}\label{contribution2}
\mathcal{L}_{e^{(0)}_{i,i}}\big(\mu^{(p)}\big)=p\,\binom{n}{p+1}\,\mu^{(p)}\end{equation}
and that $\mathcal{L}_{e^{(0)}_{i,j}}\big(\mu^{(p)}\big)=0$ if $i\neq j$.

From this, we first deduce that if $i\neq j$ then $\theta_\Omega\big(e^{(0)}_{i,j}\big)=0$ 
since the action of  $e^{(0)}_{i,j}$ on $\Omega$ is trivial. When $i=j$ however, we can compute the value of $\theta_\Omega\big(e^{(0)}_{i,i}\big)$ by adding up all previous contributions. By keeping in mind that there are minus sign coming up when acting on $\mu^{(p)*}$, we obtain: %(recall that $\Omega$ symbolizes the standard section of the Berezinian):
\begin{align}
%\theta^{\mathcal{F}_{\mathfrak{gl}_n}}
\theta_\Omega\big(e^{(0)}_{i,i}\big)&=1+0+\sum_{p=1}^{n-1}(-1)^{p}p\binom{n}{p+1}\label{eqform}\\
&=1+\sum_{p=1}^{n-1}(-1)^{p}(p+1)\binom{n}{p+1}-\sum_{p=1}^{n-1}(-1)^{p}\binom{n}{p+1}\\
&=1+\sum_{p=1}^{n-1}(-1)^{p}n\binom{n-1}{p}+\sum_{p=2}^{n}(-1)^{p}\binom{n}{p}\\
&=1+n\sum_{p=1}^{n-1}(-1)^{p}\binom{n-1}{p}+(1-1)^n+n-1\\
&=n\sum_{p=0}^{n-1}(-1)^{p}\binom{n-1}{p}\\
&=n(1-1)^{n-1} \label{eqform2}
%&=1-\sum_{p=2}^{n}(-1)^{p}\,p\binom{n}{p}+\sum_{p=2}^{n}(-1)^{p}\binom{n}{p}\\
%&=1-\sum_{p=2}^{n}(-1)^{p}\,p\binom{n}{p}+(1-1)^n+n-1\\
%&=-\sum_{p=1}^{n}(-1)^{p}\,p\binom{n}{p}=-\frac{d}{dx}(1-x)^n\Big|_{x=1}\\
%&=n(1-x)^{n-1}\Big|_{x=1}=0
\end{align}
%&=1+\sum_{p=1}^{n-1}(-1)^{p}(p+1)\binom{n}{p+1}-\sum_{p=1}^{n-1}(-1)^{p}\binom{n}{p+1}\\
%&=1+\sum_{p=1}^{n-1}(-1)^{p}n\binom{n-1}{p}-\sum_{p=1}^{n-1}(-1)^{p}\binom{n}{p+1}\\
%&=1+n\sum_{p=1}^{n-1}(-1)^{p}\binom{n-1}{p}+\sum_{p=2}^{n}(-1)^{p}\binom{n}{p}\\
%&=1+n\sum_{p=1}^{n-1}(-1)^{p}\binom{n-1}{p}+(1-1)^n+n-1\\
%&=n\sum_{p=0}^{n-1}(-1)^{p}\binom{n-1}{p}\\
%&=n(1-1)^{n-1}=0
Since $n\geq2$, we obtain that
$\theta_\Omega\big(e^{(0)}_{i,j}\big)=0$ for every $1\leq i,j\leq n$, so the modular class $\theta^{\mathcal{F}_{\mathfrak{gl}_n}}$ of the singular foliation $\mathcal{F}_{\mathfrak{gl}_n}$ is zero. %It means that there exists a $\mathcal{F}$-invariant nowhere vanishing section of the trivial line bundle  $L_{\mathcal{F}_{\mathfrak{gl}_n}}=\wedge^0N(\mathcal{F}_{\mathfrak{gl}_n})$.

As in Example \ref{quadra} then, the Berezinian is  isomorphic -- as a $\mathcal{F}_{\mathfrak{gl}_n}$-module -- to the trivial line bundle $\wedge^{0}\mathcal{F}_{\mathfrak{gl}_n}^\circ=\mathbb{R}^n\times\mathbb{R}$. In particular, the standard constant section $\Omega$ of the former can be sent to the constant section $\textbf{1}$ of the latter. %Since we have computed the modular class by evaluating the action of $E_0$ on the standard constant section $\Omega$, then we deduce that any constant nowhere vanishing section of $L_{\mathcal{F}_{\mathfrak{gl}_n}}$ is invariant under the action of $\mathcal{F}$. 
%Notice that this is not related to the fact that the transversal to the regular leaf is the zero vector bundle. 
However this should not be taken as a general rule:
%This example is important in the sense that the foliation  $\mathcal{F}_{\mathfrak{gl}_n}$   can be obtained as descending from the action Lie algebroid $\mathbb{R}^n\rtimes \mathfrak{gl}_n(\mathbb{R})$, whose modular class is not vanishing. However, we will show that the modular class of a particular universal Lie $\infty$-algebroid of  $\mathcal{F}_{\mathfrak{gl}_n}$ is zero. 
at the end of  Example \ref{exampleson} we will have the same determinant line bundle but the modular class of the singular foliation there would not vanish.  This proves that  the generators of the foliation carry more informations than the mere leaves.  %Indeed, there exist singular foliations $\mathcal{F}$ for which the regular leaf is the punctured space and thus, whose transversal is the zero vector bundle, but for which $L_{\mathcal{F}}$ does not admit an invariant nowhere vanishing section. 
%see the end of Example \ref{exampleson} for such an explanation. 
  %\textbf{tirer au clair quelle section invariante du line bundle existe et pas dans le cas son + euler}  %This was expected \textbf{since when??}. %the vector bundle $L_{\mathcal{F}_{\mathfrak{gl}_n}}$ is trivial, because the conormal bundle on the regular leaf is the zero vector bundle. 

 In the case where $n=1$, the singular foliation $\mathcal{F}_{\mathfrak{gl}_1}$ on $\mathbb{R}$ is generated by the vector field $x\partial_x$, which is the Euler vector field on $\mathbb{R}$. %The associated singular foliation has been already treated in Example \ref{exaeuler}, so we know that the modular class does not vanish.
The following argument shows that its modular class does not vanish:
 the singular foliation $\mathcal{F}_{\mathfrak{gl}_1}$ admits a universal Lie $E_0=\infty$-algebroid $\mathbb{R}\rtimes \mathbb{R}$ with anchor $\rho\big(e_{1,1}^{(0)}\big)=x\partial_x$, with no kernel. The divergence of this vector field is 1, and the action of $e_{1,1}^{(0)}$ on $\wedge^1E_0$ is trivial, so the modular differential 1-form of this foliation is not zero, and rather satisfies $ \theta_\Omega\big(e^{(0)}_{1,1}\big)=1$. This proves nonetheless that Formula \eqref{eqform}-\eqref{eqform2}  is still valid for $n=1$, and can be summarized as follows:
 \begin{equation}
 \theta_\Omega\big(e^{(0)}_{i,j}\big)=n\,0^{n-1}\hspace{1cm}
 \text{for every $1\leq i,j\leq n$}
 \end{equation}
 However, based on the discussion following Equation \eqref{espresso}, the modular class of $\mathcal{F}_{\mathfrak{gl}_1}$ is not zero.

\subsection{The action of $\mathfrak{so}_n(\mathbb{R})$ on $\mathbb{R}^n$}
\label{exampleson}
Let us now turn to the action of $\mathfrak{so}_n(\mathbb{R})$ on $\mathbb{R}^n$, for $n\geq2$. It formalizes Example \ref{ex3} and will then meet the same obstructions found in Example \ref{ex4}.  The singular foliation $\mathcal{F}_{\mathfrak{so}_n}$ consists of all the vector fields that preserve the quadratic function:
\begin{equation}\label{equtile}
\varphi(x_1,\ldots,x_n)=\frac{1}{2}\sum_{i=1}^nx_i^2
\end{equation} 
The leaves consist of concentric $n-1$-dimensional spheres (regular leaves) and the origin (singular leaf).
The function defined in Equation \eqref{equtile} is homogeneous with isolated singularity (the origin). Then a geometric resolution of the induced singular foliation $\mathcal{F}_{\mathfrak{so}_n}$ is known and has been given by Koszul, see Example 3.36  in \cite{laurent-gengouxUniversalLieInfinityAlgebroid2020}:
\begin{center}
\begin{tikzcd}[column sep=0.7cm,row sep=0.4cm]
0\ar[r]&\Gamma(\wedge^nT\mathbb{R}^n)\ar[r,"\iota_{\mathrm{d}\varphi}"]&\ldots\ar[r,"\iota_{\mathrm{d}\varphi}"]&\Gamma(\wedge^3T\mathbb{R}^n)\ar[r,"\iota_{\mathrm{d}\varphi}"]&\Gamma(\wedge^2T\mathbb{R}^n)\ar[r,"\iota_{\mathrm{d}\varphi}"]&\mathcal{F}_{\mathfrak{so}_n}\ar[r,"\iota_{\mathrm{d}\varphi}"]&0
\end{tikzcd}
\end{center}
\noindent We then set $E_{-i}=\wedge^{i+2}T\mathbb{R}^n$ %, and $d=\iota_{\mathrm{d}\varphi}$, where the straight differential $\mathrm{d}$ is the de Rham differential on $\mathbb{R}^n$. 
The Lie $\infty$-algebroid structure is given in Example 3.101 of \cite{laurent-gengouxUniversalLieInfinityAlgebroid2020}. In particular, the 2-bracket between a constant section $\partial_{\{k,l\}}=\partial_{x_k}\wedge\partial_{x_l}$ of $E_0$ and a section $\partial_I=\partial_{i_1}\wedge\ldots\wedge \partial_{i_{|I|}}$ of any other $E_{-|I|+2}$ (where $|I|\geq2$ is the length of the multi-index $I=\{i_1,\ldots,i_{|I|}\}$) is given by:
\begin{equation}\label{bracketson}
l_2\big(\partial_{\{k,l\}},\partial_I\big)=\sum_{j=1}^{|I|}\delta_{ki_j}(-1)^{j}\partial_{x_l}\wedge\partial_{I/\{i_j\}}-\delta_{li_j}(-1)^{j}\partial_{x_k}\wedge\partial_{I/\{i_j\}}
\end{equation}
where the multi-vector field $\partial_{I/\{i_j\}}$ is the multi-vector field $\partial_I$ in which the $j$-th term $\partial_{i_j}$ has been removed. The Kronecker deltas come from the fact that, e.g. the term $\frac{\partial\varphi}{\partial_{x_k}\partial_{x_{i_j}}}$ is not vanishing only when $k=i_{j}$, due to the very specific form of the function $\varphi$.

From Equation \eqref{bracketson}, one can check that $l_2\big(\partial_{\{k,l\}},\partial_I\big)$ does not give any contribution proportional to $\partial_I$, since for it to happen one should have, e.g. some $i_j=k$, and at the same time, $l=i_j$ so that  $\delta_{ki_j}(-1)^{j-1}\partial_{x_l}\wedge\partial_{I/\{i_j\}}=\pm\partial_I$. But this would imply that $k=l$ and thus that $\partial_{\{k,l\}}=0$. Hence, the adjoint action of $E_0$ on $\wedge^{\mathrm{top}}E_{-|I|+2}$ is necessarily trivial. The anchor map acts on $\partial_{\{k,l\}}$ as follows:
\begin{equation}\label{eq:anchorso}
\rho\big(\partial_{\{k,l\}}\big)=\partial_{x_k}(\varphi)\partial_{x_l}-\partial_{x_l}(\varphi)\partial_{x_k}=x_k\partial_{x_l}-x_l\partial_{x_k}
\end{equation}
Since $k\neq l$, this vector field has a null divergence. As a conclusion, the action of $E_0$ on the Berezinian $\mathrm{Ber}(E)$ is trivial, and hence so is the action of $\mathcal{F}_{\mathfrak{so}_n}$ on $\mathrm{Ber}(E)$. Thus, the modular class of $\mathcal{F}_{\mathfrak{so}_n}$ is zero.

%According to Proposition \ref{prop:ok}, the vanishing of the modular class indicates that there exists a global invariant nowhere vanishing section of the trivial vector bundle $L_{\mathcal{F}_{\mathfrak{so}_n}}$.

 On the union of regular leaves (the punctured space),  the set of vector fields defined in Equation \eqref{eq:anchorso} induces a regular distribution $F$. The conormal bundle of this distribution is a rank 1 vector bundle $F^\circ\subset T^*(\mathbb{R}^n\backslash\{0\})$ generated by the radial one-form $dr^2=\sum_{i=1}^nx_idx_i$, where we set $r^2=x_1^2+\ldots+x_n^2$. This is a transverse volume form and it is easy to check that it is invariant under the action of $F$ because it is exact.  The conormal vector bundle does not extend at the origin so the vanishing of the modular class of $\mathcal{F}_{\mathfrak{so}_n}$ does not bring any new informations on $F$ that was not already known from the study of the regular leaves. 
 
However, adding the Euler vector field $\epsilon=\sum_{i=1}^nx_i\partial_{x_i}$ to the singular foliation $\mathcal{F}_{\mathfrak{so}_n}$ induces a new singular foliation $\mathcal{F}$ because the bracket  \eqref{bracketson} with the Euler vector field is zero. The singular leaf does not change, but there is now only one regular leaf:  the entire punctured space. Then, the conormal bundle of the regular leaf is a rank zero vector bundle as in Example \ref{exempleglau}, so it extends at the origin as a rank zero vector bundle $\mathcal{F}^\circ$. The determinant line bundle of the latter  is then a trivial line bundle $\wedge^0\mathcal{F}^\circ=\mathbb{R}^n\times\mathbb{R}$, and a trivial $\mathcal{F}$-module. 
Although the situation is similar as the one in  Example \ref{exempleglau}, so that one could naively expect that there is an isomorphism of $\mathcal{F}$-modules between this trivial line bundle and the Berezinian of any universal Lie $\infty$-algebroid of $\mathcal{F}$,  %the existence of an invariant transverse volume form would be always guaranteed in such a situation,
 the following argument proves that it is not the case.
  Indeed we can chose as a universal Lie $\infty$-algebroid of $\mathcal{F}$ the sum of that of $\mathcal{F}_{\mathfrak{so}_n}$ and that of Example \ref{exaeuler}. This latter example implies that the modular class of this new universal Lie $\infty$-algebroid is not zero on any neighborhood of the origin. Then this Berezinian, and any other Berezinian of other universal Lie $\infty$-algebroids of $\mathcal{F}$, are not trivial $\mathcal{F}$-modules. %This implies in turn that the trivial line bundle $L_\mathcal{F}$ does not admit any $\mathcal{F}$-invariant nowhere vanishing section.
Hence, although the foliation $\mathcal{F}$ possesses exactly the same leaves as that of Example \ref{exempleglau}, and induces the same trivial representation on the trivial line bundle $\wedge^0\mathcal{F}^\circ$, we see that the two foliations have quite a different behavior on their respective Berezinians. 
This proves again that the some information about a singular foliation is not contained in the leaves, but in the family of vector fields generating it.

\section*{Acknowledgments}

I would like to thank Camille Laurent-Gengoux for introducing me to the notion of modular class and for helping me on various occasions during the long process through which this paper emerged. I would like to thank Georges Skandalis and St\'ephane Vassout for hosting me at University Paris VII Diderot in 2019, during which the beginning of this work has been done. This work has been completed and finalized in 2021-2022 at Steklov Institute and I would like to thank Nikolai Mnev for its warm welcome. Thanks for Alfonso Garmendia, Raquel Caseiro, Joao Nuno Mestre, Lennart Obster and Zan Grad for very insightful comments about connections and representations of almost Lie algebroids. Funding: This work was supported by the french governement [ANR Singstar - ANR-14-CE25-0012] and the government of the Russian Federation [grant No. 075-15-2019-1620].

 \bibliography{Maths}

\begin{thebibliography}{10}

\bibitem{abadRepresentationsHomotopyLie2011}
C.~A. Abad and M.~Crainic.
\newblock Representations up to homotopy of {{Lie}} algebroids.
\newblock {\em J. Reine Angew. Math.}, 2012(663):91--126, 2011.

\bibitem{androulidakisHolonomyGroupoidSingular2009}
I.~Androulidakis and G.~Skandalis.
\newblock The holonomy groupoid of a singular foliation.
\newblock {\em J. Reine Angew. Math.}, 2009(626):1--37, 2009.

\bibitem{androulidakisSmoothnessHolonomyCovers2013}
I.~Androulidakis and M.~Zambon.
\newblock Smoothness of holonomy covers for singular foliations and essential
  isotropy.
\newblock {\em Math. Z.}, 275(3):921--951, 2013.

\bibitem{bonavolontaCategoryLieNalgebroids2013}
G.~Bonavolont{\`a} and N.~Poncin.
\newblock On the category of {{Lie}} n-algebroids.
\newblock {\em J. Geom. Phys.}, 73:70--90, 2013.

\bibitem{caseiroModularClassLie2022}
R.~Caseiro and C.~{Laurent-Gengoux}.
\newblock Modular class of {{Lie}} infinity-algebroids and adjoint
  representations.
\newblock {\em J. Geom. Mech.}, 14(2):273--305, 2022.

\bibitem{connesNeumannAlgebraFoliation1978}
A.~Connes.
\newblock The von {{Neumann}} algebra of a foliation.
\newblock In G.~Dell'Antonio, S.~Doplicher, and G.~{Jona-Lasinio}, editors,
  {\em Mathematical {{Problems}} in {{Theoretical Physics}}: {{International
  Conference Held}} in {{Rome}}, {{June}} 6\textendash 15, 1977}, Lecture
  {{Notes}} in {{Physics}}, pages 145--151. {Springer}, {Berlin, Heidelberg},
  1978.

\bibitem{crainicMeasuresDifferentiableStacks2020}
M.~Crainic and J.~N. Mestre.
\newblock Measures on differentiable stacks.
\newblock {\em J. Noncommut. Geom.}, 13(4):1235--1294, 2020.

\bibitem{evensTransverseMeasuresModular1999}
S.~Evens, J.-H. Lu, and A.~Weinstein.
\newblock Transverse measures, the modular class and a cohomology pairing for
  {{Lie}} algebroids.
\newblock {\em Q. J. Math.}, 50(200):417--436, 1999.

\bibitem{fregierHomotopyTheorySingular2019}
Y.~Fregier and R.~A. {Juarez-Ojeda}.
\newblock Homotopy theory of singular foliations, {{arXiv}}:1811.03078, 2019.

\bibitem{garmendiaHausdorffMoritaEquivalence2019}
A.~Garmendia and M.~Zambon.
\newblock Hausdorff {{Morita}} equivalence of singular foliations.
\newblock {\em Ann. Glob. Anal. Geom.}, 55(1):99--132, 2019.

\bibitem{hermannDifferentialGeometryFoliations1962}
R.~Hermann.
\newblock The differential geometry of foliations, {{II}}.
\newblock {\em J. Math. Mech.}, 11(2):303--315, 1962.

\bibitem{hurderGodbillonMeasureAmenable1986}
S.~Hurder.
\newblock The {{Godbillon}} measure of amenable foliations.
\newblock {\em J. Differ. Geom.}, 23(3):347--365, 1986.

\bibitem{jotzleanModulesRepresentationsHomotopy2023}
M.~Jotz~Lean, R.~A. Mehta, and T.~Papantonis.
\newblock Modules and representations up to homotopy of {{Lie}} n-algebroids.
\newblock {\em J. Homotopy Relat. Struct.}, 2023.

\bibitem{kajiuraHomotopyAlgebrasInspired2006}
H.~Kajiura and J.~Stasheff.
\newblock Homotopy {{Algebras Inspired}} by {{Classical Open-Closed String
  Field Theory}}.
\newblock {\em Commun. Math. Phys.}, 263(3):553--581, 2006.

\bibitem{kamberFoliatedBundlesCharacteristic1975}
F.~W. Kamber and P.~Tondeur.
\newblock {\em Foliated {{Bundles}} and {{Characteristic Classes}}}.
\newblock Number 493 in Lecture {{Notes}} in {{Mathematics}}.
  {Springer-Verlag}, {Berlin, Heidelberg}, 1975.

\bibitem{kosmann-schwarzbachPoissonManifoldsLie2008}
Y.~{Kosmann-Schwarzbach}.
\newblock Poisson {{Manifolds}}, {{Lie Algebroids}}, {{Modular Classes}}: A
  {{Survey}}.
\newblock {\em SIGMA}, 4:005, 2008.

\bibitem{kosmann-schwarzbachRelativeModularClasses2005}
Y.~{Kosmann-Schwarzbach} and A.~Weinstein.
\newblock Relative modular classes of {{Lie}} algebroids.
\newblock {\em C. R. Math.}, 341(8):509--514, 2005.

\bibitem{ladaStronglyHomotopyLie1995}
T.~Lada and M.~Markl.
\newblock Strongly homotopy lie algebras.
\newblock {\em Commun. Algebra}, 23(6):2147--2161, 1995.

\bibitem{ladaIntroductionSHLie1993}
T.~Lada and J.~Stasheff.
\newblock Introduction to {{SH Lie}} algebras for physicists.
\newblock {\em Int. J. Theor. Phys.}, 32(7):1087--1103, 1993.

\bibitem{laurent-gengouxUniversalLieInfinityAlgebroid2020}
C.~{Laurent-Gengoux}, S.~Lavau, and T.~Strobl.
\newblock The universal {{Lie}} infinity-algebroid of a singular foliation.
\newblock {\em Doc. Math.}, 25:1571--1652, 2020.

\bibitem{laurent-gengouxLieRinehartAlgebrasAcyclic2022}
C.~{Laurent-Gengoux} and R.~Louis.
\newblock Lie-{{Rinehart}} algebras {$\simeq$} acyclic {{Lie}}
  {$\infty$}-algebroids.
\newblock {\em J. Algebra}, 594:1--53, 2022.

\bibitem{lavauShortGuideIntegration2018}
S.~Lavau.
\newblock A short guide through integration theorems of generalized
  distributions.
\newblock {\em Differ. Geom. Appl.}, 61:42--58, 2018.

\bibitem{mehtaLieAlgebroidModules2014}
R.~A. Mehta.
\newblock Lie algebroid modules and representations up to homotopy.
\newblock {\em Indag. Math.}, 25(5):1122--1134, 2014.

\bibitem{mehtaAlgebraActions2012}
R.~A. Mehta and M.~Zambon.
\newblock {{L}}{$\infty$}-algebra actions.
\newblock {\em Differ. Geom. Appl.}, 30(6):576--587, 2012.

\bibitem{miyamotoSingularFoliationsDiffeology2023}
D.~Miyamoto.
\newblock Singular foliations through diffeology, 2023.

\bibitem{planteFoliationsMeasurePreserving1975}
J.-F. Plante.
\newblock Foliations {{With Measure Preserving Holonomy}}.
\newblock {\em Ann. Math.}, 102(2):327--361, 1975.

\bibitem{ruelleCurrentsFlowsDiffeomorphisms1975}
D.~Ruelle and D.~Sullivan.
\newblock Currents, {{Flows}} and {{Diffeomorphisms}}.
\newblock {\em Topology}, 14:319--327, 1975.

\bibitem{schatzBFVcomplexHigherHomotopy2008}
F.~Sch{\"a}tz.
\newblock {{BFV-complex}} and higher homotopy structures.
\newblock {\em Commun. Math. Phys.}, 286(2):399, 2008.

\bibitem{vaintrobLieAlgebroidsHomological1997}
A.~Vaintrob.
\newblock Lie algebroids and homological vector fields.
\newblock {\em Russ. Math. Surv.}, 52(2):428--429, 1997.

\bibitem{voronovHigherDerivedBrackets2005}
T.~Voronov.
\newblock Higher derived brackets and homotopy algebras.
\newblock {\em J. Pure Appl. Algebra}, 202(1):133--153, 2005.

\bibitem{voronovManifoldsHigherAnalogs2010}
T.~Voronov.
\newblock Q-manifolds and higher analogs of {{Lie}} algebroids.
\newblock {\em AIP Conference Proceedings}, 1307(1):191--202, 2010.

\bibitem{weinsteinModularAutomorphismGroup1997}
A.~Weinstein.
\newblock The modular automorphism group of a {{Poisson}} manifold.
\newblock {\em J. Geom. Phys.}, 23(3):379--394, 1997.

\bibitem{weinsteinVolumeDifferentiableStack2009}
A.~Weinstein.
\newblock The {{Volume}} of a {{Differentiable Stack}}.
\newblock {\em Lett. Math. Phys.}, 90(1):353, 2009.

\bibitem{yamagamiModularCohomologyClass1986}
S.~Yamagami.
\newblock Modular cohomology class of foliation and {{Takesaki}}'s duality.
\newblock In H.~Araki and E.~G. Effros, editors, {\em Geometric {{Methods}} in
  {{Operator Algebras}}}, number 123 in Pitman {{Res}}. {{Notes Math}}.
  {{Ser}}., pages 415--439. {Longman Scientific \& Technical}, {Harlow}, 1986.

\end{thebibliography}

\end{document}